\documentclass[11pt,a4paper,reqno]{amsart}
\usepackage[english]{babel}
\usepackage[T1]{fontenc}
\usepackage{verbatim}
\usepackage{palatino}
\usepackage{amsmath}
\usepackage{mathabx}
\usepackage{amssymb}
\usepackage{amsthm}
\usepackage{amsfonts}
\usepackage{graphicx}
\usepackage{esint}
\usepackage{color}
\usepackage{mathtools}

\usepackage[colorlinks = true, citecolor = black]{hyperref}
\author{Katrin F\"assler and Tuomas Orponen}
\title[Singular integrals on regular curves]{Singular integrals on regular curves \\ in the Heisenberg group}
\address{Department of Mathematics and Statistics\\
University of Jyv\"{a}skyl\"{a}, P.O. Box 35 (MaD)\\FI-40014 University of Jyv\"{a}skyl\"{a}\\
Finland
}
\address{Department of Mathematics and Statistics\\ University of Helsinki,
P.O. Box 68 (Pietari Kalmin katu 5)\\
FI-00014 University of Helsinki\\
Finland}
\email{katrin.s.fassler@jyu.fi}
\email{tuomas.orponen@helsinki.fi}
\date{\today}
\thanks{K.F. is supported by the Academy of Finland via the project \emph{Singular integrals, harmonic functions, and boundary regularity in Heisenberg groups}, grant No. 321696. T.O. is supported by the Academy of Finland via the projects \emph{Quantitative rectifiability in Euclidean and non-Euclidean spaces} and \emph{Incidences on Fractals}, grant Nos. 309365, 314172, 321896. This work was partially supported by the grant 346300 for IMPAN from the Simons Foundation and the matching 2015-2019 Polish MNiSW fund.}

\newcommand{\R}{\mathbb{R}}
\newcommand{\W}{\mathbb{W}}
\newcommand{\He}{\mathbb{H}}
\newcommand{\N}{\mathbb{N}}
\newcommand{\Q}{\mathbb{Q}}
\newcommand{\C}{\mathbb{C}}
\newcommand{\Z}{\mathbb{Z}}

\newcommand{\calT}{\mathcal{T}}
\newcommand{\calL}{\mathcal{L}}
\newcommand{\calD}{\mathcal{D}}
\newcommand{\calH}{\mathcal{H}}

\newcommand{\calB}{\mathcal{B}}
\newcommand{\calG}{\mathcal{G}}
\newcommand{\calF}{\mathcal{F}}
\newcommand{\calS}{\mathcal{S}}
\newcommand{\calQ}{\mathcal{Q}}

\newcommand{\spt}{\operatorname{spt}}

\newcommand{\diam}{\operatorname{diam}}

\newcommand{\dist}{\operatorname{dist}}

\newcommand{\bmo}{\mathrm{BMO}}

\newcommand{\V}{\mathbb{V}}

\newcommand{\q}{\mathfrak{q}}

\def\Barint_#1{\mathchoice
          {\mathop{\vrule width 6pt height 3 pt depth -2.5pt
                  \kern -8pt \intop}\nolimits_{#1}}%
          {\mathop{\vrule width 5pt height 3 pt depth -2.6pt
                  \kern -6pt \intop}\nolimits_{#1}}%
          {\mathop{\vrule width 5pt height 3 pt depth -2.6pt
                  \kern -6pt \intop}\nolimits_{#1}}%
          {\mathop{\vrule width 5pt height 3 pt depth -2.6pt
                  \kern -6pt \intop}\nolimits_{#1}}}

\numberwithin{equation}{section}

\theoremstyle{plain}
\newtheorem{thm}[equation]{Theorem}
\newtheorem*{"thm"}{"Theorem"}

\newtheorem{lemma}[equation]{Lemma}

\newtheorem{ex}[equation]{Example}
\newtheorem{cor}[equation]{Corollary}
\newtheorem{proposition}[equation]{Proposition}
\newtheorem{question}{Question}

\theoremstyle{definition}

\newtheorem{definition}[equation]{Definition}

\theoremstyle{remark}
\newtheorem{remark}[equation]{Remark}

\newcommand{\nref}[1]{(\hyperref[#1]{#1})}

\DeclareMathSymbol{\intop}  {\mathop}{mathx}{"B3}

\begin{document}

\begin{abstract} Let $\He$ be the first Heisenberg group, and let $k \in C^{\infty}(\He \, \setminus \, \{0\})$ be a kernel which is either odd or horizontally odd, and satisfies
\begin{displaymath} |\nabla_{\He}^{n}k(p)| \leq C_{n}\|p\|^{-1 - n}, \qquad p \in \He \, \setminus \, \{0\}, \, n \geq 0. \end{displaymath}
The simplest examples include certain Riesz-type kernels first considered by Chousionis and Mattila, and the horizontally odd kernel $k(p) = \nabla_{\He} \log \|p\|$. We prove that convolution with $k$, as above, yields an $L^{2}$-bounded operator on regular curves in $\He$. This extends a theorem of G. David to the Heisenberg group.

As a corollary of our main result, we infer that all $3$-dimensional horizontally odd kernels yield $L^{2}$ bounded operators on \emph{Lipschitz flags} in $\He$. This was known earlier for only one specific operator, the $3$-dimensional Riesz transform. Finally, our technique yields new results on certain non-negative kernels, introduced by Chousionis and Li.

 \end{abstract}

\maketitle

\tableofcontents

\section{Introduction}

This paper concerns the $L^{2}$ boundedness of certain singular integral operators (SIOs) on regular curves in the \emph{Heisenberg group} $(\He,d) = (\R^3,\cdot,d)$. For a brief introduction to the space $(\He,d)$, see Section \ref{s:Heis}. We recall that a closed set $E$ in a metric space $(X,d)$ is $s$-regular, for $s \geq 0$, if there exists a constant $C \geq 1$ such that
\begin{displaymath}
C^{-1} r^{s} \leq \mathcal{H}^{s}(E\cap B(x,r))\leq C r^{s}, \qquad x\in E,\, 0<r \leq \diam(E).
\end{displaymath}

\begin{definition}\label{d:1-regCurve} A closed set $\gamma$ in a metric space $(X,d)$ is a \emph{regular curve} if $\gamma$ is a $1$-regular set, and also the Lipschitz image of a closed subinterval of $\R$.  \end{definition}

The study of SIOs on regular curves in $\R^{n}$ has a long history. Calder\'on \cite{Calderon} in 1977 proved that the \emph{Cauchy transform} $\mathcal{C}f(z) = f \ast \tfrac{1}{z}$ defines an operator bounded on $L^{2}(\Gamma)$, whenever $\Gamma \subset \C$ is the graph of Lipschitz function with small Lipschitz constant. Coifman, McIntosh, and Meyer \cite{CMM} removed the "small constant" assumption in 1982. Coifman, David, and Meyer \cite{MR700980} then proved the same with the Cauchy kernel "$\tfrac{1}{z}$" replaced by any smooth $-1$-homogeneous odd function $k \colon \R^{n} \, \setminus \, \{0\} \to \C$. David \cite{MR744071} extended the results to all regular curves $\gamma \subset \R^{n}$,  see also \cite{MR956767}. The results in \cite{MR744071,MR956767} imply that if $\gamma \subset \R^{n}$ is a regular curve, $\mu := \calH^{1}|_{\gamma}$, and $k$ is as above, then the sublinear operator
\begin{equation}\label{maximalSIO} T^{\ast}_{k,\mu}f(x) := \sup_{\epsilon > 0} \left| \int_{\{y : |x - y| > \epsilon\}} k(x - y)f(y) \, d\mu(y) \right|, \qquad f \in C_{c}(\R^{n}), \end{equation}
called the \emph{maximal SIO induced by $(k,\mu)$}, extends to a bounded operator on $L^{p}(\mu)$, for any $1 < p < \infty$. In the sequel, we will abbreviate the $L^{p}(\mu)$ boundedness of $T^{\ast}_{k,\mu}$, $1 < p < \infty$, by writing that \emph{$k$ is a Calder\'on-Zygmund (CZ) kernel for $\mu$}.

\subsection{Singular integrals on regular curves in $\He$} What are the natural kernels in $\He$? In $\R^{n}$, the oddness assumption is prevalent, so one might also study odd kernels in $\He$. In fact, Chousionis and Mattila \cite{MR2755676} first considered the odd $-1$-homogeneous Riesz-type kernels
\begin{displaymath} k_{x}(x,y,t) = \frac{x}{\|(x,y,t)\|^{2}}, \quad  k_{y}(x,y,t) = \frac{y}{\|(x,y,t)\|^{2}}, \quad k_{t}(x,y,t) = \frac{t}{\|(x,y,t)\|^{3}}. \end{displaymath}
Here, and in the introduction, $\|(x,y,t)\| = ((x^{2} + y^{2})^{2} + 16t^{2})^{1/4}$ is the \emph{Kor\'anyi norm} of $(x,y,t) \in \He$. Chousonis and Mattila showed in \cite[Corollary 4.4]{MR2755676} that $K = (k_{x},k_{y},k_{t})$ is \textbf{not} a CZ kernel for $1$-dimensional self-similar measures on $\He$, unless they are supported on \emph{horizontal lines} (see Definition~\ref{d:horizontalLine}). In contrast, our main result, Theorem \ref{t:mainRegularCurve}, will yield the positive result that $K$ is a CZ kernel for $\calH^{1}$ restricted to any regular curve in $\He$.

In $\R^{n}$, the oddness hypothesis is not only a matter of
technical convenience. It stems from the existence of "useful" odd
kernels, obtained by differentiating (negative) powers of the
Euclidean norm $|\cdot|$. In particular, the $(n - 1)$-dimensional
Riesz kernel $\nabla |x|^{2 - n}$ is of key importance in the
theory of partial differential equations, see
\cite{MR501367,MR1418902,MR1323804,MR769382}, and the \emph{removability problem} for
Lipschitz harmonic functions, see \cite{dm,ntov,ntov2}. In the
Heisenberg group, a similar role is played by the  "$\He$-Riesz
kernel" $k(p) = \nabla_{\He} \|p\|^{-2n}$, see
\cite{CM,CFO2,2018arXiv181013122F}, where $\nabla_{\He} =
(X_{1},\ldots,X_{2n})$ is the horizontal gradient as defined in
Section \ref{s:Heis}. See \cite{2018arXiv181013122F} for
up-to-date results and open questions regarding the $\He$-Riesz
kernel.

In contrast to $\R^{n}$, the horizontal derivatives of (negative) powers of the Kor\'anyi norm do not yield odd kernels,
but \emph{horizontally odd} kernels:
\begin{equation}\label{horizontallyOdd} k(-x,-y,t)= -k(x,y,t), \qquad (x,y,t) \in \He \, \setminus \, \{0\}. \end{equation}
Condition \eqref{horizontallyOdd} is not weaker than oddness, but
simply incomparable: for example, it forces $k$ to vanish on the
$t$-axis.
 Theorem \ref{t:mainRegularCurve} will apply for instance to the $-1$-homogeneous horizontally odd kernel
\begin{displaymath} \nabla_{\He} \log \|(x,y,t)\| = \left(\frac{x(x^{2} + y^{2}) - 4ty}{\|(x,y,t)\|^{4}},\frac{y(x^{2} + y^{2}) + 4tx}{\|(x,y,t)\|^{4}} \right). \end{displaymath}
Another motivation to study horizontally odd kernels stems
from applications to \emph{Lipschitz flags} in $\He$, see Section
\ref{s:flagsIntro} for further discussion.

After this motivation, here are our standing kernel assumptions:
\begin{definition}[Good kernels]\label{def:goodKernel} A function $k \colon \He \, \setminus \, \{0\} \to \C$ is a \emph{good kernel} if
\begin{enumerate}
\item $k \in C^{\infty}(\He \, \setminus \{0\})$, and for every $n \geq 0$ there exists a constant $C_{n} > 0$ such that
\begin{displaymath} |\nabla_{\He}^{n}k(p)| \leq C_{n}\|p\|^{-n - 1}, \qquad p \in \He \, \setminus \, \{0\}. \end{displaymath}
\item $k$ is either odd, or horizontally odd in the sense \eqref{horizontallyOdd},
\end{enumerate}
\end{definition}
In (1), the notation $\nabla_{\He}^{n}$ refers to any concatenation of the $X$ and $Y$ vector fields (see Definition \ref{d:DefHorizGrad}) of length at most $n$. Here is, then, the main result of the paper:
\begin{thm}\label{t:mainRegularCurve} Good kernels are CZ kernels for regular curves in $\He$. \end{thm}
The property of a good kernel "$k$ being a CZ kernel for a regular curve $\gamma$" means the same as before: the maximal SIO induced by $(k,\calH^{1}|_{\gamma})$ defines an operator bounded on $L^{p}(\calH^{1}|_{\gamma})$, for $1 < p < \infty$. See Definition \ref{def:epsilonSIO} for a more formal treatment.

\begin{remark} Our good kernels are not assumed to be $-1$-homogeneous, so the theorem is superficially stronger than the original result of David \cite{MR744071} mentioned in the first section. However, the inhomogeneous variant is well-known for odd smooth kernels in $\R^{n}$. The proof is, for example, outlined in a sequence of exercises at the end of \cite[Part II]{MR1123480}. A different proof (assuming only $C^{2}$-regularity from $k$) is also contained in \cite{To}.
\end{remark}

In the next two subsections, we will explain some further results.

\subsection{Non-negative kernels} SIOs on regular curves in $\He$ were first studied by
Chousionis and Li in \cite{MR3678492}. The kernels $k \colon \He
\, \setminus \, \{0\} \to \C$ considered in \cite{MR3678492} are
not "good" in the sense of Definition \ref{def:goodKernel}.
Instead, they are \textbf{non-negative} $-1$-homogeneous kernels
of the form
\begin{displaymath} k_{\alpha}(x,y,t) = \frac{(\sqrt{|t|}/\|p\|)^{\alpha}}{\|p\|}, \qquad p = (x,y,t) \in \He \, \setminus \{0\}, \, \alpha \geq 1.\end{displaymath}
Chousionis and Li proved that $k_{\alpha}$ with $\alpha \geq 8$ is a CZ kernel for regular curves $\gamma \subset \He$, and with Zimmerman they found a generalisation of this result to arbitrary Carnot groups \cite{MR3883316}. Conversely, they also showed in \cite{MR3678492} that if $E \subset \He$ is $1$-regular, and $k_{2}$ is a CZ kernel for $E$, then $E$ is contained on a regular curve. It may sound astounding that non-negative kernels could ever be CZ kernels. A heuristic explanation comes from noting that $k_{\alpha}$ vanishes identically on the plane $\{(x,y,t) : t = 0\}$. Consequently, if $\ell \subset \He$ is a horizontal line, then the (maximal) SIO induced by $(k_{\alpha},\calH^{1}|_{\ell})$ is the zero operator. In contrast, our good kernels restricted to horizontal lines are odd, and the induced SIOs on horizontal lines behave like the Hilbert transform (at least when the kernels are $-1$-homogeneous).

Our technique also applies to the kernels $k_{\alpha}$ by Chousionis-Li:
\begin{thm}\label{t:mainCL} The kernels $k_{\alpha}$ are CZ kernels for regular curves in $\He$ for $\alpha \geq 4$. \end{thm}
Recall that Chousionis and Li \cite{MR3678492} proved this for $\alpha \geq 8$. It would be very interesting to know (as also Chousionis and Li point out) if the result persists for $\alpha \geq 2$; then we could infer that $k_{2}$ is a CZ kernel for a $1$-regular set $E \subset \He$ if and only if $E$ is contained on a regular curve. We close the section with another open question. While our technique applies to the kernels $k_{\alpha}$, our main result, Theorem \ref{t:mainRegularCurve} does not. So, we ask for a class of kernels which simultaneously contains odd and horizontally odd kernels, and the non-negative kernels of Chousionis-Li. Here is one suggestion (\emph{caveat emptor}!):
\begin{question} Let $k \colon \He \, \setminus \, \{0\} \to \C$ be a smooth $-1$-homogeneous function which is a CZ kernel for horizontal lines, with uniform constants. Is $k$ then a CZ kernel for regular curves? \end{question}
After the first version of this paper was posted on the \emph{arXiv}, Chousionis, Li, and Zimmerman \cite{2019arXiv191213279C} established the following partial result in all Carnot groups: whenever a $1$-dimensional standard kernel (see Definition \ref{def:StandardKernel}) is a CZ kernel for all horizontal lines, with uniform constants, then it is also a CZ kernel for regular $C^{1,\alpha}$-curves for $\alpha > 0$.

\subsection{SIOs on Lipschitz flags}\label{s:flagsIntro} The $L^{2}$-boundedness of $1$-codimensional SIOs on Lipschitz surfaces is a key component in the \emph{method of layer potentials}. This is a powerful technique for solving boundary value problems (BVPs) associated to elliptic and parabolic partial differential operators (PDOs) in domains with non-smooth boundaries, see \cite{MR501367,MR1418902,MR1323804,MR769382}. The method has not been equally successful in solving BVPs for \emph{sub-elliptic} PDOs, such as the Kohn-Laplacian $\bigtriangleup_{\He} = X^{2} + Y^{2}$. One key missing piece is the $L^{2}$-boundedness of $1$-codimensional SIOs on non-smooth surfaces in Heisenberg groups (for smooth surfaces, layer potentials were already employed by Jerison \cite{MR639800,MR633978} in the 80s). Here, we make progress on \emph{Lipschitz flags} in $\He$. In the upcoming work \cite{OV}, the result will be used to implement the method of layer potentials for $\bigtriangleup_{\He}$ in domains bounded by Lipschitz flags.

A \emph{Lipschitz flag} is a subset of $\He$ of the form $\mathcal{F} = \{(A(y),y,t) : y,t \in \R\}$, where $A \colon \R \to \R$ is Lipschitz. In Section \ref{s:SIOOnFLagSUrfaces}, we derive the following result as a corollary of the "$1$-dimensional" main theorems discussed above:
\begin{thm}\label{t:flagsInto} Let $K \in C^{\infty}(\He \, \setminus \, \{0\})$ be a horizontally odd kernel satisfying
\begin{equation*} |\nabla_{\He}^{n}K(p)| \leq C_{n}\|p\|^{-3 - n}, \qquad p \in \He, \, n \geq 0, \end{equation*}
for some constants $C_{n} > 0$. Then $K$ is a CZ kernel for Lipschitz flags in $\He$. \end{thm}
In other words, the maximal SIO induced by $(K,\mathcal{H}^{3}|_{\mathcal{F}})$ is bounded on $L^{p}(\mathcal{H}^{3}|_{\mathcal{F}})$, $1 < p < \infty$, whenever $\mathcal{F} \subset \He$ is a Lipschitz flag. A slightly more general version of the result above will be needed, and proven in Theorem \ref{t:flagTechnical}, for the purpose of the application in \cite{OV}. For the kernel $K(p) = \nabla_{\He} \|p\|^{-2}$, Theorem \ref{t:flagsInto} is a corollary of the main result in \cite{2018arXiv181013122F}.

\subsection{Overview of proofs in $\R^{n}$}\label{s:overview} Before giving an outline of the proof of Theorem \ref{t:mainRegularCurve} in Section \ref{s:outline}, we discuss proof strategies in $\R^{n}$, concerning the action of odd kernels on regular curves. David's approach in \cite{MR744071} was to reduce the problem on regular curves to the one on Lipschitz graphs: the main ideas were that regular curves have \emph{big pieces of Lipschitz graphs} (BPLG), and that CZ kernels for Lipschitz graphs are also CZ kernels for $1$-regular sets with BPLG. A second "reduction" proof of this type is due to Semmes \cite{MR1087183} from 1990. He introduced the notion of sets which admit \emph{corona decompositions by Lipschitz graphs} (CDLG), and showed that CZ kernels for Lipschitz graphs are CZ kernels for $1$-regular sets admitting CDLG.

An alternative strategy was found by Jones \cite{MR1013815, MR1069238}. He introduced the notion of \emph{$\beta$-numbers}: given a set $K \subset \R^{n}$, and a ball $B(x,r)$ centred on $K$, the $\beta$-number $\beta_{K}(B(x,r))$ measures the deviation of $K \cap B(x,r)$ from the best-approximating line. Jones proved in \cite{MR1069238} that the $\beta$-numbers on regular curves in $\gamma \subset \C$ satisfy the following square function estimate:
\begin{equation}\label{betaCarleson} \int_{0}^{R} \int_{B(x_{0},R)} \beta_{\gamma}(B(x,r))^{2} \, d\calH^{1}|_{\gamma}(x) \, \frac{dr}{r} \lesssim R, \qquad B(x_{0},R) \subset \C. \end{equation}
The case of Lipschitz graphs was already contained in \cite{MR1013815}, where Jones deduced the $L^{2}$-boundedness of $\mathcal{C}$ on Lipschitz graphs from the geometric condition \eqref{betaCarleson}. The square function estimate \eqref{betaCarleson} is also valid for regular curves in $\R^{n}$, as shown by Okikiolu \cite{MR1182488}.

More recently, Tolsa \cite{To} introduced the notion of \emph{$\alpha$-numbers}. These are, roughly speaking, measure-theoretic versions of Jones' $\beta$-numbers. Tolsa showed that odd $m$-dimensional $C^{2}$-smooth kernels in $\R^{n}$ are CZ kernels for any $m$-regular measure $\mu$ on $\R^{n}$ whose $\alpha$-numbers satisfy a square function estimate analogous to \eqref{betaCarleson}. This improves on the result of David \cite{MR956767}, since only $C^{2}$-regularity of the kernel is required (and $-m$-homogeneity is not assumed). Moreover, as in Jones' argument, the proof deduces the $L^{2}$-boundedness of SIOs directly from bounds on a square function involving the $\alpha$-numbers, without passing via Lipschitz graphs.

Investigating the connections between Lipschitz graphs, sets with BPLG, or admitting CDLG, square function estimates involving $\alpha$'s, $\beta$'s, or other geometric quantities, and the $L^{2}$-boundedness of SIOs, is known as the theory of \emph{uniform rectifiability}. For more information, see \cite{DS1,MR1251061,MR3154530}.

\subsection{The proof of Theorem \ref{t:mainRegularCurve}: an outline}\label{s:outline} Above, we mentioned two approaches for studying SIOs on regular curves in $\R^{n}$: either reduce matters to the special case of Lipschitz graphs via "big piece" or "corona" methods, or take a more direct route via geometric square functions ($\alpha$-numbers or $\beta$-numbers). In this paper, we take the former approach(es), as the latter appears to be difficult to execute for two separate reasons:
\begin{itemize}
\item The oddness of kernels in $\R^{n}$ is critical in \emph{quasiorthogonality} arguments, see \cite{To}, and horizontal oddness seems to be a poor substitute in this regard.
\item Analogues of Jones' $\beta$-numbers have been extensively studied in $\He$, see \cite{MR2371434,MR2789375,MR3456155,MR3512421,li2019stratified}. A surprising example of Juillet \cite{MR2789375} shows that the $L^{2}$-integral of the $\beta$-numbers appearing in \eqref{betaCarleson} need \textbf{not} be bounded by $\calH^{1}(\gamma)$, for rectifiable curves $\gamma \subset B(x_{0},R)$. Instead, Li and Schul \cite{MR3456155} proved a version of \eqref{betaCarleson} where the exponent "$2$" is replaced by "$4$". We do not know how to use this -- weaker -- information to prove Theorem \ref{t:mainRegularCurve} in $\He$, even for odd kernels.
\end{itemize}
We then discuss the former approach. Heisenberg analogues of Lipschitz graphs are known as \emph{intrinsic Lipschitz graphs} (iLGs), and they were introduced by Franchi, Serapioni, and Serra Cassano \cite{FSS} in 2006. Their rectifiability properties, both qualitative and quantitative, have been investigated vigorously in recent years, see \cite{MR3992573,MR3168633,2018arXiv180304819F,FSSC2,MR2789472,NY,2019arXiv190610215O,2019arXiv190406904R}. However, many of these papers have focused on $1$-co-dimensional iLGs, whereas the objects relevant here are the $1$-dimensional iLGs over \emph{horizontal subgroups of $\He$}, see Section \ref{s:IntrLip}. The first objective \emph{en route} to Theorem \ref{t:mainRegularCurve} is to establish the result in the special case of $1$-dimensional iLGs in $\He$:
\begin{thm}\label{t:mainIntrLipGraph} Weakly good kernels are CZ kernels for iLGs over horizontal subgroups in $\He$. \end{thm}
A function $k \colon \He \, \setminus \, \{z=(x,y)=0\} \to \C$
 is a \emph{weakly good kernel} if
 $k \in C^{\infty}(\He \, \setminus \, \{z=0\})$, $k$ is either odd or horizontally odd, and for every $n \geq 0$ there exists a constant $C_{n} > 0$ such that
\begin{equation}\label{eq:weakGood} |\nabla_{\He}^{n}k(p)| \leq C_{n}|z|^{-n - 1}, \qquad p = (z,t) \in \He \, \setminus \, \{z=0\}. \end{equation}
In other words, weakly good kernels do not necessarily decay in
the $t$-variable; as a consequence, they may not be "standard
kernels" in $\He$ (see Definition \ref{def:StandardKernel}).
However, they are standard kernels when restricted to any iLG over
a horizontal subgroup in $\He$. Weakly good kernels arise in a
natural way from the kernels appearing in Theorem
\ref{t:flagsInto}, see Lemma \ref{l:calK}, and indeed, a slightly
stronger version of Theorem \ref{t:mainIntrLipGraph} can be used
to prove Theorem \ref{t:flagsInto} about Lipschitz flags in $\He$.

 Theorem
\ref{t:mainIntrLipGraph} is the main news of the paper. Once it
has been established, we still need to complete David's approach
in \cite{MR744071}, and prove the following statements:

\begin{thm}\label{t:RegCurveBPiLG} Regular curves in $\He$ have big pieces of intrinsic Lipschitz graphs (BPiLG) over horizontal subgroups. \end{thm}
\begin{"thm"} Let $(X,d)$ be a proper metric space, let $\calG$ be a family of $m$-regular sets in $(X,d)$, and let $K$ be an $m$-dimensional standard kernel on $X$ which is a CZ kernel for all $G \in \calG$, uniformly. Then $K$ is a CZ kernel for any $m$-regular set $B \subset X$ which has "big pieces" of sets in $\mathcal{G}$.  \end{"thm"}
For a more precise statement, see Theorem \ref{t:AbstractBigPiece}. The proof is a straightforward adaptation of \cite[Proposition 3.2]{MR1123480} to proper metric spaces, and we claim very little originality: the main point is to check that the Besicovitch covering theorem is not used in an essential way. Regarding Theorem \ref{t:RegCurveBPiLG}, we follow an approach of David and Semmes \cite{MR1132876}, by showing, first, that regular curves have \emph{big horizontal projections} (BHP), and satisfy the \emph{weak geometric lemma} for Jones' $\beta$-numbers. Then, a combination of these properties yields BPiLG. These arguments are quite well-known, and have even been adapted to $1$-co-dimensional iLGs in $\He^{n}$, see \cite{MR3992573,2018arXiv180304819F}. Only verifying the BHP property for regular curves produces a minor "new" problem. The details are contained in Section \ref{s:BPiLG}.

So, the heart of the matter is Theorem \ref{t:mainIntrLipGraph}, whose proof indeed takes up most of the paper. The problem quickly reduces to a question concerning certain $1$-dimensional SIOs on $\R$. More precisely, after repeating a decomposition due to Coifman, David, and Meyer \cite{MR700980}, one is led to consider (variants of) the standard kernel
\begin{equation}\label{KB} K_{B}(x,y) = \kappa(x - y)\exp\left(2\pi i\left[\tfrac{B_{2}(x) - B_{2}(y) - \tfrac{1}{2}[B_{1}(x) + B_{1}(y)](x - y)}{(x - y)^{2}} \right] \right), \end{equation}
where $\kappa \in C^{\infty}(\R \, \setminus \, \{0\})$ is an odd $-1$-dimensional kernel, and  $B = (B_{1},B_{2}) \colon \R \to \R^{2}$ is a \emph{tame map}. This simply means that $B_{1}$ is Lipschitz, and $\dot{B}_{2} = B_{1}$. Tame maps are thoroughly investigated in Section \ref{s:tame}. The kernel $K_{B}$ is not antisymmetric, but we nevertheless manage to prove in Theorem \ref{TABCZO} that $K_{B}$ is a CZ kernel on $\R$. In doing so, we adapt arguments of Christ \cite{MR1104656} and Hofmann \cite{MR1484857}. Unfortunately, this is not quantitative enough: to apply the kernels $K_{B}$ in the context of Theorem \ref{t:mainIntrLipGraph}, we need to know that the CZ constant of $K_{B}$, denoted $\|K_{B}\|_{\mathrm{C.Z.}}$, depends polynomially on the "tameness constant" of $B$. A similar problem for Lipschitz functions (and graphs) already appears in David's work \cite{MR744071,MR1123480}, but the solution is easier there: it is based on the "big piece theorem" stated below Theorem \ref{t:RegCurveBPiLG}, plus the simple -- and ingenious -- observation that "$L$-Lipschitz graphs have big pieces of $\tfrac{9}{10}L$-Lipschitz graphs", see \cite[p. 66]{MR1123480}. We were not able to prove an analogue of this property for tame maps, see Question \ref{q:BPTame}.

Instead, we found a weaker substitute: tame maps admit "corona decompositions" by tame maps with a smaller constant. More precise statements can be found in Section \ref{s:CoronaTame}. We mentioned in Section \ref{s:overview} that Semmes \cite{MR1087183} used corona decompositions (by Lipschitz graphs) to reduce SIO problems on regular curves to SIO problems on Lipschitz graphs. Applying his mechanism, and the tame-corona decomposition mentioned above, we can finally infer the polynomial dependence of $\|K_{B}\|_{\mathrm{C.Z.}}$ on the "tameness" of $B$. We refer to Section \ref{s:exp2} for details.

We have now summarised the proof of Theorem
\ref{t:mainRegularCurve}, and explained most of the structure of
the paper. Let us add that in Section \ref{s:prelimSIO}, we merely
collect standard preliminaries on Calder\'on-Zygmund theory. In
Section \ref{s:ILGandTame}, we introduce tame maps, the Heisenberg
group, and intrinsic Lipschitz graphs, and prove the corona
decomposition for tame maps. In Section \ref{s:Exp1}, we  reduce
the proof of Theorem \ref{t:mainIntrLipGraph} to the study of the
kernel $K_{B}$ -- or, as it really turns out, $K_{A,B}$ -- and
establish "qualitatively" that $K_{A,B}$ is a CZ kernel on $\R$.
The quantitative version is the main content of Section
\ref{s:exp2}, and this section concludes the proof of Theorem
\ref{t:mainIntrLipGraph}. In Section \ref{s:RegCurveBPiLG}, we
prove the "BPiLG" Theorem \ref{t:RegCurveBPiLG} and use it to
deduce Theorem \ref{t:mainRegularCurve} from Theorem
\ref{t:mainIntrLipGraph}. In this final "from graphs to curves"
upgrade, it is very useful to know that good kernels are standard
kernels in $\He$. This explains why weak goodness works for
Theorem \ref{t:mainIntrLipGraph}, but not for (the proof of)
Theorem \ref{t:mainRegularCurve}. In Section
\ref{s:SIOOnFLagSUrfaces} we use a version of Theorem
\ref{t:mainIntrLipGraph} to deduce Theorem \ref{t:flagsInto} about Lipschitz flags.

The proof of Theorem \ref{t:mainCL} (concerning the non-negative kernels $k_{\alpha}$) is easier than the proof of Theorem \ref{t:mainRegularCurve}. The case of intrinsic Lipschitz graphs is contained in Section \ref{s:CLProof}. The case of general regular curves is, again, reduced to this case with the BPiLG machinery, see Section \ref{s:conclusion} for the final details.

\subsection*{Acknowledgements.} This research was mostly conducted during the Simons Semester ``Geometry and analysis in function and mapping theory on Euclidean and metric measure spaces'' at IM PAN, Warsaw. We would like thank all the organisers, in particular Tomasz Adamowicz, and the staff at IM PAN, for their support and hospitality during our stay in Warsaw.

\section{Preliminaries on singular integral operators}\label{s:prelimSIO}

\subsection{Standard kernels}

We define standard kernels and Calder\'on-Zygmund operators, and recall some of their standard properties. 

\begin{definition}\label{def:StandardKernel}
Let $(X,d)$ be a metric space, write $\bigtriangleup := \{(x,x):\,x\in X\}$, and let $k>0$. A \emph{$k$-dimensional standard kernel} ($k$-SK) on $X$ is a Borel function $$K:X \times X \,\setminus \, \bigtriangleup \to \mathbb{C}$$  for which there exist constants $C > 0$ and $\alpha \in (0,1]$ such that the following holds:
\begin{enumerate}
\item  $|K(x,y)|\leq \frac{C}{d(x,y)^k}$,  for all $(x,y)\in X \times X \, \setminus \, \bigtriangleup$,
\item $\max\left\{|K(x,y)-K(x',y)|,|K(y,x)-K(y,x')|\right\}\leq C \frac{d(x,x')^{\alpha}}{d(x,y)^{k+\alpha}}$,
\end{enumerate}
whenever $x,x',y\in X$ and $d(x,x')\leq d(x,y)/2$. The smallest constant "$C$" above will be denoted by $\|K\|_{\alpha,strong}$.

A \emph{standard kernel} (SK), without reference to the dimension, will mean a $1$-SK.
\end{definition}

For the purpose of this paper, $X$ will often be the real line
$\mathbb{R}$ or the Heisenberg group $\mathbb{H}$ (Definition
\ref{d:Heis}). An important class of SKs, for this paper, are
those induced by good kernels $k \colon \He \, \setminus \, \{0\}
\to \C$, recall Definition \ref{def:goodKernel}. Setting $K(p,q)
:= k(q^{-1} \cdot p)$, one obtains an SK on $\He$ satisfying
Definition \ref{def:StandardKernel}(1)-(2) with $\alpha =
\tfrac{1}{2}$:
\begin{proposition}\label{p:HolBoundGoodHeis}
If $k \colon \He \, \setminus \, \{0\} \to \C$ is a good kernel, then $K:\He\times \He \, \setminus \, \bigtriangleup \to [0,+\infty)$, defined by $K(p,q):=k(q^{-1}\cdot p)$, is an SK on $(\He,d)$ with $\alpha = 1/2$.
\end{proposition}

\begin{proof}
The bound (1) is immediate from the good kernel assumption. The H\"older continuity (2) can be proven by arguments similar to \cite[Proposition 3.11]{CM} and \cite[Lemma 2.1]{CFO2}. The exponent $\alpha = \tfrac{1}{2}$ arises when verifying the H\"older continuity of $q \mapsto K(q^{-1} \cdot p)$.
\end{proof}

A weakly good kernel $k \in C^{\infty}(\He \, \setminus \,
\{z=0\})$, satisfying \eqref{eq:weakGood}, need not induce an SK
on $\He$ by the formula $K(p,q) = k(q^{-1} \cdot p)$. However, it
will turn out that if $\Gamma \subset \He$ is an \emph{intrinsic
Lipschitz graph over a horizontal subgroup}, then $K$ is an SK on
$(\Gamma,d)$. We record some preliminary details here, but the
matter will only be concluded in Section \ref{s:Exp1}.
\begin{ex}\label{ex1} Let $A \colon \R \to \R$ be an $M$-Lipschitz function, and let $B = (B_{1},B_{2}) \colon \R \to \R^{2}$ be an $N$-tame function (here we just need to know that $B_{1}$ is $N$-Lipschitz, and $\dot{B}_{2} = B_{1}$; see Section \ref{s:tame}), where $M,N \geq 1$. Let $k \colon \R \times \R \, \setminus \, \bigtriangleup \to \C$ be an SK, and let $\mathfrak{q} \colon \R \to \R$ be one of the functions
\begin{displaymath} \mathfrak{q}(s) := s^{2} \quad \text{or} \quad \mathfrak{q}(s) := s|s|. \end{displaymath}
Then, the kernel $K_{k,A,B}(x,y) :=$
\begin{displaymath}k(x,y)e_{A,B}(x,y) := k(x,y)\exp\left(2\pi i\left[ \tfrac{A(x) - A(y)}{x - y} + \tfrac{B_{2}(x) - B_{2}(y) - \tfrac{1}{2}[B_{1}(x) + B_{1}(y)](x - y)}{\mathfrak{q}(x - y)} \right] \right) \end{displaymath}
is an SK, with $\|K_{k,A,B}\|_{\alpha,strong} \lesssim \|k\|_{\alpha,strong}\max\{M,N\}$. To see this, fix $x,x',y \in \R$ with $|x - x'| \leq |x - y|/2$, and write
\begin{displaymath} |K_{k,A,B}(x,y) - K_{k,A,B}(x',y)| \leq |k(x,y) - k(x',y)| + |k(x',y)||e_{A,B}(x,y) - e_{A,B}(x',y)|, \end{displaymath}
and use the SK estimates for $k$. The problem then reduces to estimating $|e_{A,B}(x,y) - e_{A,B}(x',y)|$, which further reduces (using that $t \mapsto e^{2\pi i t}$ is $2\pi$-Lipschitz) at finding upper bounds for
\begin{displaymath} a(x,x',y) := \left|\frac{A(x') - A(y)}{x' - y} - \frac{A(x) - A(y)}{x - y} \right| \end{displaymath}
and
\begin{displaymath}b(x,x',y) := \left| \tfrac{B_{2}(x') - B_{2}(y) - \frac{1}{2}[B_{1}(x') + B_{1}(y)](x' - y)}{\q(x' - y)} - \tfrac{B_{2}(x) - B_{2}(y) - \frac{1}{2}[B_{1}(x) + B_{1}(y)](x - y)}{\q(x - y)} \right|. \end{displaymath}
We leave it to the reader to check that $a(x,x',y) \lesssim M|x' - x|/|x - y|$. To see that also $|b(x,x',y)| \lesssim N|x' - x|/|x - y|$, we first infer from the tameness of $B$ that $B_{2} \in C^{1}(\R)$, and $\dot{B}_{2} = B_{1}$, see Remark \ref{r:tame}. Therefore, for $x\neq y$,
\begin{equation}\label{form139} \frac{B_{2}(x) - B_{2}(y) - \tfrac{1}{2}[B_{1}(x) + B_{1}(y)](x - y)}{\q(x - y)} = \int_{x}^{y} \frac{B_{1}(x) + B_{1}(y) - 2B_{1}(s)}{2\q(x - y)} \, ds. \end{equation}
The tameness of $B$ also implies that $B_{1}$ is $N$-Lipschitz, so a little computation shows that the $x$ and $y$ derivatives of the right hand side are $\lesssim N/|x - y|$ almost everywhere. Now it follows from the fundamental theorem of calculus that $b(x,x',y) \lesssim N|x - x'|/|x - y|$, as claimed.

The same argument also works for bounding $|K_{k,A,B}(y,x) - K_{k,A,B,}(y,x')| \lesssim |x - x'|/|x - y|$.
\end{ex}

\subsection{Generalised standard kernels and CZOs}\label{sss:SIOR} In Section \ref{s:exp2}, we will encounter kernels which are not quite SKs in the sense above, but satisfy the following relaxed conditions:

\begin{definition}\label{def:GSK} Let $(X,d)$ be a proper metric space. A Borel function $K \colon X \times X \, \setminus \, \bigtriangleup \to \C$ is a $k$-dimensional \emph{generalised standard kernel} ($k$-GSK) if the "size" condition in Definition \ref{def:StandardKernel}(1) holds with constant $C \geq 1$, and moreover $K$ satisfies the following two inequalities for all Radon measures $\mu$ on $X$, for all $f \in L^{1}_{\mathrm{loc}}(\mu)$, and for all closed balls $B \subset X$:
\begin{equation}\label{hormanderC} \int_{X \, \setminus \, 2B} |K(x,y) - K(x_{0},y)||f(y)| \, d\mu(y) \leq C\mathcal{M}_{\mu,k}f(x_{0}), \qquad x,x_{0} \in B, \end{equation}
and
\begin{equation}\label{hormanderC2} \int_{X \, \setminus \, 2B} |K(y,x) - K(y,x_{0})||f(y)| \, d\mu(y) \leq C\mathcal{M}_{\mu,k}f(x_{0}), \qquad x,x_{0} \in B. \end{equation}
Here $\mathcal{M}_{\mu,k}$ is the "radial" maximal function of order $k$:
\begin{displaymath} \mathcal{M}_{\mu,k}f(x) := \sup_{r > 0} \frac{1}{r^{k}} \int_{B(x,r)} |f(y)| \, d\mu(y), \qquad x \in X. \end{displaymath}
The best constant "$C$" here will be denoted $\|K\|$.
\end{definition}
On first sight, it may appear odd that the constant "$C$" needs to be independent of the choice of the Radon measure $\mu$ on $X$. However,
Proposition \ref{p:FromHolKernelToGenStd} below shows that
any $k$-SK $K\colon X\times X\,\setminus\, \bigtriangleup \to \mathbb{C}$ is a $k$-GSK, with
\begin{displaymath} \|K\| \lesssim_{\alpha} \|K\|_{\alpha,strong}. \end{displaymath}

\begin{proposition}\label{p:FromHolKernelToGenStd}
 Let $(X,d)$ be a proper metric space, let $k > 0$, and let $K \colon X \times X \, \setminus \, \bigtriangleup \to \C$ be a $k$-SK. Then \eqref{hormanderC}-\eqref{hormanderC2} hold with a constant $C \lesssim_{\alpha,k} \|K\|_{\alpha,strong}$. \end{proposition}

\begin{proof} By symmetry, we only need to verify \eqref{hormanderC}.  Fix $B,\mu,f$, and $x,x_{0} \in B$ as in Definition \ref{def:GSK}. Then,
\begin{displaymath}
\int_{X\setminus 2B} |K(x,y)-K(x_0,y)||f(y)|\,d\mu(y)\lesssim \|K\|_{\alpha,strong} \int_{X\setminus B(x_0,r)} \frac{d(x,x_0)^{\alpha}}{d(x_0,y)^{k+\alpha}}|f(y)|\,d\mu(y).
\end{displaymath}
We used $B(x_{0},r)\subset 2B$ and the H\"older estimate for $K$. The latter is \emph{a priori} only valid for $y\in X\setminus 2B$ with $d(x,x_0)\leq d(x_0,y)/2$, but if $d(x_0,y)/2<d(x,x_0)$, then $d(x,x_0)\sim d(x,y)\sim d(x_0,y)\sim r$, and we can apply the size bounds $|K(x,y)| \leq \|K\|_{\alpha,strong}d(x,y)^{-k}$ and $|K(x_0,y)| \leq \|K\|_{\alpha,strong}d(x_0,y)^{-k}$. Decomposing $X \, \setminus \, B(x_0,r)$ into dyadic annuli, we further estimate
\begin{displaymath}
  \int_{X\setminus B(x_0,r)} \frac{d(x,x_0)^{\alpha}}{d(x_0,y)^{k+\alpha}}|f(y)|\,d\mu(y) \lesssim 2^k \sum_{j=0}^{\infty} \frac{1}{2^{j\alpha}} \frac{1}{(2^{(j+1)}r)^k} \int_{B(x_0,2^{j+1}r)} |f(y)| \, d\mu(y),
\end{displaymath}
from where \eqref{hormanderC} follows.
\end{proof}

 The main point about GSKs vs. SKs is that GSKs are stable under "sharp" truncations:
\begin{lemma}\label{hormander} Let $K \colon X \times X \, \setminus \, \bigtriangleup \to \C$ be a $k$-GSK, and let $D \colon X \times X \to [0,\infty)$ be a $\tfrac{1}{2}$-Lipschitz function in the metric $d_{X\times X} \left((x,y),(x',y')\right) = \max\left\{ d(x,x'),d(y,y') \right\}$. Then, the kernel $K^{D}$, defined by
 $$K^{D}(x,y) := K(x,y)\mathbf{1}_{\{d(x,y) \geq D(x,y)\}}(x,y),$$
is a $k$-GSK with $\|K^{D}\| \lesssim \|K\|$.

\end{lemma}

\begin{proof} By symmetry, it suffices to verify \eqref{hormanderC}. Fix $B \subset X$, $x,x_{0} \in B$, and a Radon measure $\mu$ on $X$. We claim that there are two, roughly dyadic, annuli $A_{1},A_{2}$ centred at $x_{0}$ such that either
\begin{equation}\label{form52} K^{D}(x,y) = K(x,y) \quad \text{and} \quad K^{D}(x_{0},y) = K(x_{0},y),  \end{equation}
or
\begin{equation}\label{eq:KernelZero}
K^{D}(x,y) = K^{D}(x_{0},y) =0
\end{equation}
for all $y \in (2B)^{c}$ with $y \notin [A_{1} \cup A_{2}]$. The lemma follows from this, and the computation
\begin{align*} \int_{(2B)^{c}} & |K^{D}(x,y) - K^{D}(x_{0},y)||f(y)| \, d\mu(y)\\
& \lesssim \int_{(2B)^{c}} |K(x,y) - K(x_{0},y)||f(y)| \, d\mu(y)\\
& \qquad + \|K\| \int_{(2B)^{c} \cap [A_{1} \cup A_{2}]} \frac{|f(y)|}{d(x_{0},y)^{k}} \, d\mu(y) \lesssim \|K\|\mathcal{M}_{\mu,k}f(x_{0}). \end{align*}
The points $y \in (2B)^{c}$ such that both \eqref{form52} and \eqref{eq:KernelZero} fail are contained in the union of
\begin{displaymath} B_1 :=\{y\in (2B)^c:\ D(x_{0},y) \leq d(x_{0},y) \quad \text{and} \quad d(x,y) < D(x,y)\} \end{displaymath}
and
\begin{displaymath} B_2:= \{y\in (2B)^c:\  D(x,y) \leq d(x,y) \quad \text{and} \quad d(x_{0},y) < D(x_{0},y)\}. \end{displaymath}
We will next show that
\begin{equation}\label{eq:AnnIncl}
B_1 \subset \{y\in (2B)^{c} :\; r_{1} \leq d(x_{0},y) \leq 100 r_{1}\}=: A_1
\end{equation}
with $r_1:= \inf\{d(x_{0},y):\; y\in B_1\}$.
To this end, fix $\varepsilon\in (0,1)$ and pick $y_{1}\in  B_1 \subset (2B)^{c}$ such that
\begin{displaymath}
r := d(x_{0},y_1) \in [r_1,\left(1+\varepsilon\right) r_1].
\end{displaymath}
Consider now any $y \in (2B)^{c}$ with
\begin{displaymath} d(x_{0},y) > 100r, \end{displaymath}
and note that $d(x,y) \geq d(x_{0},y) - d(x_{0},x) \geq 100r - 2d(x_{0},y_{1}) = 98r$, because
$d(x_{0},x) \leq 2 d(x_{0},y_{1})$. We claim that then $d(x,y) \geq D(x,y)$, so that $y\notin B_1$.
Indeed, using that $D$ is $\tfrac{1}{2}$-Lipschitz, and $d(x_{0},x) \leq 2 d(x_{0},y_{1})$, we have
\begin{align*} D(x,y) \leq D(x_{0},y_{1}) + \frac{d(x_{0},x)}{2} + \frac{d(y_{1},y)}{2} & \stackrel{y_1 \in B_1}{\leq} r + \frac{2r}{2} + \frac{r + d(x_{0},y)}{2}\\
& \leq \frac{d(x,y)}{98} + \frac{d(x,y)}{98} + \frac{d(x,y)}{196} + \frac{d(x_{0},y)}{2}\\
& \leq \frac{5d(x,y)}{196} + \frac{d(x,y)+d(x_{0},x)}{2}\\
&\leq \frac{5d(x,y)}{196} + \frac{d(x,y)}{2}+r\\
&\leq \left(\frac{105}{196}\right)d(x,y) < d(x,y).
 \end{align*}
 We deduce that the points $y\in B_1$ must satisfy $d(x_{0},y)\leq 100 r= 100 (1+\varepsilon) r_1$, and letting $\varepsilon \to 0$, we have established \eqref{eq:AnnIncl}. A symmetric argument yields that
\begin{equation}\label{eq:AnnIncl3}
B_2 \subseteq \{y\in (2B)^{c} :\; r_2 \leq d(x,y) \leq 100 r_2\}=: A_2'
\end{equation}
with $r_2:= \inf\{d(x,y):\; y\in B_2\}$. Since $r_{2} \geq \dist(x,(2B)^{c}) \geq d(x_{0},x)/2$, it is easy to see that $A_{2}' \subset A_{2}$, where $A_{2}$ is a slightly fatter annulus around $x_{0}$ with radius comparable to $r_{2}$. This completes the proof. \end{proof}

\begin{definition}[Induced operators and Calder\'on-Zygmund operators] Let $(X,d)$ be a proper metric space, let $k > 0$, and let $K \colon X \times X \, \setminus \, \bigtriangleup \to \C$ be a \textbf{bounded} $k$-GSK. Let $\mu$ be a Borel regular measure on $X$ satisfying
\begin{equation}\label{kLinear} \mu(B(x,r)) \leq Cr^{k}, \qquad x \in X, \; r > 0, \end{equation}
for some constant $C \geq 1$. We associate to $K$ and $\mu$ the following operator $T_{\mu}$:
\begin{displaymath} T_{\mu}f(x) := \int K(x,y)f(y) \, d\mu(y), \qquad f \in \bigcup_{1 < p < \infty} L^{p}(\mu), \; x \in X. \end{displaymath}
It is easy to see, using H\"older's inequality, \eqref{kLinear}, and the "size" bound in Definition \ref{def:StandardKernel}(1), that if $1 < p < \infty$ and $f \in L^{p}(\mu)$, then the integral defining $T_{\mu}f(x)$ is absolutely convergent. We say that \emph{$T_{\mu}$ is the operator induced by $(K,\mu)$.}

A \emph{Calder\'on-Zygmund operator} (CZO) is an operator $T_{\mu}$ induced by $(K,\mu)$, as above, which also happens to be bounded on $L^{2}(\mu)$. For a CZO $T_{\mu}$, we write
\begin{displaymath} \|T_{\mu}\|_{\mathrm{C.Z.}} := \|T_{\mu}\|_{L^{2}(\mu) \to L^{2}(\mu)} + \|K\|. \end{displaymath}
\end{definition}

\begin{definition}[$\epsilon$-SIOs and CZ kernels]\label{def:epsilonSIO} Let $K \colon X \times X \, \setminus \, \bigtriangleup \to \C$ be a $k$-GSK, not necessarily bounded, and let $\mu$ be a Borel measure on $X$ satisfying \eqref{kLinear}. For $\epsilon > 0$, we define $T_{\mu,\epsilon}$ to be the operator induced by $(K_{\epsilon},\mu)$, where
\begin{displaymath} K_{\epsilon}(x,y) := K(x,y)\mathbf{1}_{\{d(x,y) > \epsilon\}}(x,y), \qquad (x,y) \in X \times X \, \setminus \, \bigtriangleup. \end{displaymath}
The operator $T_{\mu,\epsilon}$ is called the \emph{$\epsilon$-SIO induced by $(K,\mu)$}. We also define the \emph{maximal SIO}
\begin{displaymath} T_{\mu}^{\ast}f(x) := \sup_{\epsilon > 0} |T_{\mu,\epsilon}f(x)|, \qquad f \in \bigcup_{1 < p < \infty} L^{p}(\mu), \; x \in X. \end{displaymath}
If the $\epsilon$-SIOs are uniformly bounded on $L^{2}(\mu)$,
\begin{equation}\label{induction} \sup_{\epsilon > 0} \|T_{\mu,\epsilon}\|_{L^{2}(\mu) \to L^{2}(\mu)} < \infty, \end{equation}
we say that \emph{$K$ is a Calder\'on-Zygmund kernel} (CZ kernel) \emph{for $\mu$}, and we write
\begin{displaymath} \|K\|_{\mathrm{C.Z.}(\mu)} := \sup_{\epsilon > 0} \|T_{\mu,\epsilon}\|_{L^{2}(\mu) \to L^{2}(\mu)} + \|K\|. \end{displaymath}
If $K$ is a $k$-SK with exponent $\alpha \in (0,1]$, and not just a $k$-GSK, we also use the notation
\begin{displaymath} \|K\|_{\mathrm{C.Z.}(\mu),\alpha} := \sup_{\epsilon > 0} \|T_{\mu,\epsilon}\|_{L^{2}(\mu) \to L^{2}(\mu)} + \|K\|_{\alpha,strong}. \end{displaymath}
\end{definition}

\begin{remark}\label{r:maximalRemark} In the introduction -- notably the statements of the main theorems -- we used the terminological convention that $K$ is a CZ kernel for $\mu$ if $\|T^{\ast}_{\mu}\|_{L^{p}(\mu) \to L^{p}(\mu)} < \infty$ for all $1 < p < \infty$. There is no serious conflict: if $\mu$ is a measure on a proper metric space $(X,d)$ satisfying the growth condition \eqref{kLinear}, and $K \colon X \times X \, \setminus \, \bigtriangleup \to \C$ is a $k$-SK, then the condition \eqref{induction} implies that $\|T^{\ast}_{\mu}\|_{L^{p}(\mu) \to L^{p}(\mu)} < \infty$ for all $1 < p < \infty$, see \cite[Theorem 1.1]{MR1626935}. In particular, all of this is true for kernels of the form $(p,q) \mapsto k(q^{-1} \cdot p)$, where $k \colon \He \, \setminus \, \{0\} \to \C$ is a good kernel, and for $\calH^{1}$ measures restricted to regular curves in $\He$.

The reason why we chose to define "CZ kernels" as in Definition \ref{def:epsilonSIO} is that we, sometimes, want to apply the definition to GSKs: the maximal SIO characterisation above may well remain valid in this generality, but at least we have not seen it written down. \end{remark}

For a big part of this paper, we will only be concerned with CZOs, $\epsilon$-SIOs, and maximal SIOs induced by GSKs on $\R$, and the measure $\mu = \mathcal{L}^{1}$. We will drop the sub-index "$\calL^{1}$" in this situation, and write $T,T_{\epsilon},T^{\ast}$ in place of $T_{\calL^{1}},T_{\epsilon,\calL^{1}},T^{\ast}_{\calL^{1}}$. Also, on $\R$, we will only consider CZ kernels for $\calL^{1}$, and write $\|K\|_{\mathrm{C.Z.}} := \|K\|_{\mathrm{C.Z.}(\mathcal{L}^{1})}$.

We will now gather some basic facts about the case $X = \R$ (although many of these statements have generalisations to metric spaces, see for example \cite{MR1626935}).

\begin{proposition}\label{weak11} Let $T$ be a CZO on $\R$. Then $T$ is bounded $L^{1}(\R) \to L^{1,\infty}(\R)$ with norm
\begin{displaymath} \|T\|_{L^{1} \to L^{1,\infty}} \lesssim \|T\|_{\mathrm{C.Z.}}. \end{displaymath}
\end{proposition}

\begin{proof} Applying \eqref{hormanderC2} with $f \equiv 1$ yields \emph{H\"ormander's condition}
\begin{displaymath} \int_{(2B)^{c}} |K(x,y) - K(x,y_{0})| \,dx\leq \|T\|_{\mathrm{C.Z.}}, \qquad y,y_{0} \in I. \end{displaymath}
It follows that $\|T\|_{L^{1} \to L^{1,\infty}} \lesssim \|T\|_{\mathrm{C.Z.}}$, see for example \cite[Exercise 8.2.4]{Grafakos}. \end{proof}

\begin{lemma}[Cotlar's inequality] Let $K \colon \R \times \R \, \setminus \, \bigtriangleup \to \C$ be a bounded GSK, and let $T$ be the CZO induced by $K$. Then, there exists an absolute constant $C \geq 1$ such that
\begin{equation}\label{cotlarsInequality} T^{\ast}f(x) \leq C[M(|Tf|)(x) + \|T\|_{\mathrm{C.Z.}}Mf(x)], \qquad f \in L^{2}(\R), \; x \in \R. \end{equation}
Here $M$ is the (non-centred) Hardy-Littlewood maximal function on $\R$.
\end{lemma}

For a proof, see for instance \cite[p. 56]{MR706075}.

\begin{thm}[$T1$ theorem]\label{T1Theorem} Let $T$ be an operator induced by a bounded SK $K \colon \R \times \R \, \setminus \, \bigtriangleup \to \C$. Then, $T$ is a CZO if and only if $T1,T^{t}1 \in \bmo$, and $T$ satisfies the weak boundedness property (WBP). In this case,
\begin{equation}\label{T1bound} \|T\|_{L^{2} \to L^{2}} \lesssim_{\alpha} \|T1\|_{\bmo} + \|T^{t}1\|_{\bmo} + \|T\|_{\mathrm{WBP}} + \|K\|_{\alpha,strong}. \end{equation}
\end{thm}

For a proof, see \cite[Theorem 8.3.3]{Grafakos}, or the original reference \cite{MR763911}.

\begin{definition}[Definitions of $T1$, $T^{t}1$, and WBP] Under the assumptions of the $T1$ theorem, the condition $T1 \in \bmo$ means that there exists a constant $C \geq 1$ with the following property. If $\varphi \in C^{\infty}(\R)$ is a "smooth $H^{1}$-atom" supported on a ball $B_{0}$, i.e. satisfies
\begin{equation}\label{atoms} \spt \varphi \subset B_{0}, \quad \int_{B_{0}} \varphi = 0, \quad \text{and} \quad \|\varphi\|_{L^{\infty}} \leq |B_{0}|^{-1}, \end{equation}
and $b \in C^{\infty}(\R)$ satisfies $\mathbf{1}_{2B_{0}} \leq b \leq \mathbf{1}_{3B_{0}}$, then
\begin{equation}\label{form99} |\langle T(b),\varphi \rangle| \leq C. \end{equation}
The best constant "$C$", as above, is the definition of the quantity "$\|T1\|_{\bmo}$" in \eqref{T1bound}. The condition $T^{t}1 \in \bmo$ means, by definition, that \eqref{form99} holds with $\langle T(\varphi),b \rangle$ on the left hand side. Finally, the WBP means that if $\varphi,\psi$ are smooth non-negative functions supported on $B(0,1) \subset \R$, with $\max\{\|\varphi\|_{C^{5}},\|\psi\|_{C^{5}}\} \leq 1$,  then
\begin{equation}\label{WBP} |\langle T(\varphi_{x,r}),\psi_{x,r} \rangle | \leq C r^{-1}, \qquad x \in \R, \; r > 0. \end{equation}
Here $f_{x,r}(y) := r^{-1} \cdot f((y - x)/r)$. The best constant "$C$" in \eqref{WBP} is the definition of the quantity "$\|T\|_{\mathrm{WBP}}$" in \eqref{T1bound}.
\end{definition}

\subsubsection{Verifying the $T1$ testing conditions in practice} Let $K \colon \R \times \R \, \setminus \, \bigtriangleup \to \C$ be an SK, not necessarily bounded, let $\epsilon > 0$, and let $\varphi \in C^{\infty}(\R)$ be a fixed, even, bump function satisfying $\mathbf{1}_{B(0,1/2)} \leq \varphi \leq \mathbf{1}_{B(0,1)}$. Writing $\psi_{\epsilon} := 1 - \varphi_{\epsilon}$, we define the \emph{smooth $\epsilon$-SIO} $\tilde{T}_{\epsilon}$ to be the operator induced by the bounded SK
\begin{displaymath} \tilde{K}_{\epsilon}(x,y) := \psi_{\epsilon}(x - y)K(x,y). \end{displaymath}
We also define the formal adjoint $\tilde{T}_{\epsilon}^{t}$ by replacing $K(x,y)$ by $\overline{K(y,x)}$ in the definition above. We record the standard fact that $\|\tilde{K}_{\epsilon}\|_{\alpha,strong} \lesssim_{\varphi} \|K\|_{\alpha,strong}$, where the constants do not depend on $\epsilon > 0$. Now, assume that we can prove the following for some constant $C \geq 1$: if $B_{0}$ is a ball, and $b \in C^{\infty}(\R)$ satisfies $\mathbf{1}_{2B_{0}} \leq b \leq \mathbf{1}_{3B_{0}}$, then
\begin{equation}\label{universalT1} \fint_{B_{0}} |\tilde{T}_{\epsilon}(b)| \leq C \quad \text{and} \quad  \fint_{B_{0}} |\tilde{T}^{t}_{\epsilon}(b)| \leq C. \end{equation}
We claim that
\begin{displaymath} \max\{\|\tilde{T}_{\epsilon}1\|_{\bmo}, \|\tilde{T}_{\epsilon}^{t}1\|_{\bmo}\} \leq C \quad \text{and} \quad \|\tilde{T}_{\epsilon}\|_{\mathrm{WBP}} \lesssim C + \|K\|. \end{displaymath}
The first inequality is immediate from the definitions. To infer the second, fix $x_{0} \in \R$, $r > 0$, write $B_{0} := B(x_{0},r)$, and find $b$ as above \eqref{universalT1}. Then, since $\spt \varphi_{x,r} \subset B_{0}$ (as in \eqref{WBP}), we may write
\begin{align} & |\tilde{T}_{\epsilon}(\varphi_{x_{0},r})(x)| = |\tilde{T}_{\epsilon}[b \varphi_{x_{0},r}](x)| \notag\\
&\label{form101} \quad \leq |\varphi_{x_{0},r}(x) \cdot \tilde{T}_{\epsilon}(b)(x)| + \left| \int_{B(x_{0},r)} b(y)[\varphi_{x_{0},r}(y) - \varphi_{x_{0},r}(x)]K_{\epsilon}(x,y) \, dy \right|. \end{align}
Here,
\begin{displaymath} |\langle \varphi_{x_{0},r} \cdot \tilde{T}_{\epsilon}(b),\psi_{x_{0},r} \rangle | \leq \frac{1}{r^{2}} \int_{B_{0}} |\tilde{T}_{\epsilon}(b)| \lesssim Cr^{-1} \end{displaymath}
by \eqref{form99}. But since $|[\varphi_{x_{0},r}(y) - \varphi_{x_{0},r}(x)]K_{\epsilon}(x,y)| \lesssim r^{-2}\|K\|$, and $b|_{B_{0}} \equiv 1$, the second term on line \eqref{form101} is bounded, for every $x \in B_{0}$, by $|B_{0}|r^{-2}\|K\| \sim r^{-1}\|K\|$. It follows that the WBP \eqref{WBP} holds with constant at most $\lesssim C + \|K\|$, as claimed.

We have established the following corollary of the $T1$ theorem:

\begin{cor}\label{goodT1} Let $K \colon \R \times \R \, \setminus \, \bigtriangleup \to \C$ be an SK, and assume that the testing conditions \eqref{universalT1} hold for some $C \geq 1$, uniformly for $\epsilon > 0$. Then $\|K\|_{\mathrm{C.Z.}} \lesssim_{\alpha} C + \|K\|_{\alpha,strong}$. \end{cor}
\begin{proof} Theorem \ref{T1Theorem} gives the uniform bound $\|\tilde{T}_{\epsilon}\|_{L^{2} \to L^{2}} \lesssim_{\alpha} C + \|K\|_{\alpha,strong}$. This implies \eqref{induction} (for $\mu = \mathcal{L}^{1}$) with roughly the same constants, since $|[T_{\epsilon} - \tilde{T}_{\epsilon}]f| \lesssim \|K\| Mf$.  \end{proof}

\section{Intrinsic Lipschitz graphs and tame maps}\label{s:ILGandTame}

\subsection{Tame maps}\label{s:tame} We say that a map $(\phi_{1},\phi_{2}) \colon E \to \R^{2}$, defined on $E\subset \mathbb{R}$,  is \emph{$L$-tame} if
\begin{equation}\label{tame2} \left|\frac{\phi_{2}(x) - \phi_{2}(y)}{x - y} - \phi_{1}(x) \right| + \left|\frac{\phi_{2}(x) - \phi_{2}(y)}{x - y} - \phi_{1}(y) \right| \leq L|x - y|, \qquad x,y \in E,\;x\neq y. \end{equation}

\begin{remark}\label{r:tame} We make a few hopefully clarifying remarks about the definition of tameness. First, condition \eqref{tame2} is implied (with twice the constant) by a "$1$-sided" version of itself:
\begin{equation}\label{form24} \left|\frac{\phi_{2}(x) - \phi_{2}(y)}{x - y} - \phi_{1}(x) \right| \leq L|x - y|, \quad x,y \in E,\,x\neq y. \end{equation}
Indeed, just apply the inequality above to both $(x,y)$ and $(y,x)$ to arrive at \eqref{tame2}. Second, \eqref{tame2} implies that $\phi_{1}$ is $L$-Lipschitz (by the triangle inequality). Third, assume that $E$ contains an open interval $I$. Then \eqref{tame2} clearly implies that $\dot{\phi}_{2}$ exists on $I$, and $\dot{\phi}_{2} = \phi_{1}$. Conversely, assume that $\phi = (\phi_{1},\phi_{2}) \colon I \to \R^{2}$, where $I \subset \R$ is an open interval, $\phi_{1}$ is $L$-Lipschitz, and $\dot{\phi}_{2} = \phi_{1}$. Then \eqref{form24} is satisfied, because, for $x < y$,
\begin{equation}\label{form14} \left| [\phi_{2}(x) - \phi_{2}(y)] - \phi_{1}(x)(x-y) \right| \leq \int_{x}^{y} |\phi_{1}(s) - \phi_{1}(x)| \, ds \leq L|x - y|^{2}. \end{equation}
So, \eqref{tame2}  and \eqref{form24} are essentially short ways of writing that $\dot{\phi}_{2} = \phi_{1}$ for a $\sim L$-Lipschitz function $\phi_{1}$ without actually mentioning the derivative of $\phi_{2}$. We also note for future reference that the class of $L$-tame maps is preserved under the following operations:
\begin{enumerate}
\item Pre-composing with a translation in $\R$.
\item Adding a map of the form $L_{a,b}(x) := (a,ax + b)$, with $a,b \in \R$.
\end{enumerate}
In fact, the second point is just a special case of the fact that adding an $L_{1}$-tame map to an $L_{2}$-tame map produces an $(L_{1} + L_{2})$-tame map: note that $L_{a,b}$ is $0$-tame for any $a,b \in \R$.
\end{remark}

The next lemma observes that tameness is preserved under parabolic rescaling:

\begin{lemma}\label{lemma1} Let $B = (B_{1},B_{2}) \colon E \to \R^{2}$ be $L$-tame, where $E \subset \R$, and let $r > 0$. Then, the map $B^{r} \colon r^{-1} \cdot E \to \R^{2}$, defined by
\begin{displaymath} B^{r}(x) := (B^{r}_{1}(x),B^{r}_{2}(x)) := \left(\tfrac{1}{r}B_{1}(rx), \tfrac{1}{r^{2}} B_{2}(rx) \right) \end{displaymath}
is also $L$-tame.
\end{lemma}

\begin{proof} For $x,y \in \R$, $x \neq y$, fixed, we note that
\begin{displaymath} \left|\frac{B^{r}_{2}(x) - B^{r}_2(y)}{x - y} - B_{1}^{r}(x) \right| = \frac{1}{r} \left| \frac{B_{2}(rx) - B_{2}(ry)}{(rx - ry)} - B_{1}(rx) \right| \leq \frac{L}{r}|rx - ry| = L|x - y|, \end{displaymath}
as desired. \end{proof}

We then record an extension result:

\begin{proposition}\label{extensionProp} An $L$-tame map defined on $E \subset \R$ extends to an $18L$-tame map defined on $\R$. \end{proposition}

\begin{proof} Let $\phi = (\phi_{1},\phi_{2}) \colon E \to \R^{2}$ be $L$-tame. By assumption, $\phi_{1}$ is Lipschitz, and also $\phi_{2}$ is locally Lipschitz by \eqref{tame2}. So, extending $\phi_{1},\phi_{2}$ to continuous maps on $\bar{E}$ is no problem, and then \eqref{tame2} remains valid on $\bar{E}$. So, we may assume that $E$ is closed to begin with, and we write
\begin{displaymath} \R \, \setminus \, E = \bigcup_{I \in \mathcal{I}} I, \end{displaymath}
where $\mathcal{I}$ are the components of $\R \, \setminus \, E$. We will extend $\phi$ to each interval in $\mathcal{I}$ individually. There are at most two unbounded intervals $I \in \mathcal{I}$. Both of them have an endpoint in $E$, and we define $\phi_{1}$ on $I$ to be the constant attained at the endpoint, say $x$. Then, we define
\begin{displaymath} \phi_{2}(y) := \int_{x}^{y} \phi_{1}(s) \, ds, \qquad y \in I. \end{displaymath}
Evidently $\phi_{1}$ remains $L$-Lipschitz, and we will worry about condition \eqref{tame2} later. Next, fix $I = [x,y] \in \mathcal{I}$ with $x,y \in E$ and $x < y$. Assume for minor notational convenience that
\begin{equation}\label{form10} \phi_{1}(x) = \phi_{2}(x) = x = 0. \end{equation}
This can be achieved by applying the operations (1)-(2) described above. To understand the problem we are now facing, consider any extension of $\phi = (\phi_{1},\phi_{2})$ to $I$, denoted by $\phi^{I} = (\phi_{1}^{I},\phi_{2}^{I})$. Then, if $\phi^{I}$ is supposed to be tame, we should have $\dot{\phi}^{I}_{2} = \phi^{I}_{1}$, and this forces
\begin{equation}\label{form9} \phi_{2}(y) = \phi_{2}^{I}(y) = \int_{0}^{y} \phi_{1}^{I}(s) \, ds. \end{equation}
So, $\phi^{I}_{1}$ needs to be chosen so that \eqref{form9} holds -- and on the other hand $\phi^{I}_{1}$ needs to be a $\sim L$-Lipschitz extension of $\phi_{1}$. In fact, we claim that $\phi^{I}_{1}$ can be taken $7L$-Lipschitz. Let us first attempt the linear extension
\begin{displaymath} \tilde{\phi}^{I}_{1}(s) := \frac{\phi_{1}(y)s}{y}, \qquad s \in I.  \end{displaymath}
This is an $L$-Lipschitz extension of $\phi_{1}$, but
\begin{equation}\label{form11} \int_{0}^{y} \tilde{\phi}^{I}_{1}(s) \, ds = \frac{\phi_{1}(y)y}{2}, \end{equation}
which may not agree with $\phi_{2}(y)$, i.e. the left hand side of \eqref{form9}. However, we are not too far off the mark. Recalling \eqref{form10}, and then using the tameness assumption \eqref{tame2}, we have
\begin{equation}\label{form13} \left|\phi_{2}(y) - \frac{\phi_{1}(y)y}{2} \right| \leq |y| \left| \frac{\phi_{2}(y) - \phi_{2}(0)}{y - 0} - \phi_{1}(0) \right| + \frac{|y||\phi_{1}(y) - \phi_{1}(0)|}{2} \leq \frac{3L|y|^{2}}{2}. \end{equation}
Now, to fix the discrepancy between \eqref{form11} and \eqref{form9}, we choose a $6L$-Lipschitz function $\eta_{I} \colon [0,y] \to \R$ satisfying
\begin{equation}\label{form12} \eta_{I}(0) = 0 = \eta_{I}(y) \quad \text{and} \quad \int_{0}^{y} \eta_I(s) \, ds = \phi_{2}(y) - \frac{\phi_{1}(y)y}{2}. \end{equation}
For example, one can take $\eta_{I} = c\eta_{0}$, where $|c| \leq 1$, and
\begin{equation}\label{form19} \eta_{0}(s) = \begin{cases} 6Ls, & s \in [0,\tfrac{y}{2}], \\ 6L(y - s), & s \in [\tfrac{y}{2},y], \end{cases} \end{equation}
because
\begin{displaymath} \int_{0}^{y} \eta_{0}(s) \, ds = \frac{3L|y|^{2}}{2}, \end{displaymath}
which coincides with the upper bound in \eqref{form13}. Finally, we set
\begin{displaymath} \phi_{1}^{I} := \tilde{\phi}^{I}_{1} + \eta_{I}, \end{displaymath}
which is a $7L$-Lipschitz extension of $\phi_{1}$ (by the first point in \eqref{form12}), and we define $\phi^{I}_{2}$ in the only possible way:
\begin{displaymath} \phi^{I}_{2}(s) := \int_{0}^{s} \phi_{1}^{I}(r) \, \, dr, \qquad s \in I. \end{displaymath}
This function extends $\phi_{2}$ by a combination of \eqref{form11} and the second point in \eqref{form12}.

It remains to check that the tameness condition \eqref{tame2} is satisfied on $\R$, with constant $18L$; in fact, we check the $1$-sided condition \eqref{form24} with constant $9L$. Pick distinct $x,y \in \R$. If $x,y \in E$, there is nothing to prove. The same is true if $x,y$ are contained on (the closure of) a common interval in $\mathcal{I}$, because $\dot{\phi}_{2} = \phi_{1}$ on these intervals, and recalling the estimate \eqref{form14}. So, assume that $x \in E$ and $y \in I \in \mathcal{I}$ with $x < y$, say. Let $x_{1} \in E \cap [x,y)$ be the left endpoint of $I$. Then, use the triangle inequality multiple times:
\begin{align*} |[\phi_{2}(x) - \phi_{2}(y)] - \phi_{1}(x)(x - y)| & \leq |[\phi_{2}(x) - \phi_{2}(x_{1})] - \phi_{1}(x)(x - x_{1})|\\
&\qquad + |[\phi_{2}(x_{1}) - \phi_{2}(y)] - \phi_{1}(x_{1})(x_{1} - y)|\\
&\qquad + |[\phi_{1}(x_{1}) - \phi_{1}(x)](x_{1} - y)|\\
& \leq L|x - x_{1}|^{2} + 7L|x_{1} - y|^{2} + L|x - x_{1}||x_{1} - y|\\
& \leq 9L|x - y|^{2}. \end{align*}
This completes the proof. \end{proof}

\subsubsection{Corona decomposition for tame maps}\label{s:CoronaTame}

In this section, we prove the first main result of this paper, a corona decomposition for maps that are tame in the sense of
\eqref{tame2}. We start with the following rather obvious definition:

\begin{definition}[Tame-linear and tame-affine maps] A map $\phi = (\phi_{1},\phi_{2}) \colon \R \to \R^{2}$ is called \emph{tame-linear (or affine)} if $\phi_{1} \colon \R \to \R$ is linear (or affine) and $\dot{\phi}_{2} = \phi_{1}$. A tame-linear map is \emph{$L$-tame-linear} if $\phi_{1}$ is $L$-Lipschitz. \end{definition}

It would be nice to know the answer to the following question:

\begin{question}\label{q:BPTame} Does there exist a constant $\delta > 0$ with the following property? Let $\phi \colon [0,1] \to \R^{2}$ be $1$-tame. Then there exist a tame-linear map $L \colon \R \to \R^{2}$ and a $(1 - \delta)$-tame map $\phi_{\delta} \colon [0,1] \to \R^{2}$ such that
\begin{displaymath} |\{x \in [0,1] : \phi(x) = [\phi_{\delta} + L](x)\}| \geq \delta. \end{displaymath}
\end{question}

In other words: do $1$-tame maps have big pieces of $(1 - \delta)$-tame maps (up to subtracting a tame-linear map)? Since we were not able to answer this question, we show something slightly weaker, namely that $1$-tame maps admit a "corona decomposition" with $\eta$-tame maps, for any $\eta > 0$. To formulate the statement, we recall some terminology.
\begin{definition}[Dyadic intervals and trees] We write "$\calD$" for the standard dyadic intervals of $\R$. For $j \in \Z$, we further write $\calD_{j} \subset \calD$ for the dyadic intervals $Q$ of length $|Q|=2^{-j}$. A collection $\calT \subset \calD$ is called a \emph{tree} if
\begin{itemize}
\item[(T1)] $\calT$ contains a "top interval" $Q(\calT)$, that is, a unique maximal element.
\item[(T2)] $\calT$ is "coherent": if $Q \in \calT$, then $Q' \in \calT$ for all $Q \subset Q' \subset Q(\calT)$.
\item[(T3)] If $Q \in \calT$, then either both, or neither, of the children of $Q$ lie in $\calT$.
\end{itemize}
\end{definition}

Now we are prepared to formulate the statement of the corona decomposition:

\begin{thm}\label{t:Corona} For every $\eta \in (0,1)$, there exists a constant $C \geq 1$ such that the following holds. Let $\phi \colon \R \to \R^{2}$ be $1$-tame. Then, there exists a decomposition $\calD = \calB \dot{\cup} \calG$ with the following properties. First, the intervals in $\calB$ satisfy a Carleson packing condition:
\begin{equation}\label{badCarleson} \mathop{\sum_{Q \in \calB}}_{Q \subset Q_{0}} |Q| \leq C|Q_{0}|, \qquad Q_{0} \in \calD. \end{equation}
Second, the intervals in $\calG$ can be decomposed into a "forest" $\calF$ of disjoint trees $\calT$,
\begin{equation}\label{treeDecomposition} \calG = \bigcup_{\calT \in \calF} \calT, \end{equation}
whose top intervals satisfy a Carleson packing condition:
\begin{equation}\label{topCarleson} \mathop{\sum_{\calT \in \calF}}_{Q(\calT) \subset Q_{0}} |Q(\calT)| \leq C|Q_{0}|, \qquad Q_{0} \in \calD. \end{equation}
For every $\calT \in \calF$ there exists a $2$-tame-linear map $\calL_{\calT} \colon \R \to \R^{2}$ and an $\eta$-tame map $\psi_{\calT} \colon \R \to \R^{2}$ such that $\psi_{\calT} + \calL_{\calT}$ approximates $\phi$ well at the resolution of the intervals in $\calT$:
\begin{equation}\label{form15} d_{\pi}(\phi(s),[\psi_{\calT} + \calL_{\calT}](s)) \leq \eta |Q|, \qquad s \in 2Q, \; Q \in \calT. \end{equation}
\end{thm}

In \eqref{form15}, $d_{\pi}$ refers to the parabolic metric on $\R^{2}$:
\begin{displaymath} d_{\pi}((x,s),(y,t)) := \max\{|x - y|, \sqrt{|s - t|}\}, \qquad (x,s), (y,t) \in \R^{2}, \end{displaymath}
and $2Q$ is the interval with the same midpoint but twice the length of $Q$. The proof of Theorem \ref{t:Corona} uses, as a black box, the corona decomposition for $\mathbb{R}$-valued Lipschitz functions on $\mathbb{R}$. This statement looks very similar to the one of Theorem \ref{t:Corona}:
\begin{thm}\label{t:LipCorona} For every $\eta \in (0,1)$, there exists a constant $C \geq 1$ such that the following holds. Let $\phi \colon \R \to \R$ be $1$-Lipschitz. Then, there exists a decomposition $\calD = \calB \dot{\cup} \calG$ with the properties \eqref{badCarleson}, \eqref{treeDecomposition},  \eqref{topCarleson}, and such that the following holds. For every $\calT \in \calF$ there exists a $2$-Lipschitz linear function $L_{\calT} \colon \R \to \R$ and an $\eta$-Lipschitz function $\psi_{\calT} \colon \R \to \R$ such that
\begin{equation}\label{form15Lipschitz} |\phi(s) - (\psi_{\calT} + L_{\calT})(s)| \leq \eta |Q|, \qquad s \in 2Q, \; Q \in \calT. \end{equation}
\end{thm}

This statement follows, after a moment's thought, from the corona decomposition in \cite[p.61, (3.33)]{MR1251061}. We give the details in Appendix \ref{s:coronaComparison}. Before proving Theorem \ref{t:Corona}, we record version of Theorem \ref{t:Corona} for $N$-tame maps with $N \geq 1$. The main point here is that the Carleson packing constants do not depend on "$N$", which only makes an appearance in the "quality of approximation" in \eqref{form112}.
\begin{cor}[Corona for $N$-tame maps]\label{cor:Corona} For every $\eta \in (0,1)$, there exists a constant $C \geq 1$ such that the following holds. Let $\phi \colon \R \to \R^{2}$ be $N$-tame, $N \geq 1$. Then, there exists a decomposition $\mathcal{D} = \calB \dot{\cup} \calG$ with the properties \eqref{badCarleson}, \eqref{treeDecomposition}, \eqref{topCarleson}, and such that the following holds. For every $\calT \in \calF$, there exists a $2N$-tame-linear map $\mathcal{L} \colon \R \to \R^{2}$ and an $(\eta N)$-tame map $\psi_{\calT} \colon \R \to \R^{2}$ such that
\begin{equation}\label{form112} d_{\pi}(\phi(s),[\psi_{\calT} + \calL_{\calT}](s)) \leq (\eta N)|Q|, \qquad s \in 2Q, \; Q \in \calT. \end{equation}
\end{cor}

\begin{proof} The map $\tilde{\phi} := N^{-1}\phi \colon \R \to \R^{2}$ is $1$-tame, so Theorem \ref{t:Corona} applies to it verbatim. This yields the desired decomposition $\calD = \mathcal{B} \dot{\cup} \mathcal{G}$ and, for each $\calT \in \calF$, a $2$-tame-linear map $\widetilde{\calL}_{\calT} \colon \R \to \R^{2}$, and an $\eta$-tame map $\tilde{\psi}_{\calT} \colon \R \to \R^{2}$, such that \eqref{form15} holds for $\tilde{\phi},\tilde{\psi}_{\calT},\widetilde{\mathcal{L}}_{\calT}$. Now, we define the $(\eta N)$-tame map $\psi_{\calT} := N\tilde{\psi}_{\calT}$, and the $2N$-tame-linear map $\mathcal{L}_{\calT} := N\widetilde{\mathcal{L}}_{\calT}$. Then,
\begin{displaymath} d_{\pi}(\phi(s),[\psi_{\calT} + \calL_{\calT}](s)) \leq Nd_{\pi}(\tilde{\phi}(s),[\tilde{\psi}_{\calT} + \widetilde{\calL}_{\calT}](s)) \leq (\eta N)|Q| \end{displaymath}
for $s \in 2Q$ with $Q \in \calT$. In the first inequality, we used $N \geq 1$ to infer that $\sqrt{N} \leq N$. \end{proof}

There is also a similar version of Theorem \ref{t:LipCorona} for $M$-Lipschitz functions, $M \geq 1$, but we omit stating this explicitly. We then turn to the proof of Theorem \ref{t:Corona}.

\begin{proof}[Proof of Theorem \ref{t:Corona}] Write $\phi = (\phi_{1},\phi_{2})$, where now $\phi_{1} \colon \R \to \R$ is $1$-Lipschitz. We apply the Lipschitz corona decomposition, Theorem \ref{t:LipCorona}, to $\phi_{1}$ with the parameter $\delta := \min\{\eta^{2}/5,\eta/17\} > 0$. The result is a decomposition $\calD = \calB \cup \calG$ of the type desired in the statement Theorem \ref{t:Corona}, accompanied with the trees $\calT \in \calF$, and corresponding $\delta$-Lipschitz functions $\phi_{\calT} \colon \R \to \R$ and linear $2$-Lipschitz maps $L_{\calT} \colon \R \to \R$ with the property that
\begin{equation}\label{form17} |\phi_{1}(s) - [\phi_{\calT} + L_{\calT}](s)| \leq \delta |Q|, \qquad s \in 2Q, \; Q \in \calT. \end{equation}
Fix a tree $\calT \in \calF$, and consider the top interval $Q(\calT) = [x,y]$. Based on the existence of the function $\phi_{\calT}$, we would now like to produce an $\eta$-tame function $\psi_{\calT} \colon [x,y] \to \R^{2}$ satisfying \eqref{form15}. The tame-linear part will be defined in the obvious way: $\calL_{\calT} = (L_{\calT},P_{\calT}) \colon \R \to \R^{2}$, where
\begin{displaymath} P_{\calT}(s) := \int_{x}^{s} L_{\calT}(r) \, dr, \qquad s \in \R. \end{displaymath}
To define $\psi_{\calT}$, probably the first idea to try is to set $\psi_{1} := \phi_{\calT}$, and define
\begin{equation}\label{form16} \psi_{2}(s) := \phi_{2}(x) + \int_{x}^{s} \psi_{1}(r) \, dr = \phi_{2}(x) + \int_{x}^{s} \phi_{\calT}(r) \, dr, \quad s \in \R. \end{equation}
The good news are that $\dot{\psi}_{2} = \psi_{1}$, and $\psi_{2}(x) = \phi_{2}(x)$, so at least \eqref{form15} is satisfied for $s = x$ (recalling that \eqref{form17} holds, and noting that $\phi_{2}(x) = \psi_{2}(x) + P_{\calT}(x)$). The bad news is that there is no \emph{a priori} reason why $|[\psi_{2} + P_{\calT}](s) - \phi_{2}(s)|$ would be small for any $s \in (x,y]$. To fix this, we in fact need to modify $\phi_{\calT}$ slightly \textbf{before} defining $\psi_{1}$ and $\psi_{2}$ exactly as above.

Let $\calS(\calT)$ be the collection of minimal intervals in $\calT$ (possibly an empty collection). Also, write
\begin{displaymath} E := Q(\calT) \, \setminus \, \bigcup_{S \in \calS(\calT)} S \end{displaymath}
for the set of points in $\calQ(\calT)$ in "infinite branches" of $\calT$. Observe that, by \eqref{form17}, we have
\begin{displaymath} \phi_{1}(s) = [\phi_{\calT} + L_{\calT}](s), \qquad s \in E. \end{displaymath}
Now, for $S \in \calS(\calT)$ fixed, we will slightly modify the restriction of $\phi_{\calT}$ to $\tfrac{1}{2}S$, which is the interval with the same centre but half the length as $S$. The geometric feature of $\tfrac{1}{2}S$ needed in the future is that if $Q \in \calT$ with $|Q| < |S|$, then
\begin{equation}\label{form25} 2Q \cap \tfrac{1}{2}S = \emptyset. \end{equation}
This is clear, because $|Q| < |S|$ forces $Q \cap S = \emptyset$ by the minimality of $S \in \calS(\calT)$.

While modifying $\phi_{\calT}$, we want to maintain the property that $\phi_{\calT}$ is $17\delta$-Lipschitz, and that \eqref{form17} holds with "$\delta$" replaced by "$5\delta$". However, in addition, we want to arrange that
\begin{equation}\label{form18b} \int_{S} \phi_{1}(s) \, ds = \int_{S} [\phi_{\calT} + L_{\calT}](s) \, ds. \end{equation}
The idea is the same as the one already seen during the proof of Proposition \ref{extensionProp}: we want to find a $16\delta$-Lipschitz function $\eta_{S} \colon \overline{\tfrac{1}{2}S} \to \R$ with the properties that
\begin{displaymath} \eta_{S}|_{\partial [\frac{1}{2}S]} = 0 \quad \text{and} \quad \int_{\tfrac{1}{2}S} \eta_{S}(s) \, ds = \int_{S} \phi_{1}(s) - [\phi_{\calT} + L_{\calT}](s) \, ds. \end{displaymath}
This is easily done, using the "triangle" function familiar from \eqref{form19}, and observing that
\begin{displaymath} \left| \int_{S} \phi_{1}(s) - [\phi_{\calT} + L_{\calT}](s) \, ds \right| \leq \delta |S|^{2} \end{displaymath}
by \eqref{form17}. Now, if we replace $\phi_{\calT}$ by $\phi_{\calT} + \eta_{S}$ on $S$, we find that the "new" $\phi_{\calT}$ is $17\delta$-Lipschitz, and \eqref{form18b} holds. Moreover, since $\|\eta_{S}\|_{L^{\infty}(S)} \leq 4\delta |S|$, there is some hope that \eqref{form17} remains valid with ``$\delta$'' replaced by ``$5\delta$''. To prove this carefully, fix $Q \in \calT$ and $s \in 2Q$. During the procedure above, we only modified $\phi_{\calT}$ on sets of the form $\tfrac{1}{2}S$, with $S \in \calS(\calT)$. So, if $s \notin \tfrac{1}{2}S$ for any $S \in \calS(\calT)$, then \eqref{form17} is certainly valid, with original constant. So, assume that $s \in \tfrac{1}{2}S$ for some $S \in \calS(\calT)$. Then $s \in 2Q \cap \tfrac{1}{2}S$, so \eqref{form25} forces $|S| \leq |Q|$. Consequently,
\begin{displaymath} \|\eta_{S}\|_{L^{\infty}} \leq 4\delta|S| \leq 4\delta|Q|. \end{displaymath}
Since the "original" $\phi_{\calT}$ only differs from the "new" $\phi_{\calT}$ on $\tfrac{1}{2}S$ by the function $\eta_{S}$, we see that
\begin{displaymath} |\phi_{1}(s) - [\phi_{\calT} + L_{\calT}](s)| \leq \delta |Q| + 4\delta |Q| = 5\delta |Q|, \end{displaymath}
as desired.

Now, assume that similar modifications to $\phi_{\calT}$ have been performed inside all intervals $S \in \calS(\calT)$, and in particular \eqref{form18b} holds for all $S \in \calS(\calT)$. We infer the following corollary: if $s \in Q(\calT)$, and either
\begin{displaymath} s \in E \quad \text{or} \quad s \in \partial S \text{ with } S \in \calS(\calT), \end{displaymath}
then
\begin{equation}\label{form20b} \int_{x}^{s} \phi_{1}(s) \, ds = \int_{x}^{s} [\phi_{\calT} + L_{\calT}](s) \, ds. \end{equation}
Recall that $x$ is the left endpoint of $Q(\calT)$. Now, with the fine-tuned definition of $\phi_{\calT}$, we proceed as planned, setting $\psi_{1} := \phi_{\calT}$ and defining $\psi_{2}$ as in \eqref{form16}. Since the map $\psi = (\psi_{1},\psi_{2}) \colon Q(\calT) \to \R$ is now $17\delta$-tame, and $17\delta \leq \eta$ by definition, it remains to check that \eqref{form15} holds for all $x \in Q \in \calT$. This amounts to checking that
\begin{equation}\label{form21b} |\phi_{2}(s) - [\psi_{2} + P_{\calT}](s)| \leq
\eta^2|Q|^{2}, \qquad s \in
2Q \in \calT. \end{equation}
First, consider $s \in E$. Then, since $\dot{\phi}_{2} = \phi_{1}$, we have
\begin{equation}\label{form26} \phi_{2}(s) = \phi_{2}(x) + \int_{x}^{s} \phi_{1}(s) \, ds \stackrel{\eqref{form20b}}{=} \phi_{2}(x) + \int_{x}^{s} [\phi_{\calT} + L_{\calT}](s) \, ds = \psi_{2}(s) + P_{\calT}(s). \end{equation}
So, the difference in \eqref{form21b} is zero, as it should be. Next, fix some $Q \in \calT$, and consider $s \in 2Q$. Then, there exists a point
\begin{displaymath} s_{1} \in Q \cap \left[E \cup \bigcup_{S \in \calS(\calT)} \partial S \right] \end{displaymath}
satisfying $|s - s_{1}| \leq |Q|$. Then $\phi_{2}(s_{1}) = \psi_{2}(s_{1}) + P_{\calT}(s_{1})$, repeating the computation on line \eqref{form26}. Consequently,
\begin{align*} |\phi_{2}(s) - [\psi_{2} + P_{\calT}](s)| & = |\phi_{2}(s) - \phi_{2}(s_{1}) - ([\psi_{2} + P_{\calT}](s) - [\psi_{2} + P_{\calT}](s_{1}))|\\
& = \left| \int_{s_{1}}^{s} \phi_{1}(r) \, dr - \int_{s_{1}}^{s} [\phi_{\calT} + L_{\calT}](r) \, dr \right|\\
& \leq \int
_I |\phi_{1}(r) - [\phi_{\calT} + L_{\calT}](r)| \, dr \leq 5\delta |Q|^{2}, \end{align*}
noting in the last inequality that the interval $I$ between $s_1$ and $s$ satisfies
$I \subset 2Q$, so \eqref{form17} (with ``$5\delta$'' in place of ``$\delta$'') holds for all points in
$I$. We conclude from this estimate and \eqref{form17} that
\begin{displaymath} d_{\pi}(\phi(s),[\psi + \calL_{\calT}](s)) \leq \max\{5\delta|Q|,\sqrt{5\delta}|Q|\} \leq \eta|Q|, \qquad s \in 2Q, \; Q \in \calT, \end{displaymath}
recalling that $\sqrt{5\delta} \leq \eta$. The proof is complete. \end{proof}

Tame maps will now go away for a moment, but they will return in Section \ref{s:IntrLip}, where we relate them to intrinsic Lipschitz functions on the Heisenberg group.

\subsection{The Heisenberg group}\label{s:Heis}

\begin{definition}[Heisenberg group, dilations, and
distance]\label{d:Heis} The \emph{Heisenberg group} $\He$ is the
group $(\mathbb{R}^3,\cdot)$ with
\begin{displaymath} (x,y,t) \cdot (x',y',t') := (x + x', y + y', t + t' + \tfrac{1}{2}(xy' - x'y)),\quad(x,y,t),(x',y',t')\in\mathbb{R}^3. \end{displaymath}
The \emph{Heisenberg dilations} $(\delta_{\lambda})_{\lambda>0}$ are the group automorphisms
\begin{displaymath}
\delta_{\lambda}:\He\to\He,\quad \delta_{\lambda}(x,y,t)=(\lambda x,\lambda y,\lambda^2 t).
\end{displaymath}
Given $\alpha \in \R$, a function $h \colon \He \, \setminus \, \{0\} \to \C$ is called \emph{$\alpha$-homogeneous} with respect to the dilations above if $h(\delta_{r}(p)) = r^{\alpha}h(p)$ for all $p \in \He \, \setminus \, \{0\}$ and $r > 0$. We define the \emph{Heisenberg metric} $d:\He \to \He \to [0,+\infty)$ by setting $d(p,q):=\|q^{-1}\cdot p\|$,
where
\begin{equation}\label{d:maxNorm} \|(x,y,t)\| := \max\{\sqrt{x^{2} + y^{2}},\sqrt{|t|}\}. \end{equation}
\end{definition}
\begin{remark} In the introduction, we used the notation "$\|\cdot\|$" for the \emph{Kor\'anyi norm} $\|(x,y,t)\| = ((x^{2} + y^{2})^{2} + 16t^{2})^{1/4}$, which is a quantity comparable to the "max-norm" in \eqref{d:maxNorm}. From now on, $\|\cdot\|$ always refers to the quantity in \eqref{d:maxNorm}. \end{remark}

\begin{definition}[Horizontal gradient]\label{d:DefHorizGrad}
Let $\Omega \subset \He$ be an open set. The \emph{horizontal gradient} of a $\mathcal{C}^1$ function $u:\Omega \to \mathbb{R}$ is defined by
\begin{displaymath}
\nabla_{\He}u =(Xu,Yu),
\end{displaymath}
where
\begin{displaymath}
X:=\partial_x -\tfrac{y}{2}\partial_t\quad\text{and}\quad Y:=\partial_y+\tfrac{x}{2}\partial_t.
\end{displaymath}
\end{definition}

\begin{definition}[Homogeneous subgroups]
A subgroup of $\He$ is \emph{homogeneous} if it is closed under dilations. Homogeneous subgroups of $\He$ are either contained in the $xy$-plane, in which case they are called \emph{horizontal}, or they contain the $t$-axis, in which case they are said to be \emph{vertical}.
\end{definition}

\begin{definition}[Horizontal lines]\label{d:horizontalLine} A left translate of a non-trivial horizontal subgroup $\V \subset \He$ is called a \emph{horizontal line} in $\He$. \end{definition}

\begin{definition}[Projections and components]\label{def1}
Let $\W \subset \He$ be a vertical subgroup of topological dimension $2$. We associate to $\W$ the unique horizontal subgroup $\mathbb{L} \subset \W$, and the \emph{complementary} horizontal subgroup $\V$. The choice of $\V$ is somewhat arbitrary, but we declare here $\V$ to be the Euclidean orthogonal complement of $\mathbb{L}$ in the $xy$-plane. We write $\mathbb{T}$ for the $t$-axis. Then, every point $p \in \He$ has a unique "coordinate" decomposition
\begin{displaymath} p = v \cdot w = v \cdot l \cdot t, \end{displaymath}
where $w = l \cdot t = t \cdot l \in \W$ with $l \in \mathbb{L}$ and $t \in \mathbb{T}$, and $v \in \mathbb{V}$. This decomposition gives rise to the \emph{vertical projections} $\pi_{\W} \colon \He \to \W$ and $\pi_{\mathbb{T}} \colon \He \to \mathbb{T}$, given by $p \mapsto w$ and $p \mapsto t$, and the \emph{horizontal projections} $\pi_{\V} \colon \He \to \V$ and $\pi_{\mathbb{L}} \colon \He \to \mathbb{L}$, given by $p \mapsto v$ and $p \mapsto l$, respectively. The horizontal projections are $1$-Lipschitz group homomorphisms, while $\pi_{\W}$ and $\pi_{\mathbb{T}}$ are neither Lipschitz maps nor group homomorphisms. Nevertheless, $\pi_{\mathbb{T}}$ and $\pi_{\W}$ satisfy
\begin{equation}\label{form3} \|\pi_{\mathbb{T}}(p)\| \leq \|\pi_{\W}(p)\| \leq C\|p\|, \qquad p \in \He \end{equation}
for some absolute constant $C \geq 1$. If $\phi \colon X \to \W$ is a map, where $X$ is any set, we define the \emph{first and second components of $\phi$} to be the functions $\phi_{1} = \pi_{\mathbb{L}} \circ \phi \colon X \to \mathbb{L}$ and $\phi_{2} = \pi_{\mathbb{T}} \circ \phi \colon X \to \mathbb{T}$. \end{definition}

\begin{remark}\label{rem2} If $\W = \mathbb{L} \times \mathbb{T}$ is a vertical subgroup with complementary subgroup $\V$, we will write in coordinates $\W = \{y \cdot t : y \in \mathbb{L} \text{ and } t \in \mathbb{V}\} \cong \{(y,t) : y,t \in \R\} = \R^{2}$. Similarly, $\V$ will be identified with $\R$. Under these identifications, the components $\phi_{1} \colon \V \to \mathbb{L}$ and $\phi_{2} \colon \V \to \mathbb{T}$ of any map $\phi \colon \V \to \W$ can be seen as functions $\R \to \R$, and in particular the derivative notation "$\dot{\phi}_{j}$" should be understood in this sense. \end{remark}

\subsection{Intrinsic Lipschitz graphs}\label{s:IntrLip} We define intrinsic Lipschitz functions and graphs over horizontal subgroups in $\He$.
On the one hand, this is just a special case of a definition of
Franchi, Serapioni, and Serra Cassano \cite{FSS}. On the other
hand, intrinsic Lipschitz functions over horizontal subgroups have
nicer properties than those over vertical subgroups, essentially
because $\pi_{\V}$ is a group homomorphism. Higher dimensional
intrinsic Lipschitz graphs will only be mentioned in passing in
this paper, in Section \ref{s:SIOOnFLagSUrfaces}.

\begin{definition}[Intrinsic $L$-Lipschitz graphs and functions] For $\W,\V$ as in Definition \ref{def1}, and $\alpha > 0$, we define the \emph{cone}
\begin{displaymath} C_{\V}(\alpha) := \{p \in \He : \|\pi_{\V}(p)\| \leq \alpha \|\pi_{\W}(p)\|\}. \end{displaymath}
A set $\Gamma \subset \He$ is called an \emph{intrinsic $L$-Lipschitz graph over $\V$}, or simply an \emph{intrinsic Lipschitz graph}, if there exists $L > 0$ such that
\begin{equation}\label{lip}
\left(p\cdot C_{\V}\left(\alpha\right)\right)\cap \Gamma = \{p\},\quad\text{for all }p\in \Gamma\text{ and all }\alpha<\frac{1}{L}.
 \end{equation}
Let $\phi \colon E \to \W$ be a map, where $E \subset \V$. The function $\phi$ is called \emph{intrinsic $L$-Lipschitz} if
$\Gamma(\phi):=\{v\cdot \phi(v):\; v\in E\}$ is an intrinsic $L$-Lipschitz graph. The map $v \mapsto \Phi(v) := v \cdot \phi(v)$ is called the \emph{graph map} of $\phi$.
\end{definition}

\begin{proposition}\label{prop1} A set $\Gamma \subset \He$ is an intrinsic Lipschitz graph over a horizontal subgroup $\V$ if and only if the horizontal projection $\pi_{\V}$ restricted to $\Gamma$ is injective with metric Lipschitz inverse $\Phi_{\Gamma} \colon \pi_{\V}(\Gamma) \to \Gamma$. \end{proposition}

\begin{proof} Let $\Gamma \subset \He$ be an intrinsic $L$-Lipschitz graph over $\V$. If $p,q \in \Gamma$ then
\begin{equation}\label{form2} \|\pi_{\V}(q)^{-1} \cdot \pi_{\V}(p)\| = \|\pi_{\V}(q^{-1} \cdot p)\| \stackrel{\eqref{lip}}{\geq} \tfrac{1}{L}\|\pi_{\W}(q^{-1} \cdot p)\|, \end{equation}
which implies by the triangle inequality that $\|q^{-1} \cdot p\| \leq (1 + L)\|\pi_{\V}(q)^{-1} \cdot \pi_{\V}(p)\|$. Consequently, the projection $\pi_{\V}$ restricted to $\Gamma$ is bilipschitz, so the map $\Phi_{\Gamma} \colon \pi_{\V}(\Gamma) \to \Gamma$, given by the relation $\pi_{\V}(\Phi_{\Gamma}(v)) = v$, is well-defined and $(1 + L)$-Lipschitz

Conversely, assume that $\Gamma \subset \He$ is a set such that the horizontal projection $\pi_{\V}$ restricted to $\Gamma$ is injective with $L$-Lipschitz inverse $\Phi$. Then, if $p = \Phi(v),q = \Phi(v') \in \Gamma$, we have
\begin{displaymath}  \|\pi_{\W}(\Phi(v')^{-1} \cdot \Phi(v))\| \stackrel{\eqref{form3}}{\leq} C\|\Phi(v')^{-1} \cdot \Phi(v)\| \leq CL\|(v')^{-1} \cdot v\| = CL \|\pi_{\V}(q^{-1} \cdot p)\|.\end{displaymath}
which shows that $\Gamma$ is an intrinsic $CL$-Lipschitz graph over $\V$. \end{proof}

\begin{remark} We record that every intrinsic $L$-Lipschitz graph $\Gamma \subset \He$ can be parametrised by an intrinsic $L$-Lipschitz function defined on $E := \pi_{\V}(\Gamma) \subset \V$. Simply, let $\Phi_{\Gamma} \colon E \to \Gamma$ be the map defined in Proposition \ref{prop1}, and let
\begin{equation}\label{form1} \phi_{\Gamma}(v) := \pi_{\W}(\Phi_{\Gamma}(v)). \end{equation}
Then $\Phi_{\Gamma}(v) =  \pi_{\V}(\Phi_{\Gamma}(v)) \cdot \pi_{\W}(\Phi_{\Gamma}(v)) = v \cdot \phi_{\Gamma}(v)$ for $v \in E$, so indeed $\Gamma = \Gamma(\phi)$. Thus, $\Gamma$ is parametrised by $\phi$, and $\phi$ is intrinsic $L$-Lipschitz by definition.  \end{remark}

\begin{lemma} Let $\phi \colon E \to \W$ be an intrinsic $L$-Lipschitz function, with $E \subset \V$. Then the first component $\phi_{1}$, recall Definition \ref{def1}, is $L$-Lipschitz. Consequently, under the identification from Remark \ref{rem2}, the function $\phi_1:\mathbb{R}\to\mathbb{R}$ is Euclidean Lipschitz.  \end{lemma}
\begin{proof} Indeed, recall from \eqref{form1} that $\phi(v) = \pi_{\W}(\Phi(v))$, where $\Phi \colon E \to \Gamma$ is the graph map of $\Gamma(\phi)$. Consequently $\phi_{1} = \pi_{\mathbb{L}} \circ \Phi$. Then, using the fact that $\pi_{\mathbb{L}}$ is a group homomorphism, we infer that
\begin{align*} \|\phi_{1}(v')^{-1} \cdot \phi_{1}(v)\| & = \|\pi_{\mathbb{L}}(\Phi(v'))^{-1} \cdot \pi_{\mathbb{L}}(\Phi(v))\|\\
& \leq \|\pi_{\W}(\Phi(v')^{-1} \cdot \Phi(v))\|\\
& \stackrel{\eqref{form2}}{\leq} L\|\pi_{\V}(\Phi(v'))^{-1} \cdot \pi_{\V}(\Phi(v))\| = L\|(v')^{-1} \cdot v\| \end{align*}
for all $v,v' \in E$. \end{proof}

We conclude this section with an area formula for intrinsic Lipschitz graphs over horizontal subgroups.

\begin{proposition}\label{p:areaFormula}
Let $\phi=(\phi_1,\phi_2): I\subset \mathbb{V}\to\mathbb{W}$ be an intrinsic Lipschitz map defined on an interval $I\subset\mathbb{V}$, and let $\Phi$ be its graph map. Then, $\Phi(I)$ is  a $1$-regular subset of $(\He,d)$ and
\begin{equation}\label{eq:AreaCurve}
\mathcal{H}^1(\Phi(A))= \int_A \left(1+\dot{\phi}_1(v)^2 \right)^{1/2} dv,\quad A\subset I \text{ Borel.}
\end{equation}
\end{proposition}

\begin{proof} By Proposition \ref{prop1}, the map $\Phi:I \to (\He,d)$ is a Lipschitz curve, and $\Phi$ is in fact bi-Lipschitz onto its image since horizontal projections are Lipschitz. As $\Phi$ is injective, the length with respect to the metric $d$ of a subcurve $\Phi([a,b])$, $[a,b]\subset I$, agrees with $\mathcal{H}^1(\Phi([a,b]))$, see for instance \cite[Theorem 2.6.2.]{MR1835418}. Moreover,
\begin{equation}\label{form136}
\mathrm{length}_{|\cdot|}(\pi(\Phi([a,b])))\leq \mathrm{length}_d(\Phi([a,b])) \leq \mathrm{length}_{cc}(\Phi([a,b])),
\end{equation}
where the left-hand side denotes the Euclidean length of the image of $\Phi([a,b])$ under the projection $\pi:\He \to \mathbb{R}^2$, $(x,y,t)\mapsto (x,y)$, and $d_{cc}$  is the standard sub-Riemannian distance on $\He$, see \cite{MR1421822}. Since $\pi\circ \Phi$ is (Euclidean) Lipschitz, the left-hand side of \eqref{form136} equals
\begin{displaymath}
\int_a^b |(\pi\circ\Phi)'(v)| \,dv,
\end{displaymath}
and the same is true for the right-hand side, cf.\ e.g.\ \cite{MR3417082}. Using
\begin{displaymath}
|(\pi\circ\Phi)'(v)| = \left( |\pi_{\mathbb{V}}(\Phi(v))'|^2 + |\pi_{\mathbb{L}}(\Phi(v))'|^2 \right)^{1/2} = \left(1 +\dot{\phi}_1(v)^2\right)^{1/2},
\end{displaymath}
we have thus established \eqref{eq:AreaCurve} for $A=[a,b]$. The case of Borel sets $A \subset I$ follows by approximation.
\end{proof}

\subsubsection{Connection between tame maps and intrinsic Lipschitz graphs} In this section, let $\W = \{(0,y,t) : y,t \in \R\}$, $\mathbb{L} = \{(0,y,0) : y \in \R\}$, and $\V = \{(x,0,0) : x \in \R\}$. As we discussed in  Remark \ref{rem2}, we will identify $\W \cong \R^{2}$ and $\V \cong \R \cong \mathbb{L}$. With these identifications, we have the following relationship between intrinsic Lipschitz functions and tame maps.
\begin{proposition}\label{prop2} Let $E \subset \V$. If $\phi = (\phi_{1},\phi_{2}) \colon E \to \W$ is intrinsic $L$-Lipschitz, then $(\phi_{1},-\phi_{2}) \colon E \to \R^{2}$ is $2L^{2}$-tame.
\end{proposition}

\begin{proof} A formula for the vertical projection $\pi_{\W}$ is
\begin{displaymath} \pi_{\W}(x,y,t) = (y,t - \tfrac{xy}{2}), \qquad (x,y,t) \in \He, \end{displaymath}
while $\pi_{\V}(x,y,t) = x$. The graph map of $\phi$ is given by
\begin{displaymath} \Phi(v) = v \cdot \phi(v) = (v,\phi_{1}(v),\phi_{2}(v) {+} \tfrac{\phi_{1}(v)v}{2}), \qquad v \cong (v,0,0) \in E,  \end{displaymath}
and consequently $\Phi(v_{1})^{-1} \cdot \Phi(v_{2}) = $
\begin{equation}\label{eq:form17B}  \left(v_{2} - v_{1}, \phi_{1}(v_{2}) - \phi_{1}(v_{1}), \phi_{2}(v_{2}) - \phi_{2}(v_{1}) + \frac{\phi_{1}(v_{1}) + \phi_{1}(v_{2})}{2}(v_{2} - v_{1})\right). \end{equation}
Since $\phi \colon E \to \W$ is intrinsic $L$-Lipschitz, $\Phi(E)$ is an intrinsic $L$-Lipschitz graph, which means that
\begin{align*} \|\pi_{\W}(\Phi(v_{1})^{-1} \cdot \Phi(v_{2}))\| \leq L\|\pi_{\V}(\Phi(v_{1})^{-1} \cdot \Phi(v_{2}))\|, \qquad v_{1},v_{2} \in E. \end{align*}
Spelling out the last condition, one finds that
\begin{equation}\label{form22} |\phi_{1}(v_{2}) - \phi_{1}(v_{1})| \leq L|v_{2} - v_{1}|, \qquad v_{1},v_{2} \in E, \end{equation}
and
\begin{equation}\label{form23} \left|\frac{\phi_{2}(v_{2}) - \phi_{2}(v_{1})}{v_{2} - v_{1}} + \phi_{1}(v_{1}) \right| \leq L^{2}|v_{2} - v_{1}|, \qquad v_{1},v_{2} \in E, \: v_{1} \neq v_{2}. \end{equation}
But \eqref{form23} is exactly the $1$-sided tameness condition \eqref{form24} for the map $(\phi_{1},-\phi_{2})$. \end{proof}

\begin{remark} Recall from Remark \ref{r:tame} that the first component of an $L$-tame functions is automatically $L$-Lipschitz. Thus, if conditions \eqref{form22}-\eqref{form23} hold for some $L < 1/2$, then actually \eqref{form22} holds with the better constant "$2L^{2}$"! On the other hand, assume that \eqref{form22}-\eqref{form23} hold for some $L \geq 1$, and $E$ contains an open interval $I$. Then $\dot{\phi}_{2}(v) = -\phi_{1}(v)$ for $v \in I$ which implies, by the calculation in \eqref{form14}, that \eqref{form23} actually holds with constant "$L$" for $v_{1},v_{2} \in I$.

In conclusion, if $E$ is an interval, the best constants in the inequalities \eqref{form22} and \eqref{form23} are actually within a multiple of "$2$" from each other. \end{remark}


Thanks to the connection between tame maps and intrinsic Lipschitz functions, Proposition \ref{extensionProp} (extension of tame maps) implies an extension result for intrinsic Lipschitz graphs over horizontal subgroups.

\begin{proposition}\label{p:lipext}
Let $\phi:E\to \mathbb{W}$ be an intrinsic $L$-Lipschitz function. Then there exists an intrinsic $L'$-Lipschitz function $\widetilde{\phi}:\mathbb{V}\to\mathbb{W}$ for $L'\lesssim \max\{L,L^2\}$ such that $\widetilde{\phi}|_E=\phi$.
\end{proposition}

\begin{proof}
Since $\phi=(\phi_1,\phi_2)$ is intrinsic $L$-Lipschitz by assumption, the map $(\phi_1,-\phi_2)$ is $2L^2$-tame according to
 Proposition \ref{prop2}. The extension result from  Proposition \ref{extensionProp} then allows us to find a $36 L^2$-tame map $(\widetilde{\phi}_1,-\widetilde{\phi}_2):\mathbb{R}\to\mathbb{R}^2$ with $(\widetilde{\phi}_1,-\widetilde{\phi}_2)|_E= (\phi_1,-\phi_2)$. Thus, $\widetilde{\phi}=(\widetilde{\phi}_1,\widetilde{\phi}_2)$ satisfies the conditions \eqref{form22} and \eqref{form23} for all $v_1,v_2\in\mathbb{R}$, $v_1\neq v_2$, with ``$L$'' replaced by $L'=\max\{6L, 36 L^2\}$.
\end{proof}

\section{The exponential kernel appears}\label{s:Exp1}

\subsection{Weakly good kernels on intrinsic Lipschitz graphs}\label{s:SIOOnILG}
We fix a weakly good kernel $k \colon \He \, \setminus \{z=0\} \to
\C$, and gradually start proving that it is a CZ kernel for
($\calH^{1}$ restricted to) any intrinsic Lipschitz graph over a
horizontal subgroup in $\He$. We fix a horizontal subgroup $\V$
with complementary vertical subgroup $\W$, and an intrinsic
$L$-Lipschitz function $\phi = (\phi_{1},\phi_{2}) \colon \V \to
\W$, for $L \geq 1$. We assume with no loss of generality that $\V
\cong \{(x,0,0) : x \in \R\} \cong \R$ and $\W \cong \{(0,y,t) :
y,t \in \R\} \cong \R^{2}$. The first point of this section is to
show how Theorem \ref{t:mainIntrLipGraph} can be reduced to a
statement involving only Lipschitz functions, tame maps, and
standard kernels on $\R$, see Theorem \ref{t:technical} below.

Let $\Phi$ be the graph map of $\phi$, and let $\Gamma = \Phi(\V)
\subset \He$ be the intrinsic graph of $\phi$. Write $\mu :=
\mathcal{H}^1|_{\Gamma}$. Since $k$ is only assumed to be weakly
good, we cannot hope that the function $K(p,q) := k(q^{-1} \cdot
p)$, defined for $q^{-1}\cdot p \in \He \setminus \{z=0\}$, would
extend to an SK in $\He$. However, it turns out that the
restriction of $K$ to $\Gamma \times \Gamma \, \setminus \,
\bigtriangleup$ is an SK with $\alpha = 1$, and indeed a CZ kernel
for $\mu$. This is what is meant by the statement of Theorem
\ref{t:mainIntrLipGraph}.

In place of $K$, we plan to study the parametric kernel
$K_{\Phi}(w,v) := K(\Phi(w),\Phi(v))$ on $\R \times \R \,
\setminus \, \{w = v\}$. Theorem \ref{t:technical} below will
imply that $K_{\Phi}$ is a CZ kernel for $\mathcal{L}^{1}$ in
$\R$, with $\|K_{\Phi}\|_{\mathrm{C.Z.},1} \lesssim_{k,L} 1$. Let
us briefly argue why this implies Theorem
\ref{t:mainIntrLipGraph}. First, since $\Phi \colon \R \to
(\Gamma,d)$ is $(1 + L)$-bilipschitz by Proposition \ref{prop1},
it follows easily that
\begin{displaymath} \|K\|_{1,strong} \lesssim_{L} \|K_{\Phi}\|_{1,strong} \lesssim_{k,L} 1, \end{displaymath}
where the left hand side is refers to the standard kernel constant
in the metric space $(\Gamma,d)$. We then relate the
$\epsilon$-SIOs  $T_{\mu,\epsilon}$ induced by $(K,\mu)$ to the
$\epsilon$-SIOs $T_{\epsilon}$ induced by
$(K_{\Phi},\mathcal{L}^{1})$. The area formula, Proposition
\ref{p:areaFormula}, implies that
\begin{equation}\label{eq:SIOHeis} T_{\mu,\varepsilon}g(\Phi(w)) = \int_{\{v\in\R:\,d(\Phi(v),\Phi(w))>\varepsilon\}} K(\Phi(w), \Phi(v))g(\Phi(v))\left(1 + \dot{\phi}_{1}(v)^{2}\right)^{1/2} \, dv \end{equation}
for all $w \in \R$ and $g \in L^2(\mu)$. Here $d(\Phi(v),\Phi(w))
\sim_{L} |w - v|$ for all $w,v \in \R$, and
$$1\leq J(v) := \left(1 + \dot{\phi}_{1}(v)^{2}\right)^{1/2} \leq
\left( 1+L^2 \right)^{1/2} \quad\text{for a.e.\ }v \in \R.$$ It
follows easily that
\begin{displaymath} |T_{\mu,\epsilon}g(\Phi(w)) - T_{\epsilon}(J \cdot [g \circ \Phi])(w)| \lesssim_{k,L} M(g \circ \Phi)(w), \qquad w \in \R, \, g \in L^{2}(\mu), \, \epsilon > 0. \end{displaymath}
Therefore, using the area formula again,
\begin{equation}\label{form21}
\sup_{\varepsilon>0}\|T_{\mu,\varepsilon}\|_{L^2(\mu)\to L^2(\mu)}
\leq \sup_{\varepsilon>0} C_{L}\|T_{\varepsilon}\|_{L^2(\R)\to
L^2(\R)} + C_{k,L}. \end{equation} So, Theorem
\ref{t:mainIntrLipGraph} has now been reduced to proving that
$K_{\Phi}$ is a CZ kernel on $\R$. We establish the following
slightly stronger result for future purposes:
\begin{thm}\label{t:technical} Let $K_{\Phi} \colon \R \times \R \, \setminus \, \bigtriangleup \to \C$ be the kernel
$K_{\Phi}(w,v) := k(\Phi(v)^{-1} \cdot \Phi(w))$, where $k \in
C^{\infty}(\He \, \setminus \, \{z=0\})$ is a weakly good kernel,
and $\Phi$ is the graph map of an intrinsic Lipschitz function
$\phi \colon \V \to \W$. Also, let $A_{0} \colon \R \to \R$ be
Lipschitz, let $B_{0} = (B_{1},B_{2}) \colon \R \to \R^{2}$ be
tame, and define the quantities
\begin{equation}\label{form152} D_{A_{0}}(w,v) := \tfrac{A_{0}(w) - A_{0}(v)}{w - v} \quad \text{and} \quad D_{B_{0}}(w,v) := \tfrac{B_{2}(w) - B_{2}(v) - \tfrac{1}{2}[B_{1}(w) + B_{1}(v)](w - v)}{(w - v)^{2}} \end{equation}
for $w,v \in \R$ with $w \neq v$. Then
\begin{displaymath} K_{\Phi,A_{0}} := K_{\Phi}D_{A_{0}} \quad \text{and} \quad K_{\Phi,B_{0}} := K_{\Phi}D_{B_{0}} \end{displaymath}
are SKs and CZ kernels on $\R$, with constants depending only on
the weak goodness constants of $k$, the intrinsic Lipschitz
constant of $\phi$, and the Lipschitz and tameness constants of
$A_{0}$ and $B_{0}$. \end{thm}

\begin{remark} Note that $K_{\Phi} = K_{\Phi}D_{A_{0}}$ with $A_{0}(x) = x$. The reader is encouraged to ignore the factors $D_{A_{0}}$ and $D_{B_{0}}$ completely; the additional generality will bring no extra challenges, but will be useful in an application. Why are there no extra difficulties? The proof of Theorem \ref{t:technical}, even without the factors $D_{A_{0}}$ and $D_{B_{0}}$, is based on decomposing the kernel $K_{\Phi}$ into a sum of the form
\begin{equation}\label{form153} K_{\Phi}(w,v) = \sum_{\mathbf{n}} \kappa_{\mathbf{n}}(w - v)\prod_{i = 1}^{m_{\mathbf{n}}} D^{\mathbf{n}}_{A_{i}} \prod_{i = 1}^{n_{\mathbf{n}}} D^{\mathbf{n}}_{B_{i}}, \end{equation}
where $D_{A_{i}}^{\mathbf{n}}$ and $D_{B_{j}}^{\mathbf{n}}$ are
factors of the kind appearing in \eqref{form152}, and
$\kappa_{\mathbf{n}}$ is an odd standard kernel on $\R$. So, if
$K_{\Phi}$ is multiplied by $D_{A_{0}}$ or $D_{B_{0}}$, the
decomposition \eqref{form153} will turn up looking the same, with
two additional factors, and its treatment will not change
noticeably. The same argument would even extend to show that
$K_{\Phi}D_{A_{1}}\cdots D_{A_{m}}D_{B_{1}}\cdots D_{B_{n}}$ is a
CZ kernel for any $m,n \geq 0$, and for any factors of the form
$D_{A_{i}}$ and $D_{B_{i}}$ as in \eqref{form152}.
 \end{remark}

We then start the proof of Theorem \ref{t:technical} for an
intrinsic $L$-Lipschitz function $\phi$ by defining an auxiliary
function
\begin{displaymath} \kappa(u;\theta_{1},\theta_{2}) := \chi(\theta_{1},\theta_{2})k(u,u \cdot (2L\theta_{1}),\mathfrak{q}(u) \cdot (4L^{2}\theta_{2})), \qquad u \neq 0, \, (\theta_{1},\theta_{2}) \in \R^{2}, \end{displaymath}
where $\mathfrak{q}$ is the "$\mathfrak{q}$uadratic" function
\begin{equation}\label{mathfrakq} \mathfrak{q}(u) := \begin{cases} u^{2}, & \text{if $k$ is horizontally odd}, \\ u|u|, & \text{if $k$ is odd}, \end{cases} \end{equation}
and $\chi \in C^{\infty}_{c}(\R^{2})$ satisfies
$\mathbf{1}_{[-1,1]^{2}} \leq \chi \leq \mathbf{1}_{[-2,2]^{2}}$.
Recalling \eqref{eq:form17B}, we may re-write the kernel
$K_{\Phi}$ as follows:
\begin{align} K_{\Phi} & (w,v) = k\left(w - v,\phi_1(w) - \phi_1(v),\phi_{2}(w) - \phi_{2}(v) +
\tfrac{\phi_{1}(v) + \phi_{1}(w)}{2}(w - v)
\right), \notag\\
&\label{form137a} = \kappa\left(w - v;\frac{\phi_{1}(w) -
\phi_{1}(v)}{(2L)(w - v)},\frac{\phi_{2}(w) - \phi_{2}(v) +
\tfrac{1}{2}[\phi_{1}(v) + \phi_{1}(w)](w -
v)}{(4L^{2})\mathfrak{q}(w - v)} \right).\end{align} To justify
passing to the line \eqref{form137a}, we need to recall that here
$\phi_1$ is $L$-Lipschitz by \eqref{form22}, and
$(\phi_{1},-\phi_2)\colon \R \to \R^2$ is a $2L^{2}$-tame function
by Proposition \ref{prop2}, so the terms
\begin{displaymath} \theta_{1}(w,v) := \frac{\phi_1(w) - \phi_1(v)}{(2L)(w - v)} \quad \text{and} \quad \theta_{2}(w,v) := \frac{\phi_{2}(w) - \phi_{2}(v)+\tfrac{1}{2}[\phi_{1}(v) + \phi_{1}(w)](w-v)}{4L^{2} \, \q(w - v)} \end{displaymath}
are bounded by $1$ in absolute value. This means that
\begin{displaymath} \chi(\theta_{1}(w,v),\theta_{2}(w,v)) \equiv 1, \qquad w,v \in \R. \end{displaymath}
Now, for $u \neq 0$ fixed, $(\theta_{1},\theta_{2}) \mapsto
\kappa(u;\theta_{1},\theta_{2})$ is a smooth function whose
support is a compact subset of $(-\pi,\pi) \times (-\pi,\pi)$, so
we may expand it as a Fourier series:
\begin{equation}\label{eq:fourier} \kappa(u;\theta_{1},\theta_{2}) = \sum_{\mathbf{n} \in \Z^{2}} \kappa_{\mathbf{n}}(u)e^{2\pi i (\theta_{1},\theta_{2}) \cdot \mathbf{n}},  \end{equation}
where
\begin{displaymath} \kappa_{\mathbf{n}}(u) = \int_{-\pi}^{\pi} \int_{-\pi}^{\pi} \kappa(u;\theta_{1},\theta_{2})e^{-2\pi i(\theta_{1},\theta_{2}) \cdot \mathbf{n}} \, d\theta_{1} \, d\theta_{2}. \end{displaymath}
Let  $c_{\mathbf{n}}$ be the best constant "$c$" in the
inequalities
\begin{displaymath} |\kappa_{\mathbf{n}}(u)| \leq \frac{c}{|u|} \quad \text{and} \quad |\kappa_{\mathbf{n}}'(u)| \leq \frac{c}{|u|^{2}}, \qquad u \in \R \, \setminus \, \{0\}. \end{displaymath}
Combining \eqref{form137a} and \eqref{eq:fourier}, we obtain
$K_{\Phi}(w,v) = $
\begin{align*} \sum_{\mathbf{n} \in \Z^{2}} & \kappa_{\mathbf{n}}(w - v) \exp \left(2\pi i \left[\tfrac{\phi_1(w) - \phi_1(v)}{(2L)(w - v)},\tfrac{\phi_{2}(w) - \phi_{2}(v)+\tfrac{1}{2}[\phi_{1}(v) + \phi_{1}(w)](w-v)}{4L^{2} \, \q(w - v)} \right] \cdot \mathbf{n} \right) \\
& =: \sum_{\mathbf{n} \in \mathbb{Z}^{2}} c_{\mathbf{n}} \cdot
K_{\mathbf{n}}(w,v), \end{align*} where $K_{\mathbf{n}}(w,v)$ is
the expression on the line above divided by $c_{\mathbf{n}}$. We
will verify below that the coefficients $c_{\mathbf{n}}$ are
finite, even exhibit rapid decay as $|\mathbf{n}| \to \infty$. We
record that
\begin{equation}\label{form138} K_{\Phi}D_{A_{0}} = \sum_{\mathbf{n} \in \Z^{2}} c_{\mathbf{n}} \cdot (K_{\mathbf{n}}D_{A_{0}}) \quad
\text{and} \quad K_{\Phi}D_{B_{0}} = \sum_{\mathbf{n} \in \Z^{2}}
c_{\mathbf{n}} \cdot (K_{\mathbf{n}}D_{B_{0}}). \end{equation} The
kernels $K_{\mathbf{n}}$ may not look better than $K_{\Phi}$, but
they are. They are products of the "convolution type" kernels
$c_{\mathbf{n}}^{-1} \cdot \kappa_{\mathbf{n}}$ on $\R$ with
$\|c_{\mathbf{n}}^{-1} \cdot \kappa_{\mathbf{n}}\|_{1,strong}
\lesssim 1$, and an "oscillating $L^{\infty}$-factor", which will
require further treatment.

\begin{lemma} The kernel $\kappa_{\mathbf{n}}$ is an odd SK on $\R$, and
\begin{equation}\label{eq:skDecay} c_{\mathbf{n}} \lesssim_{L,N} (1 + |\mathbf{n}|)^{-N}, \qquad N \in \N. \end{equation}
\end{lemma}
\begin{proof} This is where the weak goodness of the kernel $k$ gets used. The oddness of $\kappa_{\mathbf{n}}$ follows by noting that $\kappa(u;\theta_{1},\theta_{2})$ is an odd function of $u$:
\begin{displaymath} \kappa(-u; \theta_{1},\theta_{2}) = \chi(\theta_{1},\theta_{2})k(-u,-u \cdot (2L\theta_{1}),\mathfrak{q}(-u) \cdot (4L^{2}\theta_{2})) = -\kappa(u;\theta_{1},\theta_{2}), \end{displaymath}
using the assumption that $k$ is either horizontally odd, or odd,
and recalling from \eqref{mathfrakq} the definition of the
quadratic function $\mathfrak{q}$. The estimate \eqref{eq:skDecay}
follows by straightforward computation from the decay
$|\nabla_{\He}^{n}k(p)| \lesssim_{n} |z|^{-n - 1}$ assumed of the
weakly good kernel $k \in C^{\infty}(\He \, \setminus \,
\{z=0\})$, but let us give some details. For the case $\mathbf{n}
= 0$, one checks from the definition of $\kappa$ that
$|\kappa(u;\theta_{1},\theta_{2})| \lesssim |u|^{-1}$ (hence
$|\kappa_{0}(u)| \lesssim |u|^{-1}$), and
\begin{align*} |\partial_{u} \kappa(u;\theta_{1},\theta_{2})| & \leq |\nabla k(u,u \cdot (2L\theta_{1}),\mathfrak{q}(u) \cdot (4L^{2}\theta_{2}))| \cdot |(1,2L\theta_{1},u \cdot 8L^{2}\theta_{2})|\\
& \lesssim_{L} |\partial_{x} k(u,u \cdot (2L\theta_{1}),\mathfrak{q}(u) \cdot (4L^{2}\theta_{2}))|\\
&\qquad + |\partial_{y} k(u,u \cdot (2L\theta_{1}),\mathfrak{q}(u) \cdot (4L^{2}\theta_{2}))|\\
&\qquad + |u||\partial_{t} k(u,u \cdot
(2L\theta_{1}),\mathfrak{q}(u) \cdot (4L^{2}\theta_{2}))|, \quad u
\neq 0, \, (\theta_{1},\theta_{2}) \in \spt \chi. \end{align*}
Using the relations $\partial_{x} = X + \tfrac{y}{2}\partial_{t}$,
$\partial_{y} = Y - \tfrac{x}{2}\partial_{t}$, and $\partial_{t} =
XY - YX$, one may infer from the weak goodness of $k$ that
$|\partial_{u}\kappa(u;\theta_{1},\theta_{2})| \lesssim_{L}
|u|^{-2}$ for all $(\theta_{1},\theta_{2}) \in \spt \chi$ and $u
\neq 0$.

In general, fix $0 \neq \mathbf{n} = (n_{1},n_{2}) \in \Z^{2}$, $N
\in \N$, and assume first that $|n_{2}| > |n_{1}|$. Then, for $u
\neq 0$ and $\theta_{1} \in [-\pi,\pi]$ fixed, integrating by
parts $N$ times, and recalling that $\theta_{2} \mapsto
\kappa(u;\theta_{1},\theta_{2})$ is compactly supported in
$(-\pi,\pi)$, we find
\begin{equation}\label{form140} \int_{-\pi}^{\pi} \kappa(u;\theta_{1},\theta_{2})e^{-2\pi i \theta_{2}n_{2}} \, d\theta_{2} = \frac{1}{(2\pi i n_{2})^{N}} \int_{-\pi}^{\pi} \partial_{\theta_{2}}^{N}[\theta_{2} \mapsto \kappa(u;\theta_{1},\theta_{2})] e^{-2\pi i \theta_{2}n_{2}} \, d\theta_{2}. \end{equation}
To estimate the right hand side of \eqref{form140}, put absolute
values inside, recall $\kappa(u;\theta_{1},\theta_{2}) =
\chi(\theta_{1},\theta_{2})k(u,u \cdot
(2L\theta_{1}),\mathfrak{q}(u) \cdot (4L^{2}\theta_{2}))$, use the
Leibniz rule, and observe the estimate
\begin{align*} |\partial_{\theta_{2}}^{N} & [\theta_{2} \mapsto k(u,u \cdot (2L\theta_{1}),\mathfrak{q}(u) \cdot (4L^{2}\theta_{2}))]|\\
& = |(4L^{2})\mathfrak{q}(u)|^{N} |(\partial_{t}^{N}k)(u,u \cdot
(2L\theta_{1}),\mathfrak{q}(u) \cdot (4L^{2}\theta_{2}))|
\lesssim_{N} \frac{|(4L)^{2}\mathfrak{q}(u)|^{N}}{|u|^{2N + 1}} =
\frac{(2L)^{2N}}{|u|}. \end{align*} Here we used that
$|\mathfrak{q}(u)| = |u|^{2}$, $|(u,u \cdot (2L\theta_{1}))| \geq
|u|$, and that $\partial_{t}^{N} = [X,Y]^{N}$ is a horizontal
derivative of order $2N$. It now follows that
\begin{displaymath} |\kappa_{\mathbf{n}}(u)| \leq \int_{-\pi}^{\pi} \left| \int_{-\pi}^{\pi} \kappa(u;\theta_{1},\theta_{2})e^{-2\pi i \theta_{2}n_{2}} \, d\theta_{2} \right| \, d\theta_{1} \lesssim_{L,N} \frac{1}{|n_{2}|^{N}|u|} \lesssim \frac{1}{|\mathbf{n}|^{N}|u|}. \end{displaymath}
A similar, but rather tedious, computation yields
$|\kappa_{\mathbf{n}}'(u)| \lesssim_{L,N}
|u|^{-2}|\mathbf{n}|^{-N}$.

Finally, if $|n_1|>|n_2|$, once considers
$\partial_{\theta_{1}}^{N}$ instead of the partial derivative with
respect to $\theta_2$, and argues similarly as before, observing
that
\begin{displaymath}
\partial_y^N = \left[Y-\tfrac{x}{2}\partial_t\right]^N =
\sum_{m=0}^N \binom{N}{m} Y^m
\left(-\tfrac{x}{2}\right)^{N-m}[XY-YX]^{N-m}.
\end{displaymath}
So, we have verified \eqref{eq:skDecay}.\end{proof}

Due to the rapid decay of the coefficients $c_{\mathbf{n}}$, and
the decomposition \eqref{form138}, Theorem \ref{t:technical} will
follow once we show that
\begin{displaymath} \|K_{\mathbf{n}}D_{A_{0}}\|_{\mathrm{C.Z.},1} = O(\mathrm{poly}(|\mathbf{n}|)) \quad \text{and} \quad  \|K_{\mathbf{n}}D_{B_{0}}\|_{\mathrm{C.Z.},1} =
O(\mathrm{poly}(|\mathbf{n}|)). \end{displaymath} This is the
content of the next proposition, whose proof will combine
techniques developed by Christ \cite{MR1104656}, David
\cite{MR956767}, Hofmann \cite{MR1484857}, and
Semmes~\cite{Semmes}:

\begin{thm}\label{mainProp}
\label{polyBound} There exists a constant $C \geq 1$ such that the
following holds. Let $M,N \geq 1$. Let $A \colon \R \to \R$ be
$M$-Lipschitz, let $B \colon \R \to \R^{2}$ be $N$-tame, let $\q
\colon \R \to \R$ be one of the two functions $\q(s) := s^{2}$ or
$\q(s) := |s|s$, and let $\kappa \in C^{1}(\R \, \setminus \,
\{0\})$ be an odd function satisfying $|\partial^{j} \kappa(u)|
\leq |u|^{-1 - j}$ for $j \in \{0,1\}$. Then, the kernel
\begin{equation}\label{form22kernel} K_{A,B}(w,v) := \kappa(w - v) \exp \left( 2\pi i \left[ \tfrac{A(w) - A(v)}{w - v} + \tfrac{B_{2}(w) - B_{2}(v) - \tfrac{1}{2}[B_{1}(v) + B_{1}(w)](w-v)}{\q(w - v)}\right] \right) \end{equation}
is a CZ kernel for $\mathcal{L}^1$ with
\begin{equation}\label{form151} \|K_{A,B}\|_{\mathrm{C.Z.},1} \leq C \max\{M,N\}^{C}. \end{equation}
The same remains true of the kernels $K_{A,B}D_{A_{0}}$ and
$K_{A,B}D_{B_{0}}$, but then the multiplicative constant ("the first $C$") in
\eqref{form151} may also depend on the Lipschitz and tameness
constants of $A_{0},B_{0}$.
\end{thm}

Theorem \ref{mainProp} will be proven in Section \ref{s:exp2}, in
more general form, see Theorem~\ref{polyBound}. The letter
"$\kappa$" will from now on refer to an odd SK on $\R$, such as in
Theorem \ref{mainProp}, and the auxiliary function
"$\kappa(u;\theta_{1},\theta_{2})$" will not be seen again.

\begin{proof}[Proof of Theorem \ref{t:technical} assuming Theorem \ref{mainProp}] From \eqref{form138}, we infer that
\begin{displaymath} \|K_{\Phi}D_{A_{0}}\|_{\mathrm{C.Z.},1} \leq \sum_{\mathbf{n} \in \Z^{2}} c_{\mathbf{n}} \cdot \|K_{\mathbf{n}}D_{A_{0}}\|_{\mathrm{C.Z.},1} \quad \text{and} \quad \|K_{\Phi}D_{B_{0}}\|_{\mathrm{C.Z.},1} \leq \sum_{\mathbf{n} \in \Z^{2}} c_{\mathbf{n}} \cdot \|K_{\mathbf{n}}D_{B_{0}}\|_{\mathrm{C.Z.},1}.
\end{displaymath}
To avoid repetition and long display lines, we restrict attention
to the case $K_{\Phi}D_{A_{0}}$, and we even assume that
$D_{A_{0}} \equiv 1$. By the rapid decay of the coefficients
$c_{\mathbf{n}}$, it suffices to show that there exists a constant
$C\geq 1$ such that, for every $\mathbf{n}=(n_1,n_2) \in \Z^{2}$,
the kernel
\begin{displaymath}
 K_{\mathbf{n}}(w,v) = \frac{\kappa_{\mathbf{n}}(w - v)}{c_{\mathbf{n}}} \exp \left(2\pi i \left[\tfrac{\phi_1(w) - \phi_1(v)}{(2L)(w - v)},\tfrac{\phi_{2}(w) - \phi_{2}(v)+\tfrac{1}{2}[\phi_{1}(v) + \phi_{1}(w)](w-v)}{4L^{2} \, \q(w - v)} \right] \cdot \mathbf{n} \right)
\end{displaymath}
is a CZ kernel for $\mathcal{L}^1$ with
\begin{displaymath}
\| K_{\mathbf{n}}\|_{\mathrm{C.Z.},1} \leq C(1 +
|\mathbf{n}|)^{C}.
\end{displaymath}
This follows from Theorem \ref{mainProp} applied with $\kappa =
\kappa_{\mathbf{n}}/c_{\mathbf{n}}$,
\begin{displaymath}
A:= \frac{n_1}{2 L}\phi_1\quad\text{and}\quad B:= \frac{n_2}{4
L^2}(\phi_1,-\phi_2),
\end{displaymath}
observing that $A$ is $\frac{n_1}{2}$-Lipschitz, and $B$ is
$\tfrac{n_2}{2}$-tame by the comment below \eqref{form137a}. This
completes the proof of Theorem \ref{t:technical}. \end{proof}

\subsection{Calder\'on commutators appear}

Let $A \colon \R \to \R$ be Lipschitz, let $B \colon \R \to
\R^{2}$ be tame, let $\kappa \in C^{1}(\R \, \setminus \, \{0\})$
be an odd function satisfying $|\partial^{j} \kappa(u)| \leq
|u|^{-1 - j}$ for $j \in \{0,1\}$, and consider the SK
\begin{displaymath} K_{A,B}(x,y) := \kappa(x - y)\exp\left(2\pi i\left[ \tfrac{A(x) - A(y)}{x - y} + \tfrac{B_{2}(x) - B_{2}(y) - \tfrac{1}{2}[B_{1}(x) + B_{1}(y)](x - y)}{\q(x - y)} \right] \right), \end{displaymath}
familiar from Example \ref{ex1}.
\begin{thm}\label{TABCZO} Let $A \colon \R \to \R$ be a $1$-Lipschitz function, and let $B \colon \R \to \R^{2}$ be a $1$-tame map. Then $\|K_{A,B}\|_{\mathrm{C.Z.},1} \leq C$ for some absolute constant $C \geq 1$. The same remains true for $K_{A,B}D_{A_{0}}$ and $K_{A,B}D_{B_{0}}$, allowing $C$ also to depend on the Lipschitz and tameness constants of $A_{0},B_{0}$. \end{thm}
Theorem \ref{mainProp} does not immediately, or even easily,
follow from Theorem \ref{TABCZO}, because we are interested in the
polynomial dependence on $M$ and $N$. The sharper result will be
derived "by induction" in Section \ref{s:exp2}, and the main
result of this section, stated above, will cover the base case of
that induction.

We will show the CZ properties of $K_{A,B}D_{A_{0}}$ and
$K_{A,B}D_{B_{0}}$ by decomposing the kernel into a sum of (even)
simpler ones, resembling \emph{Calder\'on commutators}, then
proving separately that they are CZ kernels, and finally summing
up the results. In fact, using that $e^{2\pi i x} = \sum_{n \geq
0} (2\pi ix)^{n}/n!$, we first write
\begin{displaymath} K_{A,B}(x,y) = \sum_{n \geq 0}\frac{(2\pi i)^{n}}{n!} S_{n}(x,y), \end{displaymath}
where
\begin{equation}\label{form98} S_{n}(x,y) := \kappa(x - y)\left[ \tfrac{A(x) - A(y)}{x - y} + \tfrac{B_{2}(x) - B_{2}(y) - \tfrac{1}{2}[B_{1}(x) + B_{1}(y)](x - y)}{\q(x - y)} \right]^{n}. \end{equation}
Then, the terms $S_{n}$ are further decomposed as follows:
\begin{displaymath} S_{n}(x,y) = \sum_{m = 0}^{n} \binom{n}{m} \kappa(x - y) \left[\tfrac{A(x) - A(y)}{x - y} \right]^{m}\left[\tfrac{B_{2}(x) - B_{2}(y) - \tfrac{1}{2}[B_{1}(x) + B_{1}(y)](x - y)}{\q(x - y)}\right]^{n - m}. \end{displaymath}
Motivated by this decomposition, we define the standard kernels
\begin{equation}\label{eq:Cmn} C_{m,n}(x,y) := \kappa(x - y) \left[\tfrac{A(x) - A(y)}{x - y} \right]^{m}\left[\tfrac{B_{2}(x) - B_{2}(y) - \tfrac{1}{2}[B_{1}(x) + B_{1}(y)](x - y)}{\q(x - y)}\right]^{n}.\end{equation}
Note that the definition also depends on $\kappa$, and the choice
of the function $\q$ (determined  by the good kernel $k$), but we
suppress these dependencies from the notation.

\begin{ex}\label{ex:CSK} It is easy to check that if $K \colon \R \times \R \, \setminus \, \bigtriangleup \to \C$ is an SK, $A \colon \R \to \R$ is $M$-Lipschitz, and $B = (B_{1},B_{2}) \colon \R \to \R^{2}$ is $N$-tame, then both
\begin{displaymath} K_{A}(x,y) = K(x,y)\left[\frac{A(x) - A(y)}{x - y} \right] \end{displaymath}
and
\begin{displaymath} K_{B}(x,y) = K(x,y)\left[\frac{B_{2}(x) - B_{2}(y) - \tfrac{1}{2}[B_{1}(x) + B_{1}(y)](x - y)}{\q(x - y)}\right] \end{displaymath}
are SKs with
\begin{displaymath} \|K_{A}\|_{\alpha,strong} \lesssim (1 + M)\|K\|_{\alpha,strong} \quad \text{and} \quad \|K_{B}\|_{\alpha,strong} \lesssim (1 + N)\|K\|_{\alpha,strong}. \end{displaymath}
For the second inequality, use the expansion \eqref{form139},
which reduces matters to the Lipschitz constant of $B_{1}$ (i.e.
$N$). It follows, by iteration, that if $A$ is $1$-Lipschitz and
$B$ is $1$-tame, the kernel $C_{m,n}$ satisfies
$\|C_{m,n}\|_{1,strong} \leq C^{m + n + 1}$ for some absolute
constant $C \geq 1$. With the same argument, also
$\|C_{m,n}D_{A_{0}}\|_{1,strong} \lesssim_{A_{0}} C^{m + n + 1}$
and $\|C_{m,n}D_{B_{0}}\|_{1,strong} \lesssim_{B_{0}} C^{m + n +
1}$
\end{ex}

The proof of the following theorem will occupy most of this
section.

\begin{thm}\label{commutatorTheorem} Let $A \colon \R \to \R$ be $1$-Lipschitz, let $B = (B_{1},B_{2}) \colon \R \to \R^{2}$ be $1$-tame, and let $m,n \geq 0$. Then $\|C_{m,n}\|_{\mathrm{C.Z.},1} \leq C^{m + n + 1}$, where $C \geq 1$ is an absolute constant. Up to a multiplicative constant, the same remains true for the kernels $C_{m,n}D_{A_{0}}$ and $C_{m,n}D_{B_{0}}$.
\end{thm}

It follows immediately from Theorem \ref{commutatorTheorem} that
$S_{n}$ is a also a CZ kernel with
\begin{displaymath} \|S_{n}\|_{\mathrm{C.Z.},1} \leq C^{n + 1} \sum_{m = 0}^{n} \binom{n}{m} \leq (2C)^{n + 1}, \end{displaymath}
and finally that $K_{A,B}$ is a CZ kernel with
\begin{displaymath} \|K_{A,B}\|_{\mathrm{C.Z.},1} \lesssim \sum_{n \geq 0} \frac{(2\pi)^n}{n!} \|S_{n}\|_{\mathrm{C.Z.}} < \infty. \end{displaymath}
The generalisations $K_{A,B}D_{A_{0}}$ and $K_{A,B}D_{B_{0}}$ can
be handled similarly. So, Theorem \ref{TABCZO} follows from
Theorem \ref{commutatorTheorem}. We start with a few preparations
to prove the latter.
\subsection{Reminder on $\beta$-numbers} Let $A \colon \R \to \R$ be a Lipschitz function. For $x \in \R$, $s > 0$, we define
\begin{equation}\label{betaNumber} \beta_{A}(B(x,s)) := \inf_{a,b \in \R} \sup \left\{ \frac{|A(y) - [ay + b]|}{s} : y \in B(x,s)\right\}. \end{equation}
The $\beta$-numbers satisfy the following Carleson packing
condition:
\begin{equation}\label{jones} \int_{0}^{r} \frac{1}{r}\int_{B(x,r)} \beta_{A}(B(y,s))^{2} \, dy \, \frac{ds}{s} \lesssim \mathrm{Lip}(A)^{2}, \qquad x \in \R, \; r > 0. \end{equation}
This is a special case of Jones' traveling salesman theorem
\cite{MR1069238}, but the case for Lipschitz graphs in $\R^{2}$ is
much simpler, see the book of Garnett-Marshall, \cite[Chapter X,
Lemma 2.4]{MR2450237}. The quadratic dependence on
$\mathrm{Lip}(A)$ follows from the $\mathrm{Lip}(A) = 1$ case by
scaling (noting that $\beta_{c A}(B(x,s)) = c\beta_{A}(B(x,s))$).
The following standard lemma shows that the $\beta$-number in
\eqref{betaNumber} also controls deviations from affine maps
defined via averaging the gradient:

\begin{lemma}\label{lemma2} Let $\varphi \in C^{\infty}(\R)$ be a standard bump function:
\begin{equation}\label{standardBump} \int \varphi = 1, \quad \varphi \geq 0 \text{ and } \spt \varphi \subset B(0,1), \quad \text{and} \quad \varphi(-z) \equiv \varphi(z). \end{equation}
For $s > 0$, let $\varphi_{s}(x) := s^{-1} \cdot \varphi(x/s)$.
For a Lipschitz function $A \colon \R \to \R$, $x \in \R$, and $s
> 0$, define the linear map
\begin{displaymath} y \mapsto L_{x,s}(y) := P_{s}(A')(x)y, \end{displaymath}
where $P_{s}(A')(x) := (A' \ast \varphi_{s})(x)$. Then,
\begin{displaymath} \frac{|A(x) - A(y) - L_{x,s}(x - y)|}{s} \lesssim_{\varphi} \beta_{A}(B(x,s)), \qquad y \in B(x,s). \end{displaymath}
\end{lemma}

\begin{proof} To simplify notation, assume, without loss of generality, that $x = 0=A(x)$. Let $y \mapsto ay + b$ be the best approximating affine map associated to the number $\beta_{A}(B(0,s))$, that is,
\begin{displaymath} |A(y) - (ay + b)| \leq s \cdot \beta_{A}(B(0,s)), \qquad y \in B(0,s). \end{displaymath}
Applying this with $y = 0$ gives $|b| \leq s \cdot
\beta_{A}(B(0,s))$. Further,
\begin{align} |A(y) - L_{0,s}(y)| & \leq |A(y) - (ay + b)| + |b| + |ay - L_{0,s}(y)| \notag \\
&\label{form97} \leq 2s \cdot \beta_{A}(B(0,s)) + s \cdot \left|
\int \varphi_{s}(z)[A'(z) - a] \, dz \right|. \end{align} To treat
the last term, integrate by parts:
\begin{align*} \left| \int \varphi_{s}(z)[A'(z) - a] \, dz \right| & \leq \int |\dot{\varphi}_{s}(z)|| A(z) - (az + b)| \, dz\\
& \lesssim \frac{1}{s^{2}} \int_{B(0,s)} s \cdot \beta_{A}(B(0,s))
\, dz = 2\beta_{A}(B(0,s)). \end{align*} Plugging this last
estimate to \eqref{form97} completes the proof. \end{proof}

\subsection{Boundedness of the Calder\'on commutators} In this section, we prove Theorem \ref{commutatorTheorem}. To avoid a notational mess, we carry out the proof in full detail for the kernels $C_{m,n}$, and then comment on the small addenda regarding the kernels $C_{m,n}D_{A_{0}}$ and $C_{m,n}D_{B_{0}}$ afterwards, in Remark \ref{r:mod}. To a large extent, the proof of Theorem \ref{commutatorTheorem} uses arguments in
\cite{MR1104656} and \cite{MR1484857}, but the details look a
little different, and the inhomogeneity of $\kappa$ causes mild trouble, so we do not attempt to cut corners. Fix a
$1$-Lipschitz function $A \colon \R \to \R$, a $1$-tame map $B =
(B_{1},B_{2}) \colon \R \to \R^{2}$, and $m,n \geq 0$. We
abbreviate $C_{m,n}(x,y) := $
\begin{displaymath} K(x,y) := \kappa(x - y) \left[\frac{A(x) - A(y)}{x - y} \right]^{m} \left[ \frac{B_{2}(x) - B_{2}(y) - \tfrac{1}{2}[B_{1}(x) + B_{1}(y)](x - y)}{\q(x - y)} \right]^{n}. \end{displaymath}
Recall, once more, that $\q \colon \R \to \R$ is one of the
functions $\q(s) = s^{2}$ or $\q(s) = |s|s$, and $\kappa \in
C^{1}(\R \, \setminus \, \{0\})$ is an odd function with
$|\partial^{j}\kappa(u)| \leq |u|^{-1 - j}$ for $j \in \{0,1\}$
and $u \neq 0$. The case $n = 0$ is the case of "standard"
Calder\'on commutators of index $m$ associated to the kernel
$\kappa$, and it is well-known that $\|C_{m,0}\|_{\mathrm{C.Z.},1}
\lesssim C^{m + 1}$, see for example \cite[p. 53]{MR1123480}. So,
we only consider the case $n\geq 1$ in the following.

\begin{remark}\label{symmetry} For $n\geq 1$, the kernel $K$ looks a little like the kernel of the standard Calder\'on commutator, but there is a difference worth pointing out (besides the obvious one that $\kappa(u)$ is not necessarily $1/u$). Consider the case $m = 0$ and $\q(s) = s^{2}$. Then,
\begin{equation}\label{form106} K(y,x) = \kappa(y - x) \left[ \frac{B_{2}(y) - B_{2}(x) - \tfrac{1}{2}[B_{1}(y) + B_{1}(x)](y - x)}{(y - x)^{2}} \right]^{n} = (-1)^{n + 1}K(x,y), \end{equation}
so $K$ is antisymmetric only when $n$ is even. The kernels of
standard Calder\'on commutators (i.e. the kernels $K$ above with
$n = 0$, $m \geq 0$, $\kappa(u) = 1/u$) are always antisymmetric.
\end{remark}

\begin{proof}[Proof of Theorem \ref{commutatorTheorem} for $n\geq 1$]  We plan to verify the $T1$ testing conditions \eqref{universalT1}. In principle, this means that we need to consider smooth truncations of the form $K_{\epsilon}(x,y) = \psi_{\epsilon}(x - y)K(x,y)$. But since $\psi_{\epsilon} \cdot \kappa$
is a kernel of the same type as $\kappa$ (with constants
independent of $\epsilon$), we may simply absorb $\psi_{\epsilon}$
to $\kappa$ and assume that the kernel $\kappa$ is supported away
from the origin to begin with: $\kappa|_{[-\epsilon,\epsilon]}
\equiv 0$ for some $\epsilon > 0$.

To verify the testing conditions \eqref{universalT1}, let $B_{0}
=B(x_0,R)\subset \R$ be a ball, and let $b \in C^{\infty}(\R)$
with $\mathbf{1}_{2B_{0}} \leq b \leq \mathbf{1}_{3B_{0}}$. After
performing the changes of variables $x \mapsto Rx'$ and $y \mapsto
Ry'$, using Lemma \ref{lemma1}, and noting that $u \mapsto
R\kappa(R u)$ is a kernel of the same type as $\kappa$, we may
reduce to the case $R = 1$. Then, pre-composing $A,B$ with a
translation, we may also take $x_{0} = 0$. So, we claim that
whenever $b \in C^{\infty}(\R)$ with $\mathbf{1}_{B(0,2)} \leq b
\leq \mathbf{1}_{B(0,3)}$, then
\begin{equation}\label{T13} \int_{B(0,1)} |T(b)| \leq C(m + 1) \quad \text{and} \quad \int_{B(0,1)} |T^{t}(b)| \leq C(m + 1). \end{equation}
It is not a typo that the right hand sides do not depend on $n$;
the reason is clear after Section \ref{s:casen2}. The kernel of
the adjoint $T^{t}$ is $K^{t}(x,y) = K(y,x)
= (-1)^{n+1} K(x,y)$ by \eqref{form106}, so it suffices
to prove the first estimate in \eqref{T13}. At this point, we
already observe that, in proving \eqref{T13}, we may assume that
the function $B_{1}$ appearing in the kernel of $T$ satisfies
\begin{equation}\label{sptA} B_{1}(0) = 0 \quad \text{and} \quad \spt B_{1} \subset B(0,10). \end{equation}
In fact, the value of the kernel $K(x,y)$ remains
unchanged if replace $B$ by $B - \mathcal{L}$, where
$\mathcal{L}(x) = (B_{1}(0),B_{1}(0)x)$ is a $0$-tame-affine map.
Next, already using that $B_{1}(0) = 0$, it is easy to show that
there exists a $1$-Lipschitz function $\tilde{B}_{1}$ with $\spt
\tilde{B}_{1} \subset B(0,10)$ which agrees with $B_{1}$ on
$B(0,3)$. Since only the values of $B_{1}$ on $B(0,3)$ appear in
\eqref{T13}, we may replace $B_{1}$ by $\tilde{B}_{1}$ without
changing the value of \eqref{T13}. We will only use the tameness
condition $\dot{B}_{2}(z) = B_{1}(z)$ for $z \in B(0,3)$ (see
\eqref{form83}), and this now remains valid with $\tilde{B}_{1}$
instead. Alternatively, we could redefine $B_{2}$ on $\R$ so that
$\dot{B}_{2} = \tilde{B_{1}}$ on $\R$, and hence acquire a new
$1$-tame function $\tilde{B} \colon \R \to \R^{2}$ satisfying
\eqref{sptA}, but this is a little overkill.

To prove \eqref{T13}, we start roughly as in the proof of
\cite[Theorem 10, p. 58]{MR1104656}, and fix an auxiliary function
$\eta \in C^{\infty}(\R)$ satisfying
\begin{equation}\label{choiceOfVarphi} \spt \eta\subset [\tfrac{1}{4},1] \quad \text{and}\quad \int_{0}^{\infty} \eta(s) \, \frac{ds}{s} = 1.
\end{equation}
 Then, for $x,y \in \R$ with $x \neq y$ fixed, we note that
\begin{displaymath} \int_{0}^{\infty} \eta \left(\frac{|x - y|}{s} \right) \, \frac{ds}{s} \stackrel{s \mapsto r^{-1}|x - y|}{=}  \int_{0}^{\infty} \eta(r) \, \frac{dr}{r} = 1. \end{displaymath}
In particular, for $x \in B(0,1)$ (as in \eqref{T13}) fixed, we
may write
\begin{align*} T(b)(x) & = \int K(x,y)b(y) \left[\int_{0}^{\infty} \eta\left(\frac{|x - y|}{s} \right) \, \frac{ds}{s} \right] \, dy\\
& = \int_{0}^{\infty} \int \eta\left(\frac{|x - y|}{s} \right)
K(x,y)b(y) \, dy \, \frac{ds}{s}.  \end{align*} Let us
point out that the integrals above are absolutely convergent,
because, first, a necessary condition for $\eta(|x -
y|/s)K(x,y) \neq 0$ is $\epsilon/2 \leq |x - y| < s$,
so the integral over $s \leq \epsilon/2$ contributes zero. Second, if $s > 16$, then $s^{-1}|x - y| < \tfrac{1}{4}$ for
all pairs $x \in B(0,1)$ and $y \in \spt b \subset B(0,3)$, so the
integral over $s > 16$ also contributes zero. Also, the
integration over $s \in (1,16)$ only
yields an absolute constant, so we have reduced \eqref{T13} to
showing
\begin{equation}\label{T14} \int_{B(0,1)} \left| \int_{0}^{1} \int \eta\left(\frac{|x - y|}{s} \right) K(x,y)b(y) \, dy \, \frac{ds}{s} \right| \, dx \leq C(m + 1).  \end{equation}
Finally, since $\spt [1 - b] \subset \mathbb{R} \, \setminus \, B(0,2)$,
we have $|x - y| \geq 1$ for all $x \in B(0,1)$ and $y \in \spt [1
- b]$. Consequently $\eta(|x - y|/s) = 0$ whenever $s \in (0,1]$,
$x \in B(0,1)$, and $y \in \spt [1 - b]$, and it follows that
\begin{displaymath} \int_{B(0,1)} \left|\int_{0}^{1} \int \eta \left(\frac{|x - y|}{s} \right) K(x,y)[1 - b](y) \, dy \, \frac{ds}{s}\right|\, dx = 0. \end{displaymath}
Therefore, \eqref{T14} reduces further to proving that
\begin{equation}\label{T15} \int_{B(0,1)} \left| \int_{0}^{1} \int \eta\left(\frac{|x - y|}{s} \right) K(x,y) \, dy \, \frac{ds}{s} \right| \, dx \leq C(m + 1). \end{equation}
To prove \eqref{T15}, fix $x \in B(0,1)$. Recall the exponents
$m,n \geq 0$ from the definition of the kernel $K$, and remember
that we only consider $n\geq 1$, as the case $n=0$ corresponds to
the  "standard" Calder\'on commutators already treated in the
literature. The case $n \geq 2$ turns out to be easy, see the
Section \ref{s:casen2}, so the case $n = 1$ contains the main
news.
\subsection{The case $n \geq 2$}\label{s:casen2} In this case, we make the following rather crude estimate for \eqref{T15}:
\begin{displaymath} \eqref{T15} \lesssim \int_{B(0,1)} \int_{0}^{1} \frac{1}{s} \int_{\{y : \tfrac{s}{4} \leq |x - y| \leq s\}} \left| \frac{B_{2}(x) - B_{2}(y) - \tfrac{1}{2}[B_{1}(x) + B_{1}(y)](x - y)}{\q(x - y)} \right|^{2} \, dy \, \frac{ds}{s} \, dx \end{displaymath}
To proceed, we first use the tameness condition $\dot{B}_{2} =
B_{1}$ to write
\begin{equation}\label{form83} \frac{B_{2}(x) - B_{2}(y) - \tfrac{1}{2}[B_{1}(x) + B_{1}(y)](x - y)}{\q(x - y)} = \int_{x}^{y} \frac{B_{1}(x) + B_{1}(y) - 2B_{1}(r)}{2\q(x - y)} \, dr. \end{equation}
It is easy to check that the right hand side on \eqref{form83}
vanishes if $B_{1}$ is affine. In particular,
\begin{align} \Big| \int_{x}^{y} & \frac{B_{1}(x) + B_{1}(y) - 2B_{1}(r)}{2\q(x - y)} \, dr \Big| \notag\\
&\label{form75} \quad \leq \fint_{x}^{y} \frac{|B_{1}(x) -
B_{x,s}(x)| + |B_{1}(y) - B_{x,s}(y)| + 2|B_{1}(r) -
B_{x,s}(r)|}{2|x - y|} \, ds, \end{align} where $B_{x,s}(y) = ay +
b$ is an affine map minimising the $\beta$-number (introduced in
\eqref{betaNumber}) of $B_{1}$  in $B(x,s)$. Therefore, we have
\begin{equation}\label{form102} \left| \frac{B_{2}(x) - B_{2}(y) - \tfrac{1}{2}[B_{1}(x) + B_{1}(y)](x - y)}{\q(x - y)} \right| \lesssim \beta_{B_{1}}(B(x,s)) \end{equation}
for $x \in B(0,1)$ and $\tfrac{s}{4} \leq |x - y| \leq s$, and
consequently

\begin{equation}\label{form78} \eqref{T15} \lesssim \mathcal{B}^{2} := \int_{B(0,1)} \int_{0}^{1} \beta_{B_{1}}(B(x,s))^{2} \, \frac{ds}{s} \, dx \lesssim 1. \end{equation}
by Jones' estimate \eqref{jones}.
\subsection{The case $n =1$}\label{s:nIs1} We then consider the case $n =1$ and $m \geq 0$. We write
\begin{equation}\label{form77}K_{m}(x,y) :=  \kappa(x - y) \left[\tfrac{A(x) - A(y)}{x - y} \right]^{m} \tfrac{B_{2}(x) - B_{2}(y) - \tfrac{1}{2}[B_{1}(x) + B_{1}(y)](x - y)}{\q(x - y)} . \end{equation}
Let $\varphi \in C^{\infty}(\R)$ be a "standard bump function" as
in \eqref{standardBump}. Then, as in Lemma \ref{lemma2}, we
consider the linear maps
\begin{equation}\label{form141} L_{x,s}(y) := (A' \ast \varphi_{s})(x)y =: P_{s}(A')(x)y, \qquad s \in (0,1]. \end{equation}
The plan is to reduce the treatment of the kernel \eqref{form77}
to the case $m = 0$. To accomplish this, assume that initially $m
\geq 1$. Then, for $x \in B(0,1)$ and $s \in (0,1)$ fixed, we
write
\begin{align}\label{form76} \left[\frac{A(x) - A(y)}{x - y} \right]^{m} & = \left[\frac{A(x) - A(y)}{x - y} \right]^{m - 1} \left[\frac{A(x) - A(y) - L_{x,s}(x - y)}{x - y} \right]\\
& \qquad + \left[\frac{A(x) - A(y)}{x - y} \right]^{m - 1}
P_{s}(A')(x). \notag \end{align} Here, for $y\in B(x,s)\setminus
\{x\}$,
\begin{equation}\label{form103} \left| \left[\frac{A(x) - A(y)}{x - y} \right]^{m - 1} \left[\frac{A(x) - A(y) - L_{x,s}(x - y)}{x - y} \right] \right| \lesssim \beta_{A}(B(x,s)) \end{equation}
by Lemma \ref{lemma2}. We plug this information into \eqref{T15},
and use the triangle inequality, to obtain two terms
$\eqref{T15}_{1}$ and $\eqref{T15}_{2}$. For $\eqref{T15}_{1}$, we
combine \eqref{form102} and \eqref{form103} to infer that
\begin{equation}\label{form104} \eqref{T15}_{1} \lesssim \int_{B(0,1)} \int_{0}^{1} \beta_{A}(B(x,s)) \beta_{B_{1}}(B(x,s)) \, \frac{ds}{s} \, dx \lesssim 1, \end{equation}
by Cauchy-Schwarz, and Jones' estimate \eqref{jones}. Let us then
consider
\begin{equation}\label{form80} \eqref{T15}_{2} = \int_{B(0,1)} \left| \int_{0}^{1} P_{s}(A')(x) \int \eta\left(\frac{|x - y|}{s} \right) K_{m - 1}(x,y) \, dy \, \frac{ds}{s} \right| \, dx. \end{equation}
If still $m - 1 \geq 0$, we repeat the same procedure as in
\eqref{form76}, separating one power of $(A(x) - A(y))/(x - y)$
from $K_{m - 1}$, adding and subtracting $L_{x,s}(x - y)$, and
then repeating the estimates \eqref{form103}-\eqref{form104}. This
operation yields two terms, one "error" term dominated, as before,
by $\lesssim 1$ (also using that $\|P_{s}(A')\|_{L^{\infty}} \leq
1$), and then the "main" term
\begin{equation}\label{form79} \int_{B(0,1)} \left| \int_{0}^{1} P_{s}(A')(x)^{2} \int \eta\left(\frac{|x - y|}{s} \right) K_{m - 2}(x,y) \, dy \, \frac{ds}{s} \right| \, dx. \end{equation}
Comparing \eqref{form80} and \eqref{form79}, we note that if $j
\geq 1$, we can reduce the study of $K_{j}$ to the study of $K_{j
- 1}$ at the cost of
\begin{enumerate}
\item committing an additive error of magnitude $\lesssim 1$, and
\item replacing $P_{s}(A')(x)^{j}$ by $P_{s}(A')(x)^{j + 1}$ in
\eqref{form79}.
\end{enumerate}
After repeating these steps $m$ times, we see that \eqref{T15} is
bounded by $\lesssim m$ plus
\begin{align}\label{form81} \int_{B(0,1)} \Bigg| \int_{0}^{1} P_{s}(A')(x)^{m} \int \eta\left(\frac{|x - y|}{s} \right)
\kappa(x - y)\left( \int_{x}^{y} \frac{B_{1}(x) + B_{1}(y) -
2B_{1}(r)}{2\q(x - y)} \, \, dr \right)  \, dy \, \frac{ds}{s}
\Bigg| \, dx.\end{align} Here we already plugged in
\eqref{form83}. This term will be treated by applying the
following proposition:
\begin{proposition}\label{christProp} Let $\{F_{s}\}_{s \in (0,\infty)}$ be a family of $C^{1}$-functions $F_{s} \colon \R \to \R$ satisfying
\begin{displaymath} \|F_{s}\|_{L^{\infty}} + \|F_{s}'\|_{L^{\infty}} \leq C_{F}, \qquad s \in (0,\infty), \end{displaymath}
where $C_{F} \geq 1$ is a constant independent of $s \in
(0,\infty)$. Assume also that $(s,x) \mapsto F_{s}(x)$ is Borel.
Let $\varphi \in C^{\infty}_{c}(\R)$ satisfy $\int \varphi = 1$,
and write $\varphi_{s}(x) := \tfrac{1}{s}\varphi(x/s)$. Further,
let $\{\psi_{s}\}_{s > 0} \subset C^{1}(\R \, \setminus \{0\})$ be
a family of functions which satisfy the following requirements for
some $C_{\psi} > 0$ and $\alpha \in (0,1]$:
\begin{enumerate}
\item $\spt \psi_{s} \subset B(0,C_{\psi}s)$, \item
$\|\psi_{s}(x)\|_{L^{\infty}(\R)} \leq C_{\psi}/s$ and
$|\psi_{s}'(x)| \leq C_{\psi}/s^{2}$ for $x \in \R \, \setminus \,
\{0\}$, and \item $|\hat{\psi}_{s}(\xi)| \leq
C_{\psi}\min\{|s\xi|^{\alpha},|s\xi|^{-\alpha}\}$ for $\xi \in
\R$.
\end{enumerate}
For $f \in L^{1}_{\mathrm{loc}}(\R)$, define $P_{s}(f) := f \ast
\varphi_{s}$ and $Q_{s}(f) := f \ast \psi_{s}$. Finally, let $a
\in L^{\infty}(\R)$, and define the operator
\begin{displaymath} (Tf)(x) := \int_{0}^{\infty} F_{s}(P_{s}(a)(x)) \cdot Q_{s}(f)(x) \, \frac{ds}{s}, \qquad f \in C^{\infty}_{c}(\R). \end{displaymath}
Then $T$ extends to a bounded operator on $L^{2}$ with
$\|T\|_{L^{2} \to L^{2}} \leq
C(\|a\|_{L^{\infty}},C_{F},\varphi,C_{\psi})$. \end{proposition}
This is a stronger variant of \cite[Proposition 9]{MR1104656}. Notably, $\psi_{s}$ need not be of the form $\psi_{s}(x) = \tfrac{1}{s}\psi(\tfrac{x}{s})$, and $\psi_{s}$ is even allowed to have a jump discontinuity at $0$.
We postpone the proof to Appendix \ref{a:christProp}. Using this
proposition, we will show that $\eqref{form81} \lesssim m + 1$.
The proof is based on re-writing \eqref{form81} in a form to which
Proposition \ref{christProp} can be applied.
\begin{lemma}\label{l:form148}For $m \geq 0$, we have
\begin{displaymath} \eqref{form81} = \int_{B(0,1)} \left|\int_{0}^{1} P_{s}(A')(x)^{m}(B_{1}' \ast \Psi_{s})(x) \, \frac{ds}{s} \right| \, dx, \end{displaymath}
where
\begin{equation}\label{Psis}
\Psi_{s}(z) = \frac{1}{2}\frac{z|z|}{\mathfrak{q}(z)}
\int_{|z|/s}^{\infty} \eta(t)\kappa(st) \left[1 - \frac{2|z|}{st}
\right] \, \frac{dt}{t}, \qquad s > 0, \, z \neq 0.
\end{equation}
\end{lemma}

\begin{proof}
Recalling the expression of \eqref{form81}, we need to prove that
\begin{equation}\label{form143_a}
 \int \eta\left(\frac{|x - y|}{s} \right)\kappa(x - y)\left( \int_{y}^{x} \frac{B_{1}(x) + B_{1}(y) - 2B_{1}(r)}{2\q(x - y)} \, \, dr \right)  \, dy
 = (B_{1}' \ast \Psi_{s})(x). \end{equation}
We have also changed the sign of the integral compared
to the expression in \eqref{form81}, which is convenient for  the
following computation, but irrelevant for the statement of the
lemma since  \eqref{form81} anyway involves absolute values.
  To prove \eqref{form143_a}, we make the
change-of-variables $r \mapsto uy + (1 - u)x$ in the
$r$-integration, and then use Fubini's theorem, to find the
following expression for the left-hand side
\begin{align} & \int_{0}^{1} \int \eta \left(\frac{|x - y|}{s} \right)(x-y)\kappa(x - y) \frac{B_{1}(x) + B_{1}(y) - 2B_{1}(uy + (1 - u)x)}{2\q(x - y)} \, dy \, du \notag\\
&\label{form86} \qquad = \frac{1}{2} \int_{0}^{1} \bigg[ \int
\eta\left(\frac{|x - y|}{s} \right)(x-y)\kappa(x -
y)\frac{B_{1}(x)
- B_{1}(uy + (1 - u)x)}{\q(x - y)} \, dy  \, \\
&\label{form142} \qquad \qquad \qquad + \int \eta\left(\frac{|x -
y|}{s} \right)(x-y)\kappa(x - y)\frac{B_{1}(y) - B_{1}(uy + (1 -
u)x)}{\q(x - y)} \, dy \bigg] \, \, du. \end{align} Before
proceeding, we need to develop independently the two $y$-integrals
in \eqref{form86}-\eqref{form142}. We have
\begin{equation}\label{form143} \frac{B_{1}(x) - B_{1}(uy + (1 - u)x)}{x - y} = u \int_{0}^{1} B_{1}'(x + (y - x)u(1 - r)) \, dr  \end{equation}
and
\begin{equation}\label{form144_a} \frac{B_{1}(y) - B_{1}(uy + (1 - u)x)}{x - y} = (u - 1) \int_{0}^{1} B_{1}'(x + (y - x)[u + r(1 - u)]) \, dr.  \end{equation}
Therefore, plugging \eqref{form143} into \eqref{form86} and
\eqref{form144_a} into \eqref{form142}, respectively, and
performing several changes of variables, we obtain
\begin{align}\label{form146} \int & \eta\left(\frac{|x - y|}{s} \right)\kappa(x-y)(x-y) \frac{B_{1}(x) - B_{1}(uy + (1 - u)x)}{\q(x - y)} \, dy\\
& = u \int_{0}^{1} \int \eta\left(\frac{|x - y|}{s} \right) \kappa(x-y)\frac{B_{1}'(x + (y - x)u(1 - r))}{\q(x - y)/(x - y)^2} \, dy \, dr \notag \\
& =  \int_{0}^{1} \int \eta\left(\frac{|z|}{su(1 - r)} \right)
\kappa\left(\frac{|z|}{u(1 - r)}\right)B_{1}'(x - z)
 \,\frac{dz}{\q(z)/(z|z|)} \frac{1}{1-r}\, dr \notag\\
 & =  \int B_{1}'(x - z) \int_{0}^{1} \eta \left(\frac{|z|}{su r} \right) \frac{1}{r}\kappa\left(\frac{|z|}{ur}\right)dr \, \frac{dz}{\q(z)/(z |z|)}\notag\\
& = \int B_{1}'(x - z) \frac{z|z|}{\q(z)} \int_{|z|/(su)}^{\infty}
\eta(t)\kappa(st) \, \frac{dt}{t} \, dz, \notag
\end{align}
and
\begin{align}\label{form147} \int
& \eta\left(\frac{|x - y|}{s} \right)
\kappa(x-y)(x-y)\frac{B_{1}(y) - B_{1}(uy + (1 - u)x)}{\q(x - y)}
\, dy
\\
&
 = -\int B_{1}'(x - z) \frac{z|z|}{\q(z)} \int_{|z|/s}^{|z|/(su)}
\eta(t) \kappa(ts)\, \frac{dt}{t} \, dz. \notag
\end{align} Both quantities \eqref{form146} and \eqref{form147}
depend on $u \in [0,1]$, but observe from \eqref{form86}-
\eqref{form142} that we are next allowed to "integrate out" this
$u$-dependence:
\begin{align}
\notag \int& \eta\left(\frac{|x - y|}{s} \right)\kappa(x -
y)\left( \int_{y}^{x} \frac{B_{1}(x) + B_{1}(y) - 2B_{1}(r)}{2\q(x
- y)} \, \, dr \right)  \, dy\\& = \frac{1}{2}\int_{0}^{1} \int
B_{1}'(x - z) \frac{z|z|}{\q(z)}
 \left[ \int_{|z|/s}^{\infty}  [\mathbf{1}_{[|z|/(su),\infty)}(t) - \mathbf{1}_{[|z|/s,|z|/(su)]}(t)]\eta(t)
\kappa(st)\, \frac{dt}{t} \, \right] dz \, du \notag\\
& = \frac{1}{2}\int B_{1}'(x - z)\frac{z|z|}{\q(z)}
\int_{|z|/s}^{\infty} \eta(t) \kappa(st)
\left[ \int_{0}^{1} \mathbf{1}_{[|z|/(su),\infty)}(t) - \mathbf{1}_{[|z|/s,|z|/(su)]}(t) \, du \right] \, \frac{dt}{t} \, dz\notag\\
&
 \notag=\frac{1}{2} \int B_{1}'(x - z)
\frac{z|z|}{\q(z)}
 \int_{|z|/s}^{\infty} \eta(t)\kappa(st)\left[1 - \frac{2|z|}{st}
\right] \, \frac{dt}{t} \, dz = B_{1}' \ast \Psi_{s}(x),
\end{align} where $\Psi_{s}$ is the function appearing in \eqref{Psis}. This completes the proof of \eqref{form143_a}.
\end{proof}


We then record as a separate lemma that the family
$\{\Psi_{s}\}_{s > 0}$ satisfies the hypotheses of Proposition
\ref{christProp}:

\begin{lemma}\label{l:PsiProp} For $s > 0$, the function $\Psi_{s}$ belongs to $C^1(\mathbb{R}\setminus \{0\})$, has zero mean, is supported in $B(0,s)$, and satisfies
\begin{displaymath} \|\Psi_{s}\|_{L^{\infty}(\R)} \lesssim \tfrac{1}{s}, \qquad |\Psi_{s}'(z)| \lesssim \tfrac{1}{s^{2}} \quad \text{ for } z \neq 0, \end{displaymath}
and
\begin{equation}\label{form144} |\widehat{\Psi}_{s}(\xi)| \lesssim \min\{|s\xi|,|s\xi|^{-1}\}, \qquad \xi \in \R. \end{equation}
\end{lemma}

\begin{proof} Let us recall for reading convenience that
\begin{displaymath} \Psi_{s}(z) = \frac{1}{2}\frac{z|z|}{\mathfrak{q}(z)} \int_{|z|/s}^{\infty} \eta(t)\kappa(st) \left[1 - \frac{2|z|}{st} \right] \, \frac{dt}{t}, \qquad s > 0, \, z \neq 0. \end{displaymath}
The support condition $\spt \Psi_{s} \subset B(0,s)$ is immediate
from $\spt \eta \subset [0,1]$. Before discussing the mean-zero
property, let us infer from $\spt \eta \subset [\tfrac{1}{4},1]$
and $|k(st)| \leq (st)^{-1}$ that
\begin{displaymath} |\Psi_{s}(z)| \lesssim \frac{1}{2s}\int_{1/4}^{1} \left| 1 - \frac{2|z|}{ts} \right| \, \frac{dt}{t^{2}} \lesssim \frac{1 + |z|/s}{s} \lesssim \frac{1}{s}, \qquad z \in B(0,s),  \end{displaymath}
so indeed $\|\Psi_{s}\|_{L^{\infty}} \lesssim 1/s$, and in
particular $\Psi_{s} \in L^{1}(\R)$.
 If $\mathfrak{q}(z) = z^2$, then $\Psi_{s}$ is odd, so the zero-mean property is clear. However, if $\mathfrak{q}(z) = z|z|$, the function $\Psi_{s}$ is even, and the zero-mean property is a little surprising -- given how little we know about $\kappa$. We give two arguments to justify it.
 First, one may apply \eqref{form143_a} with $B_{1}(x) = x$:
\begin{displaymath} \int \Psi_{s}(z) \, dz = (1 \ast \Psi_{s})(0) = \int \eta\left(\frac{|z|}{s} \right) \frac{\kappa(z)}{2\mathfrak{q}(z)}
\left(\int_{0}^{z} z - 2r \, dr \right) \, dz = 0,
\end{displaymath} since $\int_{0}^{z} z- 2r \, dr = 0$. A more
direct proof (in the case $\mathfrak{q}(z) = z|z|$) is the
following:
\begin{align*}2 \int \Psi_{s}(z) \, dz & =  \int_{0}^{\infty} \int_{z/s}^{\infty} \eta(t)\kappa(st) \left[1 - \frac{2z}{st} \right] \, \frac{dt}{t} \, dz
 = \int_{1/s}^{\infty} \left(\int_{0}^{\infty} \eta(zu)\kappa(szu) \, dz \right)\left[1 - \frac{2}{su}\right] \, \frac{du}{u}\\
& 
= \int_{1/s}^{\infty} \left( \int_{0}^{\infty} \eta(x)\kappa(sx)
\, dx \right) \left[1 - \frac{2}{su} \right] \, \frac{du}{u^{2}} =
c(s,\kappa) \int_{1/s}^{\infty} \left[ 1 - \frac{2}{su} \right] \,
\frac{du}{u^{2}} = 0.   \end{align*} Now, combining the support
and zero-mean properties of $\Psi_{s}$ with $|e^{2\pi i t} - 1|
\lesssim |t|$, we may infer the first part of \eqref{form144}:
\begin{displaymath} |\widehat{\Psi}_{s}(\xi)| \leq \|\Psi_{s}\|_{L^{\infty}(\R)}\int_{B(0,s)} |e^{-2\pi i \xi z} - 1| \, dz \lesssim |s\xi|. \end{displaymath}
The derivative estimate $|\Psi_{s}'(z)| \lesssim s^{-2}$ for $z
\neq 0$ follows easily by differentiating under the integral sign,
plus using $\spt \eta \subset [\tfrac{1}{4},1]$ and $|\kappa(z)|
\leq |z|^{-1}$. We omit the details.

It remains to establish the "large scale" Fourier decay
$|\widehat{\Psi}_{s}(\xi)| \lesssim |s\xi|^{-1}$. Using Fubini's
theorem, we compute explicitly
\begin{displaymath} \widehat{\Psi}_{s}(\xi) = \frac{1}{2} \int_{0}^{\infty} \eta(t)\kappa(st) \left( \int_{|z| \leq st} e^{-2\pi i z\xi} \frac{z|z|}{\mathfrak{q}(z)}\left[ 1 - \frac{2|z|}{st} \right] \, dz \right) \, \frac{dt}{t}. \end{displaymath}
To evaluate the expression in brackets, first change variables $z
\mapsto xst$ to find
\begin{displaymath} \left( \int_{|z| \leq st} e^{-2\pi i z\xi} \frac{z|z|}{\mathfrak{q}(z)}\left[ 1 - \frac{2|z|}{st} \right] \, dz \right)
= st \int_{|x| \leq 1} e^{-2\pi i x(\xi st)}
\frac{x|x|}{\mathfrak{q}(x)}(1 - 2|x|) \, dx. \end{displaymath}
The right hand side is the Fourier transform, evaluated at $\xi
st$, of a certain piecewise affine function supported in $B(0,1)$.
The function is not continuous, but can be written as a sum of two
functions of the form $\mathbf{1}_{[a,b]} \cdot (cx + d)$, with
$[a,b] \subset B(0,1)$. One may easily verify that the Fourier
transform $\mathcal{F}$ of any such function satisfies
$|\mathcal{F}(\xi)| \lesssim_{a,b,c,d} |\xi|^{-1}$. So, using once
more that $\spt \eta \subset [\tfrac{1}{4},1]$, and $|\kappa(st)|
\leq |st|^{-1}$, we find that
\begin{displaymath} |\widehat{\Psi}_{s}(\xi)| \lesssim \frac{1}{s|\xi|}\int_{1/4}^{1} \frac{dt}{t^{2}} \sim \frac{1}{s|\xi|}. \end{displaymath}
This completes the proof of the lemma. \end{proof}

Having established these properties of $\Psi_s$, we infer from
Lemma \ref{l:form148} that
\begin{equation*}
\eqref{form81}  = \int_{B(0,1)} \left| \int_{0}^{1} P_{s}(A')(x)^{m} \cdot Q_{s}(B_{1}')(x) \,
\frac{ds}{s} \right| \, dx =: (m + 1)\int_{B(0,1)} |T_{m}B_{1}'(x)| \,
dx,
\end{equation*}
where $Q_{s}(B_{1}') := B_{1}' \ast \Psi_{s}$. Lemma \ref{l:PsiProp} shows
that the operator $T_{m}$ defined by this equation is of the type
treated in Proposition \ref{christProp} for any functions $F_{s} \in C^{\infty}(\R)$ satisfying
\begin{displaymath} F_s(t) := F(t) := \begin{cases} \tfrac{1}{m + 1} t^{m}, & \text{for }|t| \leq 1, \\ 0, & \text{for $|t|>2$}, \end{cases} \end{displaymath}
for $s \in (0,1]$, $F_{s} \equiv 0$ for $s > 1$, and $a = A'\in L^{\infty}(\mathbb{R})$ (noting that $|P_s(a)|\leq
1$). Hence Proposition \ref{christProp} is applicable, and, after an application of Cauchy-Schwarz, it
yields
\begin{displaymath}
 \eqref{form81} \lesssim (m + 1)\left(\int_{B(0,1)} |T_{m}B_{1}'(x)|^{2} \, dx \right)^{1/2} \lesssim (m + 1)\|B_{1}'\|_{2} \lesssim m + 1. \end{displaymath}
In the last estimate, recall that $\|B_{1}'\|_{L^{\infty}} \leq
1$, and we arranged in \eqref{sptA} that $\spt B_{1} \subset
B(0,10)$. This completes the proof of the first estimate in
\eqref{T13}, and consequently the proof of the theorem, except for
the small modifications needed to treat the kernels
$C_{m,n}D_{A_{0}}$ and $C_{m,n}D_{B_{0}}$. We record these in the
next remark. \end{proof}

\begin{remark}\label{r:mod} We finally explain the minor changes needed to treat the kernels $C_{m,n}D_{A_{0}}$ and $C_{m,n}D_{B_{0}}$. For $C_{m,n}D_{B_{0}}$, they are quite non-existent. This kernel is $C_{m,n}$ multiplied by another factor of the type
\begin{displaymath} \tfrac{\widetilde{B}_{2}(x) - \widetilde{B}_{2}(y) - \tfrac{1}{2}[\widetilde{B}_{1}(x) + \widetilde{B}_{1}(y)](x - y)}{\q(x - y)}, \end{displaymath}
but the tame map "$\widetilde{B}$" is allowed to be different from
the map "$B$" appearing in the kernel $C_{m,n}$. In the case $n =
0$ one has
\begin{displaymath} (C_{m,0}D_{B_{0}})(x,y) = \kappa(x - y) \left[\tfrac{A(x) - A(y)}{x - y} \right]^{m}\tfrac{\widetilde{B}_{2}(x) - \widetilde{B}_{2}(y) - \tfrac{1}{2}[\widetilde{B}_{1}(x) + \widetilde{B}_{1}(y)](x - y)}{\q(x - y)}, \end{displaymath}
and the proof already presented for $C_{m,1}$ works verbatim. If $n \geq 1$, then taking absolute values inside at \eqref{T15}, and applying Cauchy-Schwarz brings one into the situation of Section \ref{s:casen2}. The proof can then be completed via $\beta$-number estimates for $B_{1}$ and $\widetilde{B}_{1}$.

We then consider the kernel $C_{m,n}D_{A_{0}}$. Again, the argument of Section \ref{s:casen2} works if $n \geq 2$, and the case $m = 0$ can be treated as $C_{1,n}$. So, the only non-trivial problem involves the kernel
\begin{displaymath} (C_{m,1}D_{A_{0}})(x,y) = \kappa(x - y) \left[\tfrac{A(x) - A(y)}{x - y}\right]^{m}\left[\tfrac{A_{0}(x) - A_{0}(y)}{x - y}\right]\tfrac{B_{2}(x) - B_{2}(y) - \tfrac{1}{2}[B_{1}(x) + B_{1}(y)](x - y)}{\q(x - y)} \end{displaymath}
with $m \geq 1$. One starts by repeating the "recursive argument"
in the beginning of Section \ref{s:nIs1}. After $m + 1$ steps, as
before, matters will have been reduced to bounding an analogue of
the term \eqref{form81}, which however this time reads
\begin{displaymath} \int_{B(0,1)} \left| \int_{0}^{1} P_{s}(A')(x)^{m}P_{s}(A_{0}')(x) \int \eta\left(\tfrac{|x - y|}{s} \right)
\kappa(x - y)\left( \int_{x}^{y} \tfrac{B_{1}(x) + B_{1}(y) -
2B_{1}(r)}{2\q(x - y)} \, \, dr \right)  \, dy \, \frac{ds}{s}
\right| \, dx.\end{displaymath}
It looks problematic to apply Proposition \ref{christProp}, since only one $L^{\infty}$-function "$a$" is allowed. In fact, multiple $L^{\infty}$-functions are no problem: in Appendix \ref{a:christProp}, we directly prove a version of the proposition which allows for an arbitrary number of $L^{\infty}$-functions. With the improved proposition in hand, the term above can be handled in a familiar manner.
\end{remark}

\subsection{Proof of Theorem \ref{t:mainCL} for intrinsic Lipschitz graphs}\label{s:CLProof} We interrupt the proof of Theorem \ref{t:mainRegularCurve} for a moment in order to establish Theorem \ref{t:mainCL} for intrinsic Lipschitz graphs over horizontal subgroups; the case of general regular curves will be completed in Section \ref{s:conclusion}. Recall that Theorem \ref{t:mainCL} concerned certain non-negative kernels of the form $k_{\alpha}(x,y,t) = |t|^{\alpha/2}/\|(x,y,t)\|^{1 + \alpha}$, with $\alpha \geq 4$. During Section \ref{s:CLProof}, let us again agree that $\|\cdot\|$ refers to the Kor\'anyi norm, so there will be no issues with the smoothness of the kernels. It turns out that the proof of Theorem \ref{t:mainCL} (in the case of intrinsic Lipschitz graphs) follows closely the arguments we saw just above, in Section \ref{s:casen2}.

\begin{thm}\label{t:ChousionisLiKernel} Let $\alpha \geq 4$. Then, the kernel $k_{\alpha}$ is a CZ kernel for intrinsic $L$-Lipschitz graphs over horizontal subgroups in $\He$, with $\|k_{\alpha}\|_{\mathrm{C.Z.}}$ only depending on $\alpha$ and $L$. \end{thm}

\begin{proof} Note that $k_{\alpha}(x,y,t) \lesssim_{\alpha} k_{4}(x,y,t)$, so we may assume $\alpha = 4$. By Proposition \ref{p:HolBoundGoodHeis} or \cite[Lemma 2.7]{MR3678492}, the kernel $k_{\alpha}$ is an  SK on $\He$. Let $\mathbb{V}$ be the $x$-axis, let $\mathbb{W}$ be the $yt$-plane, and let $\phi \colon \V \to \W$ be an intrinsic $L$-Lipschitz function
with graph map $\Phi \colon \V \to \He$.  As in \eqref{form137a}, we insert the explicit formula for the graph map in the expression of the kernel (evaluated at arbitrary points of the intrinsic Lipschitz graph $\Phi(\V)$):
\begin{align*} k_{4}(\Phi(x)^{-1}\cdot \Phi(x_0)) & =  \frac{\left|\phi_{2}(x_0) - \phi_{2}(x) + \tfrac{1}{2}[\phi_{1}(x_0) + \phi_{1}(x)](x_0 - x)\right|^{2}}{\|\Phi(x)^{-1}\cdot \Phi(x_0)\|^{5}}\\
& \leq  \frac{1}{|x_0-x|} \left(\frac{|\phi_{2}(x_0) - \phi_{2}(x)
+ \tfrac{1}{2}[\phi_{1}(x_0) + \phi_{1}(x)](x_0-
x)|}{|x_0-x|^2}\right)^{2}
\end{align*}
for $x,x_{0} \in \R$ with $x \neq x_{0}$, using that
$\|\Phi(x)^{-1}\cdot \Phi(x_0)\| \geq |x - x_{0}|$. Recall from
Proposition \ref{prop2}) that $(\phi_{1},-\phi_{2})$ is a tame
map. We have now reduced the proof of Theorem
\ref{t:ChousionisLiKernel} to a real variable problem, which we
solve in the next proposition (which should be applied with
$(B_{1},B_{2}) := (\phi_{1},-\phi_{2})$).
\end{proof}

\begin{proposition}\label{NonNegCommTheorem} Let $B = (B_{1},B_{2}) \colon \R \to \R^{2}$ be $N$-tame, $N \geq 1$. Then the kernel
\begin{displaymath}
K(x,y):= \frac{1}{|x-y|}\left(\frac{|B_{2}(x) - B_{2}(y) -
\tfrac{1}{2}[B_{1}(x) + B_{1}(y)](x - y)|}{|x-y|^2}\right)^2
\end{displaymath}
satisfies $\|K\|_{\mathrm{C.Z.}} \leq C$, for a constant $C\geq 1$
that depends only on $N$.
\end{proposition}

\begin{proof}
We first observe that $\|K\|_{1,strong}\lesssim N$. Indeed,
recalling the familiar kernels $C_{m,n}$ from \eqref{eq:Cmn}, we
have $K(x,y) = |C_{0,2}(x,y)|$. Consequently, the size and
H\"older continuity properties of $K$ follow from the same
properties for $C_{0,2}$, established in Example \ref{ex:CSK}, and
the triangle inequality.

To conclude that $K$ is a CZ kernel, we argue as in the proof of
Theorem \ref{commutatorTheorem}. Using the same definition for
$\psi_{\epsilon}$ as above \eqref{universalT1}, we denote
$K_{\epsilon}(x,y):= \psi_{\epsilon}(x-y)K(x,y)$ and
\begin{displaymath}
Tf(x)= \int K_{\epsilon}(x,y)f(y)\,dy,\quad f\in\mathcal{S}.
\end{displaymath}
Since $K_{\epsilon}$ is symmetric, the  $T1$ testing conditions in
\eqref{universalT1} reduce to one condition. Moreover, from this
point on, we will assume without loss of generality that $B$ is a
$1$-tame map. This amounts to a harmless multiplicative constant
in the kernel, and the reductions starting from \eqref{sptA}
apply verbatim. The proof is then concluded as in Section
\ref{s:casen2}, recalling that $K(x,y) = |C_{0,2}(x,y)|$. The
point is that the exponent "$2$" spares us from any arguments
involving cancellation.
\end{proof}



\section{The exponential kernel returns}\label{s:exp2}

In Theorem \ref{TABCZO}, we showed that if $A \colon \R \to \R$ is
$1$-Lipschitz, and $B \colon \R \to \R^{2}$ is $1$-tame, then
$K_{A,B}$ is a CZ-kernel. Moreover, if  $A_0 \colon \R \to \R$ is
Lipschitz, and $B_0 \colon \R \to \R^{2}$ is tame, then also
$K_{A,B}D_{A_0}$ and $K_{A,B}D_{B_0}$ are CZ-kernels. In this
section, we prove Theorem \ref{mainProp}, which contained the more precise claim that
\begin{displaymath} \|K_{A,B}D_{A_{0}}\|_{\mathrm{C.Z.}} \lesssim_{A_{0}} \mathrm{poly}(M,N) \quad \text{and} \quad \|K_{A,B}D_{B_0}\|_{\mathrm{C.Z.}} \lesssim_{B_0}
\mathrm{poly}(M,N) \end{displaymath}
whenever $A
\colon \R \to \R$ is $M$-Lipschitz, and $B \colon \R \to \R^{2}$
is $N$-tame. The result will be reduced to the case $M = 1
= N$ via the corona decompositions for Lipschitz functions and
tame maps from Section \ref{s:CoronaTame}. In fact, this manner of
reasoning works, without extra effort, in slightly higher
generality. Let us fix, for the entire section, an SK $k \colon \R
\times \R \, \setminus \, \bigtriangleup \to \R$ such that
$\|k\|_{\alpha,strong} \leq 1$, $\alpha \in (0,1]$. We also assume
that
\begin{equation}\label{form125} k(x,y) = 0, \qquad |x - y| \leq \epsilon, \end{equation}
for some fixed $\epsilon > 0$. Then, let us (re-)define
\begin{equation}\label{KAB}K_{A,B}(x,y)
:= k(x,y) \exp\left(2\pi i\left[ \tfrac{A(x) - A(y)}{x - y} + \tfrac{B_{2}(x) - B_{2}(y) - \tfrac{1}{2}[B_{1}(x) + B_{1}(y)](x - y)}{\q(x - y)} \right]
 \right),\end{equation}
where $A \colon \R \to \R$ is Lipschitz,  $B \colon \R \to \R^{2}$
is tame, and $\q:\mathbb{R}\to \mathbb{R}$ is one of the functions
$\q(s)=s^2$ or $\q(s)=|s|s$. For $M,N \geq
1$, and the fixed kernel $k$ as above \eqref{form125}, we define
\begin{displaymath} \wp_{k}(M,N) := \wp(M,N) := \sup \{\|K_{A,B}\|_{\mathrm{C.Z.}} : A \text{ is } M\text{-Lipschitz and } B \text{ is } N\text{-tame}\}. \end{displaymath}
Thus, Theorem \ref{TABCZO} implies that $\wp_{k}(1,1) < \infty$ for the SKs
\begin{equation}\label{form154} k(x,y) = \kappa(x - y) \cdot D_{A_{0}}(x,y) \quad \text{and} \quad k(x,y) = \kappa(x - y) \cdot D_{B_{0}}(x,y). \end{equation}
(To be perfectly precise, the kernels first need to be multiplied
by inverses of kernel constants, and smoothly truncated, to fit
into the framework of the section.) From now on, we work
abstractly with the \emph{a priori} assumption
\begin{equation}\label{form124} C_{0}(k) := \wp(1,1) < \infty. \end{equation}
Now, we arrive at the main result of the section:

\begin{thm}\label{polyBound} There exists a constant $C_{1} := C_{1}(k) \geq 1$, depending only on $C_{0}(k)$ in \eqref{form124}, and there exists a constant
$C_2:=C_2(\alpha)$, depending only on
  $\alpha \in (0,1]$, such that the following holds. Let $M,N \geq 1$.
  Let $A \colon \R \to \R$ be $M$-Lipschitz, and let $B \colon \R \to \R^{2}$ be $N$-tame. Then
\begin{equation}\label{form110} \|K_{A,B}\|_{\mathrm{C.Z.}} \leq C_{1}\max\{M,N\}^{C_{2}}. \end{equation}
\end{thm}

We already noted above that the assumption \eqref{form124} is valid for the kernels \eqref{form154} relevant for Theorem
\ref{mainProp}. So, Theorem \ref{mainProp} is a consequence of Theorem \ref{polyBound}.
Theorem \ref{polyBound} will be inferred from the following recursive statement:
\begin{thm}\label{recursionTheorem} Let $M,N \in 2^{\N}$. Then, there exists a constant $C = C_{\alpha} \geq 1$ such that
\begin{equation}\label{form49} \wp(M,N) \leq \min\{C_{M/2,N},C_{M,N/2}\}, \end{equation}
where
\begin{displaymath} C_{M,N} := C \max\{M,N^{2},\wp(M,N)\}.  \end{displaymath}
\end{thm}

Let us quickly deduce Theorem \ref{polyBound} from Theorem \ref{recursionTheorem}.
\begin{proof}[Proof of Theorem \ref{polyBound} assuming Theorem \ref{recursionTheorem}]
Let $C_{1} := \max\{C_{0}(k),1\}$, $C_{2} := \max\{2\log_2C,2\}$.
Assume that we already have \eqref{form110} with these constants
"$C_{1}$" and "$C_2$" for some $M = N \in 2^{\N}$, that is,
$\wp(N,N) \leq C_{1}N^{C_{2}}$. This is true for $M = 1 = N$ by
\eqref{form124}. From two applications of \eqref{form49}, the
inductive hypothesis, and noting that $2^{C_{2}} \geq C^{2}$, we
find that
\begin{align*} \wp(2N,2N) & \stackrel{\eqref{form49}}{\leq} C\max\{2N,N^{2},\wp(2N,N)\}\\
& \stackrel{\eqref{form49}}{\leq} C\max\{2N,N^{2},C\max\{N^{2},\wp(N,N)\}\}\\
& \leq C\max\{2N,N^{2},C\max\{N^{2},C_{1}N^{C_{2}}\}\}\\
& = C^{2}C_{1}N^{C_{2}} \leq C_{1}(2N)^{C_{2}}. \end{align*} This
completes the proof. \end{proof}

For the remainder of the section, we will view the H\"older continuity parameter $\alpha \in (0,1]$ as "fixed", so any "absolute constants" are actually allowed to depend on $\alpha$.

\subsection{Proof of Theorem \ref{recursionTheorem}: getting started} We begin the proof of Theorem \ref{recursionTheorem}. The argument is based on ideas from Semmes' paper \cite{MR1087183}, although our setting allows for some simplifications. We fix an $M$-Lipschitz function $A \colon \R \to \R$, and an $N$-tame map $B = (B_{1},B_{2}) \colon \R \to \R^{2}$, with $M,N \in 2^{\N}$. Write
\begin{displaymath} Tf(x) := \int K_{A,B}(x,y)f(y) \, dy, \end{displaymath}
which is well-defined for e.g. $f \in L^{2}(\R)$ due to \eqref{form125}. In the sequel, we abbreviate $K(x,y) := K_{A,B}(x,y)$.
The plan will be to show that for any dyadic interval $Q_{0} \in \calD$, the $T1$ testing condition
\begin{equation}\label{form29} \fint_{Q_{0}} |T(b)| \, dx \leq \min\{C_{M/2,N},C_{M,N/2}\}, \end{equation}
familiar from \eqref{universalT1}, holds for all functions $b \in C^{\infty}(\R)$ with $\mathbf{1}_{2Q_{0}} \leq b \leq \mathbf{1}_{3Q_{0}}$. The estimate \eqref{form29} (and a similar, completely symmetric, estimate for $T^{t}$) imply by Corollary \ref{goodT1} that $\|T\|_{L^{2} \to L^{2}} \lesssim \min\{C_{M/2,N},C_{M,N/2}\} + \|K_{A,B}\|_{strong}$. To conclude from here, recall from Example \ref{ex1} that $\|K_{A,B}\|_{strong} \lesssim \max\{M,N\} \leq \min\{C_{M/2,N},C_{M,N/2}\}$. So, \eqref{form49} follows.

Fix $b \in C^{\infty}(\R)$, as in \eqref{form29}. Now, \eqref{form29} is actually composed of two distinct inequalities: we will mostly concentrate on proving the inequality
\begin{equation}\label{form51} \fint_{Q_{0}} |T(b)| \, dx \leq C_{M,N/2}, \end{equation}
that is, the one where the "tameness constant" is reduced by a factor of $2$. The argument for the other inequality in \eqref{form29} is virtually the same, and we will indicate the small differences in Section \ref{s:form29Section}. To show \eqref{form51}, we start by applying the tame corona decomposition, Theorem \ref{t:Corona} -- or more precisely its Corollary \ref{cor:Corona} -- to the $N$-tame function $B$, with parameter $\eta = \tfrac{1}{2}$. The result is a decomposition $\calD = \calB \dot{\cup} \calG$, as explained in the statement of Theorem \ref{t:Corona}, a collection $\calF$ of trees $\calT \subset \calD$, and for each tree a function of the form $\Psi_{\calT} = \psi_{\calT} + \calL_{T}$, where $\psi_{\calT}$ is $(N/2)$-tame, $\calL_{\calT}$ is tame-linear, and the good approximation property \eqref{form112} holds. To recap:
\begin{equation}\label{form113} d_{\pi}(B(s),\Psi_{\calT}(s)) \leq \tfrac{1}{2}N|Q|, \qquad s \in 2Q, \; Q \in \calT \in \calF. \end{equation}
 Were we proving the second inequality in \eqref{form29}, we would, instead, start with the corona decomposition in Theorem \ref{t:LipCorona} of the $M$-Lipschitz function $A$, at level $M/2$.

To benefit from the decomposition $\calD = \calB \dot{\cup} \calG$, we will now decompose the operator $T$ in an analogous manner. For $j \in \Z$, we first define the operator $T_{j}$ by
\begin{displaymath} T_{j}f(x) := \int_{\{y : 2^{-j} \leq |x - y| \leq 2^{-j + 1}\}} K(x,y)f(y) \, dy.  \end{displaymath}
Then, we set
\begin{displaymath} T_{Q}f := \chi_{Q}T_{j}f, \qquad Q \in \calD_{j}, \: j \in \Z, \end{displaymath}
and write
\begin{equation}\label{form39} Tf = \sum_{Q \in \calD} T_{Q}f = \sum_{Q \in \calB} T_{Q}f + \sum_{\calT \in \calF} \sum_{Q \in \calT} T_{Q}f =: \sum_{Q \in \calB} T_{Q}f + \sum_{\calT \in \calF} T_{\calT}f. \end{equation}
We begin by disposing of the first sum. Note that for $Q \in \calD_{j}$, we have
\begin{displaymath} |T_{Q}(b)(x)| \lesssim \mathbf{1}_{Q}(x) \fint_{B(x,2^{j + 1})} |b(y)| \, dy \leq \mathbf{1}_{Q}(x), \end{displaymath}
using that $|K(x,y)| \leq |x - y|^{-1}$ and $\|b\|_{L^{\infty}} \leq 1$. Therefore, for $g \in L^{\infty}(Q_{0})$ with $\|g\|_{L^{\infty}(Q_{0})} = 1$, we have
\begin{displaymath} \left| \int_{Q_{0}} \left[ \sum_{Q \in \calB} T_{Q}(b) \right] \, g \right| \lesssim \mathop{\sum_{Q \in \calB}}_{Q \subset Q_{0}} \langle |g| \rangle_{Q}|Q| + \mathop{\sum_{Q \in \calD}}_{Q \supset Q_{0}} \langle |g| \rangle_{Q}|Q| \lesssim |Q_{0}|. \end{displaymath}
The implicit constants only depend on the Carleson packing constant of the family $\calB$. This is better than what we need for \eqref{form51}.

We then concentrate on the second sum in \eqref{form39}.  We claim that for individual trees $\calT \in \calF$, we have the estimate
\begin{equation}\label{form28} \|T_{\calT}\|_{L^{2} \to L^{2}} \leq C_{M,N/2}. \end{equation}
This will imply that
\begin{equation}\label{form41} \int_{Q_{0}} \sum_{\calT \in \calF} \left| T_{\calT}(b) \right| \, dx \lesssim C_{M,N/2}|Q_{0}|, \end{equation}
as we will next check, and hence complete the proof of \eqref{form51}. Assume then for the moment that \eqref{form28} holds, and write
\begin{equation}\label{form42} \int_{Q_{0}} \sum_{\calT \in \calF} \left| T_{\calT}(b) \right| \, dx = \int_{Q_{0}} \sum_{\calT \in \calF_{0}} |T_{\calT}(b)| \, dx + \int_{Q_{0}} \sum_{\calT \in \calF \, \setminus \, \calF_{0}} |T_{\calT}(b)| \, dx,  \end{equation}
where $\calF_{0} = \{\calT \in \calF : Q(\calT) \subset Q_{0}\}$. The second term in \eqref{form42} is straightforward to estimate, so we start from there. If $\calT \in \calF \, \setminus \, \calF_{0}$ is tree satisfying
\begin{equation}\label{form115} \int_{Q_{0}} |T_{\calT}(b)| \, dx \neq 0, \end{equation}
then $Q_{0} \subset Q(\calT)$, since $T_{\calT}(b)$ is supported on $Q(\calT)$. In addition, there exists $Q \in \calT$ and $x \in Q_{0}$ such that
\begin{equation}\label{form114} T_{Q}(b)(x) = \mathbf{1}_{Q}(x) \int_{\{y : |Q| \leq |x - y| \leq 2|Q|\}} K(x,y)b(y) \, dy \neq 0. \end{equation}
Hence $x \in Q \cap Q_{0}$, so either $Q \subset Q_{0}$, or $Q_{0} \subset Q$. In the second case, \eqref{form114} forces $|Q| \lesssim |Q_{0}|$, because $\spt b \subset 3Q_{0}$. In the first case, since $Q_{0} \subset Q(\calT)$, there anyway exists a parent $Q' \in \calT$ of $Q$ such that $Q_{0} \subset Q'$ and $|Q'| \sim |Q_{0}|$. We conclude that whenever \eqref{form115} holds for some $\calT \in \calF \, \setminus \, \calF_{0}$, there exists $Q \in \calT$ with $Q_{0} \subset Q$ and $|Q| \lesssim |Q_{0}|$. But since the trees $\calT \in \calF$ are disjoint, this implies that \eqref{form115} can only occur for boundedly many $\calT \in \calF \, \setminus \, \calF_{0}$. Hence, the second sum in \eqref{form42} is bounded by a constant times
\begin{displaymath} \int_{Q_{0}} |T_{\calT}b| \, dx \leq |Q_{0}|^{1/2} \|T_{\calT}(b)\|_{L^{2}(\R)} \leq C_{M,N} |Q_{0}|^{1/2}\|b\|_{L^{2}(Q_{0})} \lesssim C_{M,N}|Q_{0}|, \end{displaymath}
as desired. To estimate the first sum in \eqref{form42}, we use the Carleson packing condition for the top intervals $\calQ(\calT)$ with $\calT \in \calF_{0}$. Recalling that $\|b\|_{L^{\infty}} \leq 1$ and $\spt b \subset 3Q_{0}$, and also observing that
\begin{displaymath} T_{\calT}(b) = \mathbf{1}_{Q(\calT)}T_{\calT}(\mathbf{1}_{5Q(\calT)}b), \qquad \calT \in \calF, \end{displaymath}
we estimate as follows:
\begin{align*} \int_{Q_{0}} \sum_{\calT \in \calF_{0}} |T_{\calT}(b)| \, dx & \leq \sum_{\calT \in \calF_{0}} \left( \fint_{Q(\calT)} |T_{\calT}(b)|^{2} \, dx \right)^{1/2}|Q(\calT)|\\
& \lesssim C_{M,N} \sum_{\calT \in \calF_{0}} \left(\fint_{5Q(\calT)} |b|^{2} \, dx \right)^{1/2} |Q(\calT)|\\
& \leq C_{M,N} \mathop{\sum_{\calT \in \calF_{0}}}_{5Q(\calT) \cap 3Q_{0} \neq \emptyset} |Q(\calT)| \lesssim C_{M,N} |Q_{0}|. \end{align*}
The implicit constants only depend on the Carleson packing constant of the top intervals $Q(\calT)$, $\calT \in \calF$. We have now reduced \eqref{form41} to proving \eqref{form28}.

To prove \eqref{form28}, fix $\calT \in \calF$ and $f \in L^{2}(\R)$, and write $j_{0}$ for the generation of $Q(\calT)$, that is, $Q(\calT) \in \calD_{j_{0}}$. Note that
\begin{align*} T_{\calT}f(x) & = \sum_{Q \in \calT} T_{Q}f(x) = \mathop{\sum_{Q \in \calT}}_{x \in Q} \int_{\{y : |Q| \leq |x - y| \leq 2|Q|\}} K(x,y)f(y) \, dy\\
& = \mathbf{1}_{Q(\calT)}(x)\int_{\{y : h(x) \leq |x - y| \leq \rho\}} K(x,y)f(y) \, dy, \qquad x \in \R, \end{align*}
where $\rho = 2^{-j_{0} + 1} = 2|Q(\calT)|$, and
\begin{equation}\label{functionH} h(x) := \inf\{|Q| : x \in Q \in \calT\},\qquad\text{for }x\in Q(\calT). \end{equation}
Now, following an idea in \cite{MR1087183}, we want to "replace" $T_{\calT}$ by the somewhat regularised operator
\begin{equation}\label{form62} \bar{T}_{\calT}f(x) := \mathbf{1}_{Q(\calT)}(x)\int_{\{y : D(x,y) \leq |x - y| \leq \rho\}} K(x,y)f(y) \, dy, \end{equation}
where
\begin{equation}\label{D} D(x,y) := \frac{d(x) + d(y)}{4}, \end{equation}
and $d \colon \R \to \R$ is the $1$-Lipschitz function
\begin{equation}\label{functionD} d(x) = \inf \{|Q| + \dist(x,Q) : Q \in \calT\}, \qquad x \in \R. \end{equation}
By "replacement", we mean that $\|T_{\calT}\|_{L^{2} \to L^{2}} \lesssim \|\bar{T}_{\calT}\|_{L^{2} \to L^{2}} + \max\{M,N\}$, so it will suffice to prove \eqref{form28} for $\bar{T}_{\calT}$ in place of $T_{\calT}$. Let us now see carefully how to dominate $T_{\calT}$ by $\bar{T}_{\calT}$.
\begin{lemma} If $x,y \in \R$ with $x \in Q(\calT)$ and $|x - y| \geq h(x)$, then $|x - y| \geq D(x,y)$.
\end{lemma}

\begin{proof} We use the facts that $d$ is $1$-Lipschitz, and $d(x) \leq h(x)$ to estimate as follows:
\begin{displaymath} D(x,y) \leq \frac{d(x) + d(x) + |x - y|}{4} \leq \frac{h(x)}{2} + \frac{|x - y|}{4} \leq \frac{3|x - y|}{4} \leq |x - y|. \end{displaymath}  \end{proof}
\begin{cor} Consider the kernel $K^{D,\rho}(x,y) := K(x,y)\mathbf{1}_{\{D(x,y) \leq |x - y| \leq \rho\}}(x,y)$. Then,
\begin{equation}\label{form117} |T_{\calT}f(x)| \leq \sup_{\delta > 0} \left| \int_{\{y : |x - y| \geq \delta\}} K^{D,\rho}(x,y)f(y) \, dy \right| =: \bar{T}_{\calT}^{\ast}f(x), \qquad x \in \R. \end{equation}
\end{cor}

\begin{proof} The estimate \eqref{form117} is clear if $x\notin Q(\calT)$, since then $T_{\calT}f(x)=0$, so we assume in the following that $x\in Q(\calT)$.
Choose $\delta := \max\{h(x),\epsilon\} > 0$, where $\epsilon > 0$ was the \emph{a priori} truncation from \eqref{form125} (in other words, $K(x,y) = 0$ whenever $|x - y| < \epsilon$). Then, if $h(x) \leq |x - y| \leq \rho$, and $K(x,y) \neq 0$, we also have $|x - y| \geq \delta$, and $D(x,y) \leq |x - y| \leq \rho$ by the previous lemma. This shows that
\begin{displaymath}  \int_{\{y : h(x) \leq |x - y| \leq \rho\}} K(x,y)f(y) \, dy = \int_{\{y : |x - y| \geq \delta\}} K^{D,\rho}(x,y)f(y) \, dy,  \end{displaymath}
and the claim follows. \end{proof}

So, at least $T_{\calT}$ is dominated by $\bar{T}_{\calT}^{\ast}$. But since $D,\rho$ are $\tfrac{1}{2}$-Lipschitz functions ($\rho$ being a $0$-Lipschitz function), we find from Lemma \ref{hormander} that $K^{D,\rho}$ is a GSK with
\begin{equation}\label{form44} \|K^{D,\rho}\| \lesssim \|K\| \lesssim \max\{M,N\}, \end{equation}
and hence Cotlar's inequality \eqref{cotlarsInequality} applies:
\begin{displaymath} \bar{T}_{\calT}^{\ast}f(x) \lesssim M(|\bar{T}_{\calT}f|)(x) + \|\bar{T}_{\calT}\|_{\mathrm{C.Z.}}Mf(x), \qquad f \in L^{2}(\R), \; x \in \R. \end{displaymath}
Here $\|\bar{T}_{\calT}\|_{\mathrm{C.Z.}} = \|K^{D,\rho}\| + \|\bar{T}_{\calT}\|_{L^{2} \to L^{2}}$ by definition. Combining this inequality with \eqref{form117} and \eqref{form44}, we infer that
\begin{displaymath} \|T_{\calT}\|_{L^{2} \to L^{2}} \lesssim \|\tilde{T}_{\calT}\|_{L^{2} \to L^{2}} + \max\{M,N\},   \end{displaymath}
as desired. Consequently, \eqref{form28} will follow (with a slightly worse constant) once we manage to establish that
\begin{equation}\label{form50} \|\tilde{T}_{\calT}\|_{L^{2} \to L^{2}} \leq C_{M,N/2}. \end{equation}
To simplify notation a little bit, we will, from now on, write "$T_{\calT}$" in place of "$\bar{T}_{\calT}$" for the operator associated to the $D(x,y)$-truncation. This should cause no confusion, because there will be no further reference to the original operator $T_{\calT}$.

\subsection{Applying the corona decomposition} To prove \eqref{form50}, we recall the functions
\begin{displaymath} \Psi_{\calT} := \Psi = \psi_{\calT} + \calL_{\calT} =: \psi + \calL \end{displaymath}
associated to the fixed tree $\calT$, where $\psi = (\psi_{1},\psi_{2}) \colon \R \to \R^{2}$ is $(N/2)$-tame, and $\calL = (L,P) := \R \to \R^{2}$ is $2N$-tame-linear. We recall from \eqref{form112} that
\begin{equation}\label{form119} d_{\pi}(B(s),\Psi(s)) \leq (N/2)|Q|, \qquad s \in 11Q, \; Q \in \calT. \end{equation}
To be accurate, \eqref{form112} only gives \eqref{form119} for $s \in 2Q$, but enlarging the constant from "$2$" (or anything $> 1$) to "$11$" is a standard trick, see e.g. the argument on \cite[p. 20]{DS1}. Alternatively, one could just prove \eqref{form112} directly with constant "$11$". To establish the good $L^{2}$-bound for $T_{\calT}$, we want to compare it to a suitable operator $T_{\Psi}$ associated to the kernel
\begin{equation}\label{form34}K_{A,\Psi}(x,y) = k(x,y)\exp\left(2\pi i \left[\tfrac{A(x) - A(y)}{x - y} + \tfrac{\psi_{2}(x) - \psi_{2}(y) - \tfrac{1}{2}[\psi_{1}(x) + \psi_{1}(y)](x - y)}{\q(x - y)} \right]\right). \end{equation}
The reader should protest that the right hand side of \eqref{form34} is, in fact, the kernel of $K_{A,\psi}$ instead of $K_{A,\Psi}$. Have we forgotten about the tame-linear part $\calL = (L,P)$ altogether? No: recalling that $L$ is linear, and $\dot{P} = L$, one easily checks that
\begin{displaymath} P(x) - P(y) - \tfrac{1}{2}[L(x) + L(y)](x - y) \equiv 0. \end{displaymath}
In other words,
\begin{equation}\label{form73} K_{A,\Psi} = K_{A,\psi}. \end{equation}
This is crucial: the kernel $K_{A,\Psi}$ approximates $K_{A,B}$ well (using information from the corona decomposition, as we will soon see), while $K_{A,\psi}$ is a kernel associated to an $(N/2)$-tame function $\psi$. On the other hand, $\Psi$ can be, at worst, $2N$-tame, so without knowing \eqref{form73}, the kernel $K_{A,\Psi}$ would be no better than $K_{A,B}$!

Now, we abbreviate
\begin{displaymath} \tilde{K}(x,y) := K_{A,\psi}(x,y) = K_{A,\Psi}(x,y), \end{displaymath}
and define the operator $T_{\Psi}$ with the same $D(x,y)$-truncation as in the definition of $T_{\calT}$:
\begin{equation}\label{form63} T_{\Psi}f(x) = \mathbf{1}_{Q(\calT)}(x)\int_{\{y : D(x,y) \leq |x - y| \leq \rho\}} \tilde{K}(x,y)f(y) \, dy. \end{equation}
To prove \eqref{form50}, we will establish that
\begin{equation}\label{form33} \|T_{\calT}\|_{L^{2} \to L^{2}} \lesssim \|T_{\Psi}\|_{L^{2} \to L^{2}} + \max\{M,N^{2}\} \lesssim \wp(M,N/2) + \max\{M,N^{2}\}. \end{equation}
The second inequality in \eqref{form33} is virtually a consequence of the definition of the number $\wp(M,N/2)$, and \eqref{form73}, since $A$ is $M$-Lipschitz, and $\psi$ is $(N/2)$-tame. A little technicality is the presence of the $D(x,y)$-truncation, but we can dispose of it by easy maximal function tricks, as follows. Recalling that $D(x,y) = (d(x) + d(y))/4$, we claim that
\begin{equation}\label{form118} \left| T_{\Psi}f(x) - \int_{\{y : d(x)/4 \leq |x - y| \leq \rho\}} \tilde{K}(x,y)f(y) \, dy \right| \lesssim Mf(x). \end{equation}
Indeed, since $D(x,y) \geq d(x)/4$, the left hand side of \eqref{form118} is bounded by
\begin{displaymath} \int_{\{y : d(x)/4 \leq |x - y| < D(x,y)\}} |\tilde{K}(x,y)||f(y)| \, dy \leq \frac{4}{d(x)} \int_{B(x,d(x))} |f(y)| \, dy \lesssim Mf(x).  \end{displaymath}
We used that $D(x,y) \leq d(x)/2 + |x - y|/4$, so $|x - y| \leq D(x,y)$ implies that $|x - y| \leq d(x)$. Now, it follows from \eqref{form118} and Cotlar's inequality that
\begin{displaymath} \|T_{\Psi}f\|_{L^{2}} \lesssim \|T_{A,\Psi}^{\ast}f\|_{L^{2}} + \|f\|_{L^{2}} \lesssim \|T_{A,\Psi}f\|_{L^{2}} + (1 + \|\tilde{K}\|)\|f\|_{L^{2}}. \end{displaymath}
Here $\|T_{A,\Psi}f\|_{L^{2}} = \|T_{A,\psi}f\|_{L^{2}} \leq \wp(M,N/2)\|f\|_{L^{2}}$ by \eqref{form73} and the definition of $\wp(M,N/2)$, while $\|\tilde{K}\| \lesssim \max\{M,N\}$. This completes the proof of the second part of \eqref{form33}, and the rest of the section is devoted to establishing the first part.

\subsection{A Whitney decomposition} Recall that $d(x) = \inf\{\dist(x,Q) + |Q| : Q \in \calT\}$, so $d$ is $1$-Lipschitz, and well-defined on $\R$. However, the set
\begin{displaymath} E := \{x \in \R : d(x) = 0\} \end{displaymath}
is a compact subset of $\overline{Q(\calT)}$. It follows easily from \eqref{form119} that
\begin{equation}\label{form61} \Psi(s) = B(s), \qquad s \in E. \end{equation}
In this short section, we perform a Whitney type decomposition of $\R \, \setminus \, E$. Fix $x \in \R \, \setminus \, E$. Since $0 < d(x) \leq \dist(x,Q(\calT)) + |Q(\calT)| < \infty$, and $d$ is continuous (hence $d$ stays positive in a neighbourhood of $x$), there exists a maximal dyadic interval $I \ni x$ with
\begin{equation}\label{form54} \inf_{y \in I} d(y) =  \inf_{y \in I} \inf_{Q \in \calT} \{d(y,Q) + |Q|\} \geq |I|. \end{equation}
These intervals are disjoint and cover $\R \, \setminus \, E$, and we will denote them $\calS$. We first observe that
\begin{equation}\label{form56} |S| \leq d(y) \leq 4|S|, \qquad y \in S \in \calS. \end{equation}
Indeed, the lower bound is immediate from the definition \eqref{form54}. To see the upper bound, note that by the maximality of $S \in \calS$ there exists $y'$ in the parent $\widehat{S}$ of $S$ with $d(y') < |\widehat{S}| = 2|S|$, whence $d(y) \leq d(y') + |\widehat{S}| \leq 4|S|$, as claimed. We next observe that
\begin{equation}\label{form59} S \in \calS \text{ and } S \subset 11Q(\calT) \quad \Longrightarrow \quad d_{\pi}(B(s),\Psi(s)) \lesssim N|S|, \qquad s \in S. \end{equation}
Indeed, fix $x \in S$ and, based on \eqref{form56}, find $Q \in \calT$ with $d(x,Q) + |Q| \leq 5|S|$. Then, let $Q' \in \calT$ be the minimal ancestor of $Q$ in $\calT$ with $S \subset 11Q'$ (this exists because $S \subset 11Q(\calT)$). It is easy to check that $|Q'| \sim |S|$, and now \eqref{form59} follows from \eqref{form119} applied to $s \in 11Q'$.

\subsection{Comparing $T_{\calT}$ and $T_{\Psi}$} Recall that $T_{\calT}$ and $T_{\Psi}$ are the operators defined in \eqref{form62} and \eqref{form63}, respectively. To prove the first inequality in \eqref{form33}, that is,
\begin{displaymath}
 \|T_{\calT}\|_{L^{2} \to L^{2}} \lesssim \|T_{\Psi}\|_{L^{2}
 \to L^{2}
 } + \max\{M,N^{2}\} ,
\end{displaymath}
we fix $f,g \in L^{2}(\R)$. It suffices to show that
\begin{equation}\label{form60} \left| \int (T_{\calT}f)g - \int (T_{\Psi}f)g \right| \lesssim \max\{M,N^{2}\}\|f\|_{L^{2}}\|g\|_{L^{2}}.  \end{equation}
Since $T_{\calT}(f) = \mathbf{1}_{Q(\calT)}T_{\calT}(f\mathbf{1}_{5Q(\calT)})$ and $T_{\Psi}(f) = \mathbf{1}_{Q(\calT)}T_{\Psi}(f\mathbf{1}_{5Q(\calT)})$, which follows from the upper $\rho$-truncation in \eqref{form62} and \eqref{form63} (recall: $\rho = 2|Q(\calT)|$), it moreover suffices to prove \eqref{form60} for $f,g$ satisfying
\begin{displaymath} \spt f \cup \spt g \subset 5Q(\calT). \end{displaymath}
To estimate the difference in \eqref{form60}, we introduce the following auxiliary notation. If $x \in E$, we define $S(x) = \{x\}$, and otherwise $S(x)$ is the unique element in $\calS$ containing $x$. If $h \colon \R \to \R$ is a function, and $x \in \R$, we then define
\begin{displaymath} h_{\geq x}(y) := h(y)\mathbf{1}_{\{|S(y)| \geq |S(x)|\}}(y) \quad \text{and} \quad h_{> x}(y) := h(y)\mathbf{1}_{\{|S(y)| > |S(x)|\}}(y).  \end{displaymath}
The functions $h_{\leq x}$ and $h_{< x}$ are defined similarly, swapping the inequalities. Note that $h_{> x}|_{E} \equiv 0$ for any $x \in \R$, and $h_{< x} \equiv 0$ whenever $x \in E$. With this notation, we have
\begin{displaymath} \int (T_{\calT}f)(x)g(x) \, dx = \int (T_{\calT}f_{\geq x})(x)g(x) \, dx + \int (T_{\calT}f_{< x})(x)g(x) \, dx, \end{displaymath}
where further
\begin{align*}  \int (T_{\calT}f_{< x})(x)g(x) \, dx & = \int g(x) \left[ \int_{\{D(x,y) \leq |x - y| \leq \rho\}} K(x,y)f(y)\mathbf{1}_{\{|S(y)| < |S(x)|\}}(y) \, dy \right] \, dx\\
& = \int f(y) \left[ \int_{\{D(x,y) \leq |x - y| \leq \rho\}} K(x,y) g(x)\mathbf{1}_{\{|S(x)| > S(y)|\}}(x) \, dx \right] \, dy\\
& = \int (T^{t}_{\calT}g_{> y})(y)f(y) \, dy. \end{align*}
The same calculation works if "$\calT$" is replaced with "$\Psi$". Consequently,
\begin{align}\label{form64} \int (T_{\calT}f)g - \int (T_{\Psi}f)g & = \int [T_{\calT}f_{\geq x} - T_{\Psi}f_{\geq x}](x)g(x) \, dx\\
&\label{form120} + \int [T^{t}_{\calT}g_{> y} - T^{t}_{\Psi}g_{> y}](y)f(y) \, dy. \end{align}
We will only estimate the term on line \eqref{form64}, since the argument for the second term is virtually the same. This is actually a reason why we introduced the "symmetric" $D(x,y)$-truncation: to make the term on line \eqref{form120} look as similar to \eqref{form64} as possible.

\subsubsection{Estimate for \eqref{form64}}\label{s:form64} The plan is to fix $x \in \spt g \subset 5Q(\calT)$, and obtain pointwise bounds for the expression $[T_{\calT}f_{\geq x} - T_{\Psi}f_{\geq x}](x)$, which we spell out as follows:
\begin{equation}\label{form65} [T_{\calT}f_{\geq x} - T_{\Psi}f_{\geq x}](x) = \mathop{\sum_{S \in \calS}}_{|S| \geq |S(x)|} \int_{\{y \in S : D(x,y) \leq |x - y| \leq \rho\}} K(x,y)f(y) - \tilde{K}(x,y)f(y) \, dy. \end{equation}
But is this also accurate when $x \in E$, that is, when $|S(x)| = 0$? Then, the \emph{a priori} correct expression for $[T_{\calT}f_{\geq x} - T_{\Psi}f_{\geq x}](x)$ is actually
\begin{displaymath}\sum_{S \in \calS} \int_{\{y \in S : D(x,y) \leq |x - y| \leq \rho\}} K(x,y)f(y) - \tilde{K}(x,y)f(y) \, dy + \int_{E} f(y)[K(x,y) - \tilde{K}(x,y)] \, dy. \end{displaymath}
However, when $x,y \in E$, as in the second integration, then $B(x) = \Psi(x)$ and $B(y) = \Psi(y)$ by \eqref{form61}, so $K(x,y) = \tilde{K}(x,y)$. Consequently, the second integral contributes nothing, and \eqref{form65} is indeed true even when $x \in E$.

We will now write "$I_{x}(S)$" for the individual terms in \eqref{form65}, with $|S| \geq |S(x)|$. Note that intervals $S \in \calS$ with $S \cap 5Q(\calT) = \emptyset$ contribute nothing to \eqref{form65}, so they can be discarded. But if $S \cap 5Q(\calT) \neq \emptyset$, then $d(y) \leq \dist(y,Q(\calT)) + |Q(\calT)| \leq 3|Q(\calT)|$ for some $y \in S$. This implies by \eqref{form56} that $|S| \leq 3|Q(\calT)|$, and consequently,
\begin{equation}\label{form121} S \subset 11Q(\calT). \end{equation}
In fact this inclusion explains our choice of the constant ``$11$'' in \eqref{form119}.
We proceed to estimate the pieces $I_{x}(S)$ in a manner adapted from \cite{MR1087183}, eventually proving the following claim: the intervals $S \in \calS$ with $|S| \geq |S(x)|$ and $S \subset 5Q(\calT)$ can be split into two groups $\calG_{1}(x)$ and $\calG_{2}(x)$, where
\begin{equation}\label{form36} |I_{x}(S)| \lesssim \frac{\max\{M,N^{2}\}|S|}{\dist(x,S)^{2} + |S|^{2}} \int_{S} |f(y)| \, dy, \qquad S \in \calG_{1}(x), \end{equation}
and
\begin{equation}\label{form37} \sum_{S \in \calG_{2}(x)} |I_{x}(S)| \lesssim Mf(x). \end{equation}
The estimate (for \eqref{form64}) concerning group $\calG_{2}(x)$ is straightforward:
\begin{displaymath} \int |g(x)| \sum_{S \in \calG_{2}(x)} |I_{x}(S)| \, dx \lesssim \int |g(x)|Mf(x) \, dx \lesssim \|g\|_{L^{2}}\|f\|_{L^{2}}. \end{displaymath}
Before proceeding with the proofs of \eqref{form36}-\eqref{form37}, let us briefly see that the estimate \eqref{form36} leads to essentially the same conclusion (up to multiplication by $\max\{M,N^{2}\}$):
\begin{lemma} Let $1 < p < \infty$, and $1/p + 1/q = 1$. Then, for $g \in L^{p}$ and $f \in L^{q}$, we have
\begin{equation}\label{form38} \int |g(x)| \left[ \sum_{S \in \calS} \frac{|S|}{\dist(x,S)^{2} + |S|^{2}} \int_{S} |f(y)| \, dy \right] \, dx \lesssim_{p} \|g\|_{L^{p}}\|f\|_{L^{q}}. \end{equation}
\end{lemma}

\begin{proof} We start by rewriting and estimating the left hand side as follows:
\begin{align*} \textrm{L.H.S. of \eqref{form38}} & = \sum_{S \in \calS} \left[ \int \frac{|S| \, |g(x)| \, dx}{\dist(x,S)^{2} + |S|^{2}} \right] \left(\fint_{S} |f(y)| \, dy \right) |S|\\
& \leq \left( \sum_{S \in \calS} \left[ \int \frac{|S| \, |g(x)| \, dx}{\dist(x,S)^{2} + |S|^{2}} \right]^{p} |S| \right)^{1/p} \left( \sum_{S \in \calS} \left( \fint_{S} |f(y)| \, dy \right)^{q} |S| \right)^{1/q}.  \end{align*}
Since the intervals in $\calS$ are disjoint, the second factor is evidently controlled by $\|Mf\|_{L^{q}} \lesssim_{p} \|f\|_{L^{q}}$. The first factor is also dominated by the maximal function, since for $S \in \calS$ fixed,
\begin{align*} \int \frac{|S| \, |g(x)| \, dx}{\dist(x,S)^{2} + |S|^{2}} & \lesssim \fint_{2S} |g(x)| \, dx + \sum_{j \geq 0} \frac{1}{2^{2j}|S|} \int_{\{x : 2^{j}|S| \leq \dist(x,S) \leq 2^{j + 1}|S|\}} |g(x)| \, dx\\
& \lesssim \sum_{j \geq 0} 2^{-j} \left( \inf_{y \in S} Mg(y) \right) \leq \inf_{y \in S} Mg(y), \end{align*}
and consequently
\begin{displaymath}\left( \sum_{S \in \calS} \left[ \int \frac{|S| \, |g(x)| \, dx}{\dist(x,S)^{2} + |S|^{2}} \right]^{p} |S| \right)^{1/p} \lesssim \left(\sum_{S \in \calS} \int_{S} [Mg(y)]^{p} \, dy \right)^{1/p} \lesssim_{p} \|g\|_{L^{p}}, \end{displaymath}
as desired. \end{proof}

This allows us to conclude the estimate for \eqref{form64} (but see Section \ref{s:summary} for a final "wrap-up" of the whole argument). We then begin to verify the estimates \eqref{form36}-\eqref{form37}. We fix $x \in 5Q(\calT)$ and $S \in \calS$ with $|S| \geq |S(x)|$ and $S \subset {11}Q(\calT)$. Since $S(x) \cap 5Q(\calT) \neq \emptyset$, the argument above \eqref{form121} also yields
\begin{equation}\label{form121a} S(x) \subset {11}Q(\calT). \end{equation}

\subsubsection{Case where $\dist(S(x),S) \geq 2|S|$ and $\{y \in S : D(x,y) \leq |x - y| \leq \rho\} = S$}\label{s:mainCase} This is the "main case", and we write
\begin{align} I_{x}(S) &= \int_{\{y \in S : D(x,y) \leq |x - y| \leq \rho\}} K(x,y)f(y) - \tilde{K}(x,y)f(y) \, dy \notag\\
&\label{form66} = \int_{S} [K(x,y) - K(x,y_{0})]f(y) + [K(x,y_{0}) - \tilde{K}(x,y)]f(y) \, dy, \end{align}
where $y_{0}$ is the midpoint of $S$. In particular, $|y - y_{0}| \leq |S| \leq |x - y_{0}|/2$. We give pointwise estimates for the two differences of the kernels in \eqref{form66}. The first difference is easier, as the same kernel "$K$" appears twice, and
\begin{equation}\label{form67} |K(x,y) - K(x,y_{0})| \lesssim \frac{\max\{M,N\}|S|}{\dist(x,S)^{2}} \end{equation}
follows from standard estimates for $K$. We claim a similar estimate also for the second difference in \eqref{form66}, and we start by writing
\begin{equation}\label{form31} |K(x,y) - \tilde{K}(x,y_{0})| \leq |K(x,y) - K(x,y_{0})| + |K(x,y_{0}) - \tilde{K}(x,y_{0})|. \end{equation}
The first term is the same as \eqref{form67}, so let us concentrate on the second one. Recalling the definitions, and writing $\Psi = \psi + \calL =: (\Psi_{1},\Psi_{2})$, this term equals
\begin{align} |K & (x,y_{0}) - \tilde{K}(x,y_{0})| \label{form68}\\
& \leq \frac{1}{|x - y_{0}|} \Bigg| \exp\left(2\pi i\left[\frac{A(x) - A(y_{0})}{x - y_{0}} + \frac{B_{2}(x) - B_{2}(y_{0}) - \tfrac{1}{2}[B_{1}(x) + B_{1}(y_{0})](x - y_{0})}{\q(x - y_{0})} \right] \right) \notag\\
&\quad - \exp\left(2\pi i\left[\frac{A(x) - A(y_{0})}{x - y_{0}} + \frac{\Psi_{2}(x) - \Psi_{2}(y_{0}) - \tfrac{1}{2}[\Psi_{1}(x) + \Psi_{1}(y_{0})](x - y_{0})}{\q(x - y_{0})} \right] \right) \Bigg|. \notag \end{align}
To estimate the difference, we just use that $t \mapsto \exp(2\pi i t)$ is $2\pi$-Lipschitz, and $|\q(s)|=s^2$. The ensuing upper bound for \eqref{form68} is
\begin{displaymath} \frac{2\pi}{|x - y_{0}|} \left( \frac{|B_{2}(x) - \Psi_{2}(x)| + |B_{2}(y_{0}) - \Psi_{2}(y_{0})|}{|x - y_{0}|^{2}} + \frac{|B_{1}(x) - \Psi_{1}(x)| + |B_{1}(y_{0}) - \Psi_{1}(y_{0})|}{2|x - y_{0}|} \right). \end{displaymath}
To estimate these terms, we plug in the information from the corona decomposition on the quality of approximation of $B$ by $\Psi$. Since $x \in S(x) \subset {11}Q(\calT)$ (by \eqref{form121a}) and $y_{0} \in S \subset {11}Q(\calT)$, and $|S(x)| \leq |S|$, we deduce from \eqref{form59} that
\begin{displaymath} |B_{2}(x) - \Psi_{2}(x)| \lesssim N^{2}|S(x)|^{2} \leq N^{2}|S|^{2} \quad \text{and} \quad |B_{2}(y_{0}) - \Psi_{2}(y_{0})| \lesssim N^{2}|S|^{2}. \end{displaymath}
For the same reasons,
\begin{displaymath} |B_{1}(x) - \Psi_{1}(x)| \lesssim N|S| \quad \text{and} \quad |B_{1}(y_{0}) - \Psi_{1}(y_{0})| \lesssim N|S|. \end{displaymath}
Combining these estimates, and recalling that $|x - y_{0}| \geq \dist(S(x),S) \geq 2|S|$, we infer that
\begin{displaymath} |K(x,y_{0}) - \tilde{K}(x,y_{0})| \lesssim \frac{N^{2}|S|^{2}}{\dist(x,S)^{3}} + \frac{N|S|}{\dist(x,S)^{2}} \lesssim \frac{N^{2}|S|}{\dist(x,S)^{2}}. \end{displaymath}
Combining \eqref{form67} and the estimate above, we conclude that
\begin{equation}\label{form71} I_{x}(S) \lesssim \frac{\max\{M,N^{2}\}|S|}{\dist(x,S)^{2} + |S|^{2}} \int_{S} |f(y)| \, dy. \end{equation}
This matches the estimate in \eqref{form36}, so in this case $S \in \calG_{1}(x)$.

\subsubsection{Case where $\dist(S(x),S) \leq 2|S|$} Recall from \eqref{form56} that $d(y) \sim |S|$ for all $y \in S$. Therefore, if $|x - y| \geq D(x,y) = (d(x) + d(y))/4$, we certainly have $|x - y| \gtrsim |S|$. Since $\max\{|K(x,y)|,|\tilde{K}(x,y)|\} \leq |x - y|^{-1}$, and $d(x,S) \leq |S(x)| + \dist(S(x),S) \leq 3|S|$, we conclude that
\begin{displaymath} I_{x}(S) \lesssim \frac{1}{|S|} \int_{S} |f(y)| dy \lesssim \frac{|S|}{d(x,S)^{2} + |S|^{2}} \int_{S} |f(y)| \, dy. \end{displaymath}
This matches the estimate in \eqref{form36}, so again $S \in \calG_{1}(x)$.

\subsubsection{Case where $\dist(S(x),S) \geq 2|S|$ and $\{y \in S : D(x,y) \leq |x - y| \leq \rho\} \neq S$} This case \emph{a priori} divides into two further sub-cases: either
\begin{equation}\label{form69} |x - y_{0}| < D(x,y_{0}) \quad \text{or} \quad |x - y_{0}| > \rho \end{equation}
for some $y_{0} \in S$. We assume that the former option holds, and pick $y_{0} \in S$ with $|x - y_{0}| < D(x,y_{0}) = (d(x) + d(y_{0}))/4$. Then, using the $1$-Lipschitz property of $d$, we first deduce that
\begin{displaymath} |x - y_{0}| < \frac{d(x) + d(y_{0})}{4} \leq \frac{d(x)}{2} + \frac{|x - y_{0}|}{4}, \end{displaymath}
and consequently
\begin{displaymath} d(x) \geq \tfrac{3}{2}|x - y_{0}| \geq \tfrac{3}{2}\dist(x,S). \end{displaymath}
Since $|S| \leq \dist(x,S)$, this implies that $S \subset B(x,3d(x))$. Consequently, also noting that the integration in $I_{x}(S)$ only takes into account such $y \in \R$ with $|x - y| \geq D(x,y) \gtrsim d(x)$, we find from the estimates $\max\{|K(x,y)|,|\tilde{K}(x,y)|\} \leq |x - y|^{-1}$ that
\begin{equation}\label{form122} \mathop{\sum_{S \subset {11}Q(\calT)}}_{\inf_{y_{0} \in S} [|x - y_{0}| - D(x,y_{0})] < 0} |I_{x}(S)| \lesssim \frac{1}{d(x)} \int_{B(x,3d(x))} |f(y)| \, dy \lesssim Mf(x). \end{equation}
This is the estimate desired in \eqref{form37}, so we can include all $S \in \calS$ with $\inf_{y \in S} [|x - y| - D(x,y)] < 0$ to the collection $\calG_{2}(x)$.

Finally, assume that the second option in \eqref{form69} is realised, and pick $y_{0} \in S$ accordingly. If $|S| \leq \rho/2$, then $\inf_{y \in S} |x - y| \geq \rho/2$ by the triangle inequality. But even if $|S| \geq \rho/2$, we have $\inf_{y \in S} |x - y| = \dist(x,S) \geq 2|S| \geq \rho$ by the case assumption. So,
\begin{displaymath} \mathop{\sum_{S \subset {11}Q(\calT)}}_{\sup_{y_{0} \in S} |x - y_{0}| > \rho} |I_{x}(S)| \lesssim \rho^{-1} \int_{5Q(\calT)} |f(y)| \,dy\lesssim Mf(x),  \end{displaymath}
which is the same estimate as in \eqref{form122}. The proof of this -- final -- case is complete.

\subsubsection{Summary}\label{s:summary} We have now proven that all the intervals $S \in \calS$ with $|S| \geq |S(x)|$ and $S \subset {11}Q(\calT)$, for $x \in 5Q(\calT)$, can be split into the groups $\calG_{1}(x)$ and $\calG_{2}(x)$ so that \eqref{form36}-\eqref{form37} hold. As we saw directly under \eqref{form36}-\eqref{form37}, we can then conclude the estimate
\begin{displaymath} \int |T_{\calT}f_{\geq x} - T_{\Psi}f_{\geq x}](x)||g(x)| \, dx \leq \int |g(x)| \sum_{|S| \geq |S(x)|} I_{x}(S) \, dx \lesssim \max\{M,N^{2}\}\|f\|_{L^{2}}\|g\|_{L^{2}}. \end{displaymath}
Repeating rather verbatim the same argument, we could also show that
\begin{displaymath} \int |[T^{t}_{\calT}g_{> y} - T^{t}_{\Psi}g_{> y}](y)||f(y)| \, dy \lesssim \max\{M,N^{2}\}\|f\|_{L^{2}}\|g\|_{L^{2}}, \end{displaymath}
and consequently the splitting in \eqref{form64} shows that
\begin{displaymath} \left| \int (T_{\calT}f)g - \int (T_{\Psi}f)g \right| \lesssim \max\{M,N^{2}\}\|f\|_{L^{2}}\|g\|_{L^{2}}. \end{displaymath}
Since $f,g \in L^{2}(\R)$ were arbitrary functions, this allows us to conclude the first inequality in \eqref{form33}, namely that $\|T_{\calT}\|_{L^{2} \to L^{2}} \lesssim \|T_{\Psi}\|_{L^{2} \to L^{2}} + \max\{M,N^{2}\}$. Since we already established the second inequality in \eqref{form33}, we may then infer \eqref{form50}, which then implies \eqref{form28}, and finally \eqref{form51} (one of the two inequalities in \eqref{form29}).

\subsubsection{The second inequality in \eqref{form29}}\label{s:form29Section} As we explained above, we have now established one of the two inequalities claimed in \eqref{form29}. We still need to establish the second:
\begin{equation}\label{form70} \fint_{Q_{0}} |T(b)| \, dx \leq C_{M/2,N}. \end{equation}
As we noted below \eqref{form51}, the first step is to apply Theorem \ref{t:LipCorona} to the $M$-Lipschitz function $A$ at level $M/2$, and then decompose the operator $T$ with respect to the ensuing families of intervals $\calB$ and $\{\calT\}_{\calT \in \calF}$, as in \eqref{form39}. For each tree $\calT \in \calF$, the corona decomposition yields an $(M/2)$-Lipschitz function $\psi_{\calT} \colon \R \to \R$, and a linear map $L_{\calT}:\R \to \R$. However, the proof presented above makes no explicit reference to these "approximating" functions before the introduction of the kernel $K_{A,\Psi}$ in \eqref{form34}. So, the argument is literally the same until that point. In proving \eqref{form70}, the relevant "approximating" kernel is
\begin{displaymath}\tilde{K}(x,y) =  k(x,y) \exp \left(2\pi i \left[\tfrac{(\psi + L)(x) - (\psi + L)(y)}{x - y} + \tfrac{B_{2}(x) - B_{2}(y) - \tfrac{1}{2}[B_{1}(x) + B_{2}(y)](x - y)}{\q(x - y)}\right] \right), \end{displaymath}
because $|A(x) - (\psi + L)(x)|$ is the quantity controlled by the corona information for $x \in 2Q$ and $Q \in \calT$, recall the estimates in Section \ref{s:mainCase}. As before, the crux of the proof is to prove the analogue of \eqref{form33}, namely
\begin{equation}\label{form72} \|T_{\calT}\|_{L^{2} \to L^{2}} \lesssim \|T_{\Psi}\|_{L^{2} \to L^{2}} + \max\{M,N\} \lesssim \wp(M/2,N) + \max\{M,N\}. \end{equation}
Here $T_{\calT}$ is precisely the same object as in the previous sections, and
\begin{displaymath} T_{\Psi}f(x) = \int_{\{y : D(x,y) \leq |x - y| \leq \rho\}} \tilde{K}(x,y)f(y) \, dy. \end{displaymath}
The proof of the first inequality in \eqref{form72} is virtually the same as above: the formula of the kernel $\tilde{K}$ only plays a role in Section \ref{s:mainCase}, and the upper bound for $|A(x) - (\psi + L)(x)|$, coming from the corona decomposition, is exactly of the form applicable in \eqref{form68}. So, one can conclude \eqref{form71}, in fact with constant "$\max\{M,N\}$" in place of "$\max\{M,N^{2}\}$".

The proof of the second inequality in \eqref{form72} contains the only essential, albeit easy, difference in the proofs. Namely, recall from the discussion around \eqref{form73} that the equation $K_{A,\Psi} = K_{A,\psi}$ was crucially important. Now, the same is not true, but we have something comparable, and good enough. Namely, if $L(x) = cx$, we have
\begin{displaymath} K_{\psi + L,B}(x,y) = e^{2\pi i c} K_{\psi,B}(x,y), \qquad x,y \in \R, \; x \neq y. \end{displaymath}
Thus, even though $\psi + L$ is not $(M/2)$-Lipschitz, the $L^{2} \to L^{2}$ operator norm of
\begin{displaymath} T_{\psi + L,B}f(x) = \int K_{\psi + L,B}(x,y)f(y) \, dy = e^{2\pi i c} \int K_{\psi,B}(x,y)f(y) \, dy \end{displaymath}
is bounded from above by $\wp(M/2,N)$. This fact (in combination with Cotlar's inequality, as discussed after \eqref{form33}) allows us to conclude the second inequality in \eqref{form72}. This completes the proof of \eqref{form70}, and hence the proof of \eqref{form29} and of Theorem \ref{recursionTheorem}.

\section{Regular curves and big pieces of intrinsic Lipschitz graphs}\label{s:RegCurveBPiLG}

In this section, we prove Theorem \ref{t:mainRegularCurve}, which states that certain SKs in $\He$ are CZ kernels for (Hausdorff measures on) regular curves. The plan is to reduce the assertion to its special case concerning intrinsic Lipschitz graphs, Theorem \ref{t:mainIntrLipGraph}, through the observation that regular curves have big pieces of intrinsic Lipschitz graphs (Theorem \ref{p:BPiLG}). Further, the transition from ``intrinsic Lipschitz graphs'' to sets with ``big pieces of intrinsic Lipschitz graphs'' is based on an abstract argument, originally due to David \cite{MR744071,MR956767} in $\R^{n}$. We will record a version of this argument in all proper metric spaces $(X,d)$, see Theorem \ref{t:AbstractBigPiece} below, although the case $X = \He$ suffices for our application.

\subsection{David's big piece theorem in metric spaces}

\begin{definition}[Regular measures]\label{d:Sigma_k} Let $(X,d)$ be a metric space, and let $k > 0$. We write $\Sigma_k$ for the class of \emph{$k$-regular measures on $X$}, that is, Borel regular measures $\mu$ on $X$ with the property that there exists a finite constant $C \geq 1$ such that
\begin{equation}\label{regConstant}
C^{-1}r^{k} \leq \mu(B(x,r))\leq C r^k,\qquad x\in \spt \mu,\, r > 0.
\end{equation}
The smallest constant $C \geq 1$ such that \eqref{regConstant} holds will be denoted $\mathrm{reg}_{k}(\mu)$, or just $\mathrm{reg}(\mu)$.
\end{definition}

If $\mu\in \Sigma_k$, then $\spt \mu$ is a $k$-regular set and, since the lower bound is required to hold for arbitrarily large $r>0$, it follows that $\diam(X,d) \geq \diam(\spt \mu) = \infty$. This is a matter of technical convenience. Anyway, our focus will be on $1$-regular curves in the metric space $X=\He$, and every such curve is contained in an unbounded $1$-regular curve.

\begin{thm}\label{t:AbstractBigPiece}  Let $(X,d)$ be a proper metric space, and let $k > 0$. Let $K \colon X \times X \, \setminus \, \bigtriangleup \to \C$ be a $k$-GSK, and assume that $\mu \in \Sigma_k$ has the following properties. There exist constants $0<\theta<1$, $C\geq 1$ and, for each $1<p< \infty$, a finite constant $A_p\geq 0$ such that the following is true. For every closed ball $B$ centred on $\spt \mu$, there exists a Borel regular measure $\sigma$ on $X$, and a compact set $E\subset B\cap \spt \mu$, such that
\begin{enumerate}
\item $\sigma\in \Sigma_k$ with $\mathrm{reg}(\sigma) \leq C$,
\item $\mu(E)\geq \theta \mu(B)$,
\item $\mu(A\cap E) \leq \sigma(A)$ for all $A\subset X$,
\item\label{i:assT*} $\|T_{\sigma}^{\ast} f\|_{L^p(\sigma)}\leq A_p \|f\|_{L^p(\sigma)}$ for $f\in \mathcal{C}_c(X)$.
\end{enumerate}
Then, there are constants $C_p > 0$, for $1<p< \infty$, depending only on $(k,p,A_p,C,\mathrm{reg}(\mu),\|K\|,\theta)$ such that
\begin{equation}\label{eq:ConclAbstrThm}
\|T_{\mu}^{\ast} f\|_{L^p(\mu)}\leq C_p \|f\|_{L^p(\mu)},\quad f \in C_{c}(X).
\end{equation}
\end{thm}

Theorem \ref{t:AbstractBigPiece} in $\R^{n}$ is due to David \cite[Proposition 4 bis.]{MR956767}, see also \cite[III.3,Proposition 3.2]{MR1123480} and \cite[Proposition 1.18]{MR1251061}, and it is  based on ``good $\lambda$ inequalities''. The proof of the $(X,d)$ version follows David's proof very closely, and there are no real difficulties. The main differences are:
\begin{itemize}
\item David only considers $k$-SKs $K \colon \R^{n} \times \R^{n} \, \setminus \, \bigtriangleup \to \C$ satisfying
\begin{displaymath} |\nabla_{x} K(x,y)| + |\nabla_{y} K(x,y)| \lesssim |x - y|^{-1 - k}, \qquad x \neq y. \end{displaymath}
In contrast, we consider $k$-GSKs, and associated operators $T^{\ast}$. In this generality, we do not know if $T^{\ast}f$ is lower semicontinuous, which causes minor technical trouble in the proof of Lemma \ref{l:ProofuANDv}.
\item At one point of the original proof, David seems to refer to the Besicovitch covering theorem, which is not available in metric spaces. However, it turns out that the $5r$-covering theorem suffices, see Lemma \ref{l:mixed}.
\end{itemize}
Often, when arguments follow \cite[Proposition 4 bis.]{MR956767} verbatim, we will omit details.

\subsubsection{Proof of Theorem \ref{t:AbstractBigPiece}}
The version of the ``good $\lambda$ inequalities'' which we use in the proof of Theorem \ref{t:AbstractBigPiece} is borrowed from \cite[III, Lemma 3.1]{MR1123480}:
\begin{proposition}\label{l:GoodLambda}
Let $(X,\mu)$ be a measure space, and let $1< p<\infty$. Let $u:X\to [0,+\infty]$ be a $\mu$ measurable function that agrees with an $L^p(\mu)$ function outside a set of finite $\mu$ measure, and let $v: X\to [0,+\infty]$ be an $L^p(\mu)$ function. Assume that there exists a constant $0<\nu<1$ such that, for all $\varepsilon>0$, there is a constant $\gamma>0$ so that, for all $\lambda>0$,
\begin{equation}\label{eq:lambdaeq}
\mu(\{x\in X: u(x) >\lambda + \varepsilon \lambda\text{ and }v(x)\leq \gamma \lambda \})\leq (1-\nu) \mu(\{x\in X:\; u(x)>\lambda\}).
\end{equation}
Then $u\in L^p(\mu)$ with $\|u\|_{L^p(\mu)}\leq C(p,\varepsilon,\nu,\gamma) \|v\|_{L^p(\mu)}$.
\end{proposition}
A proof for the case $X=\mathbb{R}$ and $\mu =\mathcal{L}^1$ is included below \cite[Lemme 12]{MR744071} (we do not need an explicit expression of $C(p,\varepsilon,\nu,\gamma) $ for our purposes). The version for an arbitrary measure space $(X,\mu)$ is proven in the same way (David leaves this as an exercise in \cite{MR1123480}).

The proof of Theorem \ref{t:AbstractBigPiece} follows by applying Proposition \ref{l:GoodLambda} for given $f\in\mathcal{C}_c(X)$ and $1<p<\infty$ to the functions \begin{equation}\label{eq:test_functions}
u:= T^{\ast}_{\mu}f \quad \text{and}\quad v:= \mathcal{M}_{\mu,k}f + \left((\mathcal{M}_{\mu,k} |f|^{\sqrt{p}})\right)^{\frac{1}{\sqrt{p}}},
\end{equation}
where $\mathcal{M}_{\mu,k}$ is the radial maximal function of order $k$ (see Section \ref{sss:SIOR}). For $\mu \in \Sigma_{k}$, we will abbreviate $\mathcal{M}_{\mu} := \mathcal{M}_{\mu,k}$. In order to employ Proposition  \ref{l:GoodLambda}, we want to show that $u$ agrees with an $L^p(\mu)$ function outside a compact set, namely outside a closed ball $B(x_{\ast},2R)$, where $x_{\ast} \in X$, and $R>0$ is so large that $\mathrm{spt}f\subseteq B(x_{\ast},2R)$. Moreover, we have to verify that $u$ and $v\in L^p(\mu)$ satisfy \eqref{eq:lambdaeq}. This will yield Theorem \ref{t:AbstractBigPiece} since $\|v\|_{L^p(\mu)}\leq C(p,\mathrm{reg}(\mu))\, \|f\|_{L^p(\mu)}$. We start with some preliminaries.

Whenever $\mu \in \Sigma_k$, the triple $(\spt \mu,\mu,d)$ a doubling metric measure space, and $\mathcal{M}_{\mu}$ is bounded on $L^p(\mu)$ for $1<p<\infty$. We need a more general version of this result that involves two distinct measures in $\Sigma_{k}$ with potentially distinct, even disjoint, supports. David states this in \cite[Lemma 2.2, p. 58]{MR1123480}, and writes that the proof is easy, and based on the Besicovitch covering theorem. This tool is not available in our generality, but, in fact, the $5r$-covering theorem is good enough.
\begin{lemma}\label{l:mixed} Assume that $(X,d)$ is a proper metric space and $k > 0$. Let $\mu,\sigma \in \Sigma_{k}$, and $1<p< \infty$. Then, there exists a constant $0<C<\infty$, depending only on $p$ and $\mathrm{reg}(\mu),\mathrm{reg}(\sigma)$, such that
\begin{displaymath}
\|\mathcal{M}_{\mu} f\|_{L^p(\sigma)}\leq C \|f\|_{L^p(\mu)},\quad f \in L^p(\mu).
\end{displaymath}
\end{lemma}

\begin{proof}
Lemma \ref{l:mixed} is proved in the same way as \cite[Proposition 4]{MR744071}, using Marcinkiewicz interpolation. One has to show that $\mathcal{M}_{\mu}$ maps $L^{\infty}(\mu)$ into $L^{\infty}(\sigma)$, which is clear (only using $\mu \in \Sigma_{k}$), and that it also maps $L^1(\mu)$ into $L^{1,\infty}(\sigma)$:
\begin{equation}\label{form126}
\sigma(\{x \in X:\; \mathcal{M}_{\mu}f(x)>\lambda\})\leq\frac{C}{\lambda}\|f\|_{L^1(\mu)}, \qquad f \in L^{1}(\mu).
\end{equation}
This follows from the "standard" proof, and only uses that $\sigma \in \Sigma_{k}$, but to convince the reader that no Besicovitch covering theorem is needed, let us record the details. Fix $f \in L^{1}(\mu)$, and consider the ball family
\begin{displaymath} \mathcal{B} := \left\{B(x,r) \subset X : x \in \spt \sigma \text{ and } \frac{1}{r^{k}} \int_{B(x,r)} |f| \, d\mu > \lambda \right\}. \end{displaymath}
Since $f \in L^{1}(\mu)$, the radii of the balls in $\calB$ are uniformly bounded. Second, $\calB$ is a cover for the set $E = \{x \in \spt \sigma : \mathcal{M}_{\mu}f(x) > \lambda\}$, which has the same $\sigma$-measure as the left hand side of \eqref{form126}. Using the $5r$-covering theorem, we extract a countable disjoint subfamily $\mathcal{B}_{0} := \{B(x_{i},r_{i})\}_{i \in \N} \subset \mathcal{B}$ with $x_{i} \in \spt \sigma$, and
\begin{displaymath} E \subset \bigcup_{i \in \N} B(x_{i},5r_{i}). \end{displaymath}
Finally,
\begin{displaymath} \sigma(E) \leq \sum_{i \in \N} \sigma(B(x_{i},5r_{i})) \leq C\sum_{i \in \N} r_{i}^{k} \leq \frac{C}{\lambda} \sum_{i \in \N} \int_{B(x_{i},r_{i})} |f| \, d\mu \leq \frac{C}{\lambda} \|f\|_{L^{1}(\mu)}, \end{displaymath}
as claimed. \end{proof}

Lemma \ref{l:mixed} yields a ``two-measure statement'' for SIOs, Proposition \ref{l:mixedSIO} below.
We follow closely David's proof of  \cite[Proposition 2]{MR956767} and deduce Proposition \ref{l:mixedSIO} from two auxiliary lemmas.

\begin{lemma}\label{l:Cotlar1} Let  $(X,d)$ be a proper metric space, $k > 0$, and let $K \colon X \times X \, \setminus \bigtriangleup \to \C$ a $k$-GSK. Assume that $\sigma\in \Sigma_k$. Then there exists a constant $C > 0$, depending only on $k,\|K\|$, and $\mathrm{reg}(\sigma)$, such that
\begin{equation}\label{eq:Cotlar1}
T_{\sigma}^{\ast} f(x_0)\leq C \left(\mathcal{M}_{\sigma} (T_{\sigma}^{\ast} f)\right)(x_0)+ C \mathcal{M}_{\sigma} f(x_0),\quad f\in\mathcal{C}_c(X),\;x_0\in X.
\end{equation}
\end{lemma}

The main point is that we can take $x_{0} \in X \, \setminus \, \spt \sigma$.

\begin{proof}
One first shows that there exists a constant $C_0 > 0$, depending only on $k$ and $\|K\|$, such that for all $\varepsilon>0$ and $x_0\in X$,  one has
\begin{equation}\label{eq:LocalEstT}
|T_{\sigma,\varepsilon} f(x_0)|\leq T_{\sigma}^{\ast} f(x)+ C_0 \mathcal{M}_{\sigma}f(x_0),\quad x\in B(x_0,\varepsilon/2).
\end{equation}
This can be done as in the proof of \cite[Lemme 4]{MR956767}.

To show \eqref{eq:Cotlar1}, we fix $x_0 \in X$ and write $d:=\mathrm{dist}(x_0,\spt \sigma)$. The proof is divided in three cases, exactly as the proof of  \cite[Lemme 3]{MR956767}. First, if $\varepsilon\geq 4d$, then $\sigma(B(x_0,\varepsilon/2)) > 0$. Integrating \eqref{eq:LocalEstT} with respect to $\frac{1}{\sigma(B(x_0,\varepsilon/2))}d\sigma$ over $B(x_0,\varepsilon/2)$ and using the assumption $\sigma \in \Sigma_k$, we find a constant $C > 0$, depending only on $C_0$ in \eqref{eq:LocalEstT}, $k$, and $\mathrm{reg}(\sigma)$, such that
\begin{equation}\label{eq:TFirstCase}
|T_{\sigma,\varepsilon} f(x_0)|\leq C \left(\mathcal{M}_{\sigma} (T_{\sigma}^{\ast} f)\right)(x_0)+ C \mathcal{M}_{\sigma} f(x_0).
\end{equation}
Second, if $d/2 \leq \varepsilon < 4d$, then by \eqref{eq:TFirstCase} for $\varepsilon = 4d$ and the size estimate $|K(x_0,y)|\lesssim d(x_0,y)^{-k}$ on the annulus $B(x_0,4d)\setminus B(x_0,\varepsilon)$ yield again a bound of the form \eqref{eq:TFirstCase}. Third, if $\varepsilon<d/2$, then $T_{\sigma,\varepsilon}f(x_0)=T_{\sigma,d/2}f(x_0)$, and we are reduced to the second case.
\end{proof}

The next lemma is a Cotlar-type inequality. Such inequalities are available in very general settings, cf.\ \cite[I.7.3, Proposition 2]{stein1993harmonic},  \cite[p.56]{MR706075},  \cite[p.606]{Christ}, and \cite{MR1626935}, but we are not aware of one that would be precisely in the desired form for our purposes. In particular, we have to deal simultaneously with two measures $\mu$ and $\sigma$ in a metric space $(X,d)$.

\begin{lemma}\label{l:Cotlar2}  Let $(X,d)$ be a proper metric space, $k > 0$, and $\mu \in \Sigma_{k}$. Let $\bar{K} \colon X \times X \, \setminus \bigtriangleup \to \He$ be a bounded $k$-GSK, and let $\overline{T}$ be the operator induced by $(\bar{K},\mu)$. Let $\sigma \in \Sigma_{k}$ with regularity constant $C_{0} \geq 1$, and assume, for some $1 < s < \infty$, that
\begin{displaymath} A := \|\overline{T}\|_{L^{s}(\mu) \to L^{s}(\sigma)} < \infty. \end{displaymath}
Then, there exists a constant $C = C(A,C_{0},k,\|K\|,s)$\footnote{This constant does not depend on the regularity constant of $\mu$, so the assumption $\mu \in \Sigma_{k}$ is only made to ensure that $T$ is well-defined. It would suffice to assume that $\mu(B(x,r)) \lesssim r^{k}$ instead.} such that
\begin{equation}\label{eq:Cotlar2}
\overline{T}_{\mu}^{\ast} f(x)\leq C \left[ \mathcal{M}_{\sigma} (\overline{T}_{\mu} f)(x)+ \mathcal{M}_{\mu} f(x) + \left(\mathcal{M}_{\mu} |f|^{s}\right)^{\frac{1}{s}}(x) \right],\quad f \in C_{c}(X), \, x \in \spt \sigma.
\end{equation}
\end{lemma}

\begin{proof} The proof is verbatim the same as for \cite[Lemme 5]{MR956767}.
\end{proof}

\begin{proposition}\label{l:mixedSIO} Let $(X,d)$ be a proper metric space, $k > 0$, let $K \colon X \times X \, \setminus \, \bigtriangleup \to \C$ be a $k$-GSK, and let $\sigma \in \Sigma_k$. Assume that, for all $1<p<\infty$, there is a constant $C_{p} \geq 1$ such that
\begin{equation}\label{form128}
\|T_{\sigma}^{\ast}f\|_{L^p(\sigma)}\leq C_p \|f\|_{L^p(\sigma)},\quad f\in\mathcal{C}_c(X).
\end{equation}
Then for all $1<p<\infty$ and $\mu \in \Sigma_{k}$, there is a constant $C_{p}' \geq 1$ such that for all $f\in\mathcal{C}_c(X)$,
\begin{enumerate}
\item $\|T_{\sigma}^{\ast}f\|_{L^p(\mu)}\leq C_p' \|f\|_{L^p(\sigma)}$,
\item $\|T_{\mu}^{\ast}f\|_{L^p(\sigma)}\leq C_p' \|f\|_{L^p(\mu)}$.
\end{enumerate}
The constants $C_p'$ depend only on $p$, $C_p$, $k$, $\|K\|$, and $\mathrm{reg}(\mu),\mathrm{reg}(\sigma)$.
\end{proposition}

\begin{proof} Part (1) is a straightforward consequence Lemmas \ref{l:Cotlar1} and \ref{l:mixed}.

Part (2) is proved by duality. Fix $\mu \in \Sigma_{k}$, $1<p<\infty$, and let $q=p/(p-1)$. From the first part of the lemma, we know that the operators $T_{\sigma,\varepsilon}$ are uniformly bounded $L^q(\sigma) \to L^q(\mu)$. Now we define $K^{t}(x,y):= \overline{K(y,x)}$, and let $T^{t}_{\mu,\varepsilon}$ be the (adjoint) $\epsilon$-SIO induced by $(K_{\epsilon}^{t},\mu)$. Then,
\begin{displaymath}
\sup_{\varepsilon>0} \|T^{t}_{\mu,\varepsilon}\|_{L^p(\mu)\to L^p(\sigma)}\leq C_p.
\end{displaymath}
As an intermediate step towards (2), we wish to deduce from Lemma \ref{l:Cotlar2} the corresponding bound for the maximal SIO $T_{\mu}^{t,\ast}$. A small technical issue is that $K$ is not necessarily a bounded GSK, as required in the hypothesis (to even make sense of $\overline{T}$). To remedy this, fix $\epsilon > 0$, and note that $K^{t}_{\epsilon}$ is a bounded GSK, with GSK constants independent of $\epsilon$, by Lemma \ref{hormander}. Consequently, Lemma \ref{l:Cotlar2}, applied with $K^{t}_{\epsilon}$ and $s := \sqrt{p}$, implies that
\begin{equation}\label{form127} \|T_{\mu,\epsilon}^{t,\ast}f\|_{L^{p}(\sigma)} \lesssim \|T_{\mu,\epsilon}^{t}f\|_{L^{p}(\sigma)} + \|\mathcal{M}_{\mu}f\|_{L^{p}(\sigma)} + \|(\mathcal{M}_{\mu}|f|^{s})^{1/s}\|_{L^{p}(\sigma)} \lesssim \|f\|_{L^{p}(\mu)} \end{equation}
for $f \in C_{c}(X)$. Here $T^{t,\ast}_{\mu,\epsilon}$ is the maximal SIO associated to $K^{t}_{\epsilon}$, and we also used the $L^{p}(\mu) \to L^{p}(\sigma)$ and $L^{s}(\mu) \to L^{s}(\sigma)$ boundedness of $\mathcal{M}_{\mu}$ from Lemma \ref{l:mixed}, and the $L^{p}(\sigma) \to L^{p}(\sigma)$ boundedness of $\mathcal{M}_{\sigma}$. To proceed, we note that
\begin{displaymath} T^{t,\ast}_{\mu,\epsilon}f(x) = \sup_{\delta \geq \epsilon} |T^{t}_{\mu,\delta}f(x)|, \qquad f \in L^{p}(\mu), \, x \in X, \, \epsilon > 0, \end{displaymath}
so $T^{t,\ast}_{\mu,\epsilon}f(x) \nearrow T^{t,\ast}_{\mu}f(x)$ as $\epsilon \searrow 0$. Now, \eqref{form127} and monotone convergence yield
\begin{equation}\label{eq:IntermediateStep}
\|T^{t,\ast}_{\mu}f\|_{L^p(\sigma)} \lesssim \|f\|_{L^p(\mu)},\quad f\in\mathcal{C}_c(X). \end{equation}
This almost looks like (2), except that it concerns $T^{t}$ in place of $T$. However, applying \eqref{eq:IntermediateStep} to $\mu:=\sigma$, we conclude that also $K^{t}$ satisfies \eqref{form128}. Hence, we can re-run the whole argument with $K^{t}$! But since $(K^{t})^{t} = K$, this time we end up with (2). \end{proof}

Let us continue with the proof of Theorem \ref{t:AbstractBigPiece}. Fix $\mu \in \Sigma_{k}$ as in the statement, fix $1<p<\infty$, and let $f\in\mathcal{C}_{c}(X)$. Our task is to show that
\begin{equation}\label{eq:GoalIneqAbstrBP}
\|T_{\mu}^{\ast} f\|_{L^p(\mu)}\leq C_p \|f\|_{L^p(\mu)},\quad f\in \mathcal{C}_c(X).
\end{equation}
This will follow from Proposition \ref{l:GoodLambda} (``good $\lambda$ inequality'') applied to
\begin{equation}\label{eq:uANDv}
u:= T^{\ast}_{\mu} f\quad\text{and}\quad v:= \mathcal{M}_{\mu}f + \left((\mathcal{M}_{\mu} |f|^{\sqrt{p}})\right)^{\frac{1}{\sqrt{p}}}.
\end{equation}
 The rest of the proof consists of explaining how
 Proposition \ref{l:mixedSIO} can be used to verify that the assumptions of Proposition \ref{l:GoodLambda} are fulfilled.

\begin{lemma}\label{l:ProofuLpOutsideCpct} Let $(X,d)$ be a proper metric space, $k > 0$, and let $K \colon X \times X \, \setminus \, \bigtriangleup \, \to \C$ be a $k$-GSK. Let $\mu\in \Sigma_k$, $f\in\mathcal{C}_c(X)$ and $1<p<\infty$. Then $u:= T^{\ast}_{\mu} f$
is a Borel function on $(X,d)$ and it
agrees with an $L^p(\mu)$ function outside a ball, hence outside a set of finite $\mu$ measure.
\end{lemma}

\begin{proof} First we note that
\begin{equation}\label{eq:RationalSup}
T_{\mu}^{\ast} f(x) = \sup_{\varepsilon \in \Q\cap (0,+\infty)}|T_{\mu,\varepsilon} f(x)|.
\end{equation}
Indeed, for every  $\varepsilon \in (0,+\infty)$, there exists a sequence $(\varepsilon_j)_{j\in \mathbb{N}}\subset \mathbb{Q}$ with $\varepsilon_j \searrow \varepsilon$ as $j\to \infty$, and it follows that
\begin{displaymath}
|T_{\mu,\varepsilon}f(x)-T_{\mu,\varepsilon_j}f(x)| \leq \int_{\varepsilon<d(x,y)\leq \varepsilon_j} |K(x,y)f(y)|\,d\mu(y) \to 0\quad\text{as }j\to \infty.
\end{displaymath}
Since $T_{\mu,\varepsilon}f$ is a Borel function for every $\varepsilon>0$, we deduce from \eqref{eq:RationalSup} that $u$ is a Borel function.

Regarding the second claim,  if $\spt f \subset B(x_{0},R)$, the "size" condition for $K$ alone implies that $T^{\ast}_{\mu}f(x) \lesssim \mathcal{M}_{\mu,k}f(x)$ for $x \in X \, \setminus \, B(x_{0},2R)$. Now, the claim follows from the $L^{p}(\mu)$-boundedness of $\mathcal{M}_{\mu,k}$. \end{proof}

\begin{lemma}\label{l:ProofvLp} Let $(X,d)$ be a proper metric space, $k > 0$, $\mu \in \Sigma_k$, $f\in\mathcal{C}_c(X)$, and $1<p<\infty$. Then
\begin{displaymath}
v:= \mathcal{M}_{\mu,k} f + \left(\mathcal{M}_{\mu,k} |f|^{\sqrt{p}}\right)^{\frac{1}{\sqrt{p}}} \in L^{p}(\mu)
\end{displaymath}
with $\|v\|_{L^p(\mu)}\leq C \|f\|_{L^p(\mu)}$, where $C$ depends only on $p$ and $\mathrm{reg}(\mu)$.
\end{lemma}

\begin{proof} This follows from the boundedness of $\mathcal{M}_{\mu,k}$ on $L^{p}(\mu)$ and $L^{\sqrt{p}}(\mu)$.
\end{proof}

\begin{lemma}\label{l:ProofuANDv}
Assume that $(X,d)$, $k>0$, $K \colon X \times X \, \setminus \, \bigtriangleup \to \C$, and $\mu \in \Sigma_{k}$ are as in Theorem \ref{t:AbstractBigPiece}. Then there exists $\nu \in (0,1)$, depending only on $\mathrm{reg}(\mu)$ and the parameter $\theta > 0$, such that the following holds. Let $1<p<\infty$, $f\in\mathcal{C}_c(X)$, and define the functions $u$ and $v$ as in Lemmas \ref{l:ProofuLpOutsideCpct} and \ref{l:ProofvLp}. Then, for all $\varepsilon>0$, there is $\gamma = \gamma(\epsilon) >0$ such that
\begin{equation}\label{eq:GoalIneq}
\mu(\{x\in X: u(x) >\lambda + \varepsilon \lambda\text{ and }v(x)\leq \gamma \lambda \})\leq (1-\nu) \mu(\{x\in X:\; u(x)>\lambda\})
\end{equation}
for $\lambda > 0$. The choice of $\gamma$ is also allowed to depend on $p$, and the "data" of Theorem \ref{t:AbstractBigPiece}.
\end{lemma}

\begin{proof} The proof follows \cite[p.234ff]{MR956767} closely. The main difference is that $T^{\ast}_{\mu}f$ may not be lower semicontinuous when $K$ is only a \emph{generalised} standard kernel; this causes minor technical issues. Fix $\varepsilon,\lambda>0$ and abbreviate
\begin{displaymath}
\Omega:= \Omega_{\lambda}:=\{x\in \spt \mu :\; u(x) > \lambda\},
\end{displaymath}
and
\begin{displaymath}
A:= A_{\lambda,\varepsilon,\gamma}=\{x\in \spt \mu: u(x) >\lambda + \varepsilon \lambda\text{ and }v(x)\leq \gamma \lambda \}\subseteq \Omega.
\end{displaymath}
Our task is to ensure that $\mu(A)\leq (1-\nu)\mu(\Omega)$ for some $\nu = \nu(\mathrm{reg}(\mu),\theta) > 0$. We may evidently assume that $\mu(\Omega) > 0$.

We start by constructing a cover for $\Omega$. Since $f \in C_{c}(X)$, it follows from the "size" estimate $|K(x,y)|\lesssim \, d(x,y)^{-k}$, and from $\mu \in \Sigma_k$, that $T_{\mu}^{\ast}f(x) \to 0$ as $\dist(x,\spt f) \to \infty$. Hence $\Omega$ is a bounded set. On the other hand, for $\mu$ almost every $x \in \Omega$,
\begin{equation}\label{form131} \lim_{j \to \infty} \frac{\mu(B(x,2^{-j}) \cap \Omega)}{\mu(B(x,2^{-j}))} = 1, \end{equation}
by Lebesgue differentiation in the doubling metric measure space $(\spt \mu,\mu,d)$. Combining \eqref{form131} and the fact that $\Omega$ is bounded, it follows that for $\mu$ almost every $x \in \Omega$, there exists a maximal dyadic radius $r_{x} = 2^{-j_{x}} \lesssim_{\Omega,\mu} 1$, with $j_{x} \in \Z$, such that
\begin{equation}\label{form133} \frac{\mu(B(x,r_{x}) \cap \Omega)}{\mu(B(x,r_{x}))} \geq 1 - \frac{\theta}{2}. \end{equation}
In particular, since the reverse inequality already holds for $2r_{x}$, we can find
\begin{equation}\label{form130} a_{x} \in B(x,2r_{x}) \cap \Omega^{c}. \end{equation}
We then apply the $5r$-covering theorem to find a disjoint family $\{B(x_{i},r_{i})\}_{i \in \N} \subset \{B(x,r_{x}) : x \in \Omega\}$ with the property that $\mu$ almost all of $\Omega$ is contained in
\begin{displaymath} \bigcup_{i \in \N} B(x_{i},5r_{i}). \end{displaymath}
We write $B_{i} := B(x_{i},r_{i})$, $5B_{i} := B(x_{i},5r_{i})$, and $a_{i} := a_{x_{i}}$. In order to prove \eqref{eq:GoalIneq}, it suffices to show that
\begin{equation}\label{eq:GoalIneqCompl} \frac{\mu\left([B_i\cap \Omega] \, \setminus \, A \right)}{\mu(B_{i})} > \frac{\theta}{4}, \qquad i \in \N.
\end{equation}
This will establish \eqref{eq:GoalIneq}, because
\begin{align*}
\mu(\Omega \, \setminus \, A) & \geq \sum_i \mu([B_i\cap \Omega] \, \setminus \, A) > \tfrac{\theta}{4} \sum_{i \in \N} \mu(B_i) \gtrsim_{\mathrm{reg}(\mu),\theta} \sum_{i \in \N} \mu(5 B_i) \geq \mu(\Omega), \end{align*}
and consequently $\mu(A) \leq (1 - \nu)\mu(\Omega)$ for some $\nu = \nu(\mathrm{reg}(\mu),\theta) > 0$, as desired.

We then prove that \eqref{eq:GoalIneqCompl} holds if $\gamma = \gamma(\epsilon) > 0$ is chosen small enough (recall that $A = A_{\lambda,\epsilon,\gamma}$). For now, let $\gamma>0$ be arbitrary, and fix $B_i$. Note that \eqref{eq:GoalIneqCompl} is clear if $\nu(x) > \gamma \lambda$ for all $x \in B_{i}$ (then $[B_{i} \cap \Omega] \, \setminus \, A = B_{i} \cap \Omega$, which has density $\geq 1 - \theta/2 \geq \theta/2$), so we may assume that there exists a point $\xi_{i} \in B_i$ with
\begin{equation}\label{eq:v(xi)}
\mathcal{M}_{\mu}f(\xi_{i}) + \left(\mathcal{M}_{\mu} |f|^{\sqrt{p}}\right)^{\frac{1}{\sqrt{p}}}(\xi_{i})= v(\xi_{i})\leq \gamma \lambda.
\end{equation}
Now, we decompose $f = f_{1} + f_{2}$, where $f_{1} = f\phi$, and $\phi \in C_{c}(X)$ satisfies
\begin{displaymath} \mathbf{1}_{B(\xi_{i},10r_{i})} \leq \phi \leq \mathbf{1}_{B(\xi_{i},20r_{i})}. \end{displaymath}
Then
\begin{equation}\label{form129}
u(x) \leq T_{\mu}^{\ast}f_1(x)+ T_{\mu}^{\ast}f_2(x), \qquad x \in B_{i},
\end{equation}
and we will check in a moment that
\begin{equation}\label{eq:smallness_f2}
T_{\mu}^{\ast} f_2(x)\leq \lambda + \varepsilon \tfrac{\lambda}{2},\qquad x\in B_i, \end{equation}
if $\gamma = \gamma(\epsilon) > 0$ is small enough. Thus, \eqref{form129}-\eqref{eq:smallness_f2} imply that
\begin{displaymath}
\{x\in B_i:\; T_{\mu}^{\ast} f_1(x) \leq \varepsilon \tfrac{\lambda}{2}\}\subseteq B_i \, \setminus \, A,
\end{displaymath}
and the proof of \eqref{eq:GoalIneqCompl} has been reduced to showing that
\begin{equation}\label{eq:largeness_f1}
\mu(\{x\in B_i \cap \Omega :\; T_{\mu}^{\ast} f_1(x) \leq \varepsilon \tfrac{\lambda}{2}\})\geq \tfrac{\theta}{4} \mu(B_i).
\end{equation}
Before tackling \eqref{eq:largeness_f1}, we verify \eqref{eq:smallness_f2}. In fact, \eqref{eq:smallness_f2} follows from the chain
\begin{equation}\label{eq:ProofTf2Bdd}
T_{\mu}^{\ast}f_2(x) \leq T_{\mu}^{\ast}f(a_{i}) + C \mathcal{M}_{\mu}f(\xi_{i}) \leq   \lambda + C\gamma \lambda,\quad x\in B_i,
\end{equation}
by choosing $\gamma$ small enough so that $C\gamma \leq \varepsilon/2$. The second inequality in \eqref{eq:ProofTf2Bdd} follows from the choices of $a_{i} \in \Omega^{c}$ and $\xi_{i}$ in \eqref{eq:v(xi)}.
The first inequality can be obtained by writing $R_i := 10 r_{i}$, and decomposing
\begin{align*}| T_{\mu}f_{2}(x)| & = \left|\int K(x,y)[1 - \phi(y)]f(y) \, d\mu(y)\right|\\
& \leq \left| \int_{B(a_{i},R_{i})^{c}} K(a_{i},y)f(y) \, d\mu(y) \right| + \int |[\phi - \mathbf{1}_{B(a_{i},R_{i})}](y)||K(a_{i},y)||f(y)| \, d\mu(y)\\
& \quad + \int_{B(\xi_{i},R_{i})^{c}} |K(a_{i},y) - K(\xi_{i},y)||f(y)| \, d\mu(y)\\
& \quad + \int_{B(\xi_{i},R_{i})^{c}} |K(x,y) - K(\xi_{i},y)||f(y)| \, d\mu(y). \end{align*}
The first term is bounded by $T_{\mu}^{\ast}f(a_{i})$, as desired. The three latter ones are bounded by $\lesssim \mathcal{M}_{\mu}f(\xi_{i})$, using the GSK bounds of $K$, and recalling that $x,a_{i},\xi_{i} \in 2B_{i} \subset B(\xi_{i},R_{i}/2)$, and that $\phi|_{B(\xi_{i},R_{i})} = 1$. Similar, but slightly messier, estimates also work for $T_{\mu,\delta}$, $\delta > 0$, in place of $T_{\mu}$, so \eqref{eq:ProofTf2Bdd} has been confirmed.

Finally, we turn to \eqref{eq:largeness_f1}, which is based on the ``big piece'' assumption: there exists a measure $\sigma = \sigma_{B_{i}} \in \Sigma_{k}$, and a compact set $E\subseteq B_i \cap \spt \mu$ with the property that $\mu(E)\geq \theta \mu(B_i)$ and such that
\begin{equation}\label{form123}
\mu\left(\{x\in E:\; T_{\mu}^{\ast}f_1(x)>\tfrac{\varepsilon \lambda}{2}\}\right)\leq
\sigma\left(\{x\in X:\; T_{\mu}^{\ast}f_1(x)>\tfrac{\varepsilon \lambda}{2}\}\right).
\end{equation}
Since $\mu(\Omega \cap B_{i}) \geq (1 - \theta/2)\mu(B_{i})$, we moreover find that $\mu(E \cap \Omega) \geq (\theta/2)\mu(B_{i})$. We will show that $\gamma = \gamma(\epsilon) > 0$ can be chosen small enough so that the assumption \eqref{eq:v(xi)} implies that
\begin{equation}\label{eq:IntermediatGoal}
\mu\left(\{x\in E :\; T_{\mu}^{\ast}f_1(x) >\tfrac{\varepsilon \lambda}{2}\}\right) < \tfrac{\theta}{4}\mu(B_i).
\end{equation}
This of course yields \eqref{eq:largeness_f1}:
\begin{displaymath} \mu(\{x\in B_i \cap \Omega :\; T_{\mu}^{\ast} f_1(x) \leq \varepsilon \tfrac{\lambda}{2}\}) \geq \mu(\{x\in E \cap \Omega :\; T_{\mu}^{\ast} f_1(x) \leq \varepsilon \tfrac{\lambda}{2}\} \geq \tfrac{\theta}{4}\mu(B_{i}). \end{displaymath}
To prove \eqref{eq:IntermediatGoal}, start by combining \eqref{form123} with Chebyshev's inequality with $s := \sqrt{p}$:
\begin{equation}\label{form134}
\mu\left(\{x\in E:\; T_{\mu}^{\ast}f_1(x)>\tfrac{\varepsilon \lambda}{2}\}\right) \leq 2^s \varepsilon^{-s}\lambda^{-s} \|T_{\mu}^{\ast}f_1\|_{L^s(\sigma)}^s.
\end{equation}
To proceed, we plan to apply \eqref{eq:v(xi)}. By the hypothesis (4) of Theorem \ref{t:AbstractBigPiece}, $\|T_{\sigma}^{\ast}g\|_{L^{p}(\sigma)} \leq A_{p}\|g\|_{L^{p}(\sigma)}$ for all $g \in C_{c}(X)$. This is the assumption \eqref{form128} in Proposition \ref{l:mixedSIO}, so part (2) of that proposition yields
\begin{equation}\label{form135}
\|T_{\mu}^{\ast}f_1\|_{L^s(\sigma)}^{s} \leq C_s'\|f_1\|_{L^s(\mu)}^{s} \lesssim r_{i}^{k} \mathcal{M}_{\mu}(|f|^{s})(\xi_{i}) \stackrel{\eqref{eq:v(xi)}}{\leq} r_{i}^{k}\gamma^{s}\lambda^{s}.
\end{equation}
Combining \eqref{form134}-\eqref{form135}, we find that
\begin{displaymath}
\mu\left(\{x \in E:\; T_{\mu}^{\ast}f_1(x)>\tfrac{\varepsilon \lambda}{2}\}\right) \lesssim_{p} r_{i}^k\varepsilon^{-s} \gamma^s  \lesssim_{\mathrm{reg}(\mu)} \varepsilon^{-s} \gamma^s \mu(B_i).
\end{displaymath}
Choosing $\gamma > 0$ small enough, depending on $\theta$, $\varepsilon, p$, and $\mathrm{reg}(\mu)$, we conclude the proof of \eqref{eq:IntermediatGoal}, and therefore the lemma.
\end{proof}

We are now in possession of all ingredients necessary for the proof of Theorem \ref{t:AbstractBigPiece}.

\begin{proof}[Proof of Theorem \ref{t:AbstractBigPiece}] Lemmas \ref{l:ProofuLpOutsideCpct}, \ref{l:ProofvLp}, and \ref{l:ProofuANDv} show that Proposition \ref{l:GoodLambda} can be applied to the functions $u$ and $v$ as defined in \eqref{eq:uANDv}. This establishes \eqref{eq:GoalIneqAbstrBP}.
\end{proof}

\subsection{Regular curves and BPiLG}\label{s:BPiLG}
Recall that a closed set $E$ in $\He$ is \emph{$1$-regular} if there exists a finite constant $C\geq 1$ such that
\begin{equation}\label{eq:1_reg_set}
C^{-1} r \leq \mathcal{H}^1(B(p,r)\cap E)\leq C r,\quad\text{for all }p\in E,\,0<r\leq \mathrm{diam}E.
\end{equation}
The smallest constant $C\geq 1$ such that \eqref{eq:1_reg_set} holds will be denoted $\mathrm{reg}(E)$.

Recall further that a \emph{regular curve in $\He$} is a closed $1$-regular subset of $\He$ which has a Lipschitz parametrisation by an interval $I \subset \R$. In this section, we will use the letter "$\gamma$" for both the set, and the Lipschitz map $I \to \gamma$. A \emph{compact regular curve} is a regular curve parametrised by a compact interval $I \subset \R$.

\begin{definition}[Big pieces of intrinsic Lipschitz graphs]
A closed $1$-regular set $E\subset \He$ has \emph{big pieces of intrinsic Lipschitz graphs (over horizontal subgroups)} (BPiLG) if there exist constants $c,L>0$ such that for all $p\in E$ and all $0<r \leq \diam(E)$ there is an intrinsic $L$-Lipschitz graph $\Gamma \subset \He$ over some horizontal subgroup such that $\mathcal{H}^1(E\cap \Gamma \cap B(p,r)) \geq c r$.
\end{definition}

In this section, we prove the following:
\begin{thm}\label{p:BPiLG} Every regular curve in $\He$ has BPiLG.  \end{thm}

A short proof for the fact that regular curves in $\R^{n}$ have big pieces of $1$-dimensional Lipschitz graphs can be found in \cite[III.4]{MR1123480}. It is based on the rising sun lemma, and we did not find a way to adapt it to intrinsic Lipschitz graphs. Instead, we follow \cite{MR1132876}.

The proof of Theorem \ref{p:BPiLG} employs a system $\mathcal{D}$ of \emph{dyadic cubes} on a closed $1$-regular set $E \subset \He$, see \cite[Section 3.0.1]{MR3992573} for a more thorough introduction. These are Borel subsets of $E$ with the following properties:
\begin{itemize}
\item $\calD = \cup_{j} \calD_{j}$, $j \in \Z$, where each $\calD_{j}$ is a partition of $E$.
\item There exist $0 < c_{0} < C_{0} < \infty$, depending on $\mathrm{reg}(E)$, such that $\diam(Q) \leq C_{0}\ell(Q)$ for $Q \in \calD_{j}$, where $\ell(Q) := 2^{-j}$. For every $Q \in \calD_{j}$, there exists a "midpoint" $z_{Q} \in Q$ such that $E \cap B(z_{Q},c_{0}\ell(Q)) \subset Q$.
\end{itemize}
With this notation, we write $B_{Q} := B(z_{Q},2C_{0}\ell(Q))$, so that $Q \subset B_{Q}$ (with room to spare). For $Q\in\mathcal{D}$, we define the \emph{horizontal $\beta$-number}
\begin{displaymath}
\beta(Q) := \beta_{E}(Q):= \inf_{\ell \in \mathcal{L}}\sup_{q \in B_Q\cap E}\frac{\mathrm{dist}(q,\ell)}{\ell(Q)},
\end{displaymath}
where the infimum is taken over the horizontal lines familiar from Definition~\ref{d:horizontalLine},
\begin{displaymath}
\mathcal{L}:= \{p \cdot \mathbb{V}:\; p \in \mathbb{H},\,\mathbb{V}\text{ is a horizontal subgroup}\}.
\end{displaymath}
These numbers, notably their summability on horizontal curves, has been investigated extensively, see for example \cite{MR2789375,MR3456155} and the discussion in the introduction. Given a system $\mathcal{D}$ of dyadic cubes on a closed $1$-regular set $E$, we introduce the following subclass of \emph{good} cubes in $\mathcal{D}$:
\begin{definition}\label{d:GoodCube}
Let $E\subset \mathbb{H}$ be a closed $1$-regular set with a system $\mathcal{D}$ of dyadic cubes. Given $0<c,\varepsilon<1$ and a horizontal subgroup $\mathbb{V}$, we say that $Q \in \mathcal{D}$ is \emph{$(c,\varepsilon,\mathbb{V})$-good} if
\begin{enumerate}
\item $\mathcal{H}^1(\pi_{\mathbb{V}}(Q))\geq c \mathcal{H}^1(Q)$,
\item $\beta(Q) \leq \varepsilon.$
\end{enumerate}
\end{definition}
Here $\pi_{\V}$ is the horizontal projection introduced in Definition \ref{def1}. Recall also the cones $C_{\V}(\alpha)$ from Section \ref{s:IntrLip}. The next lemma shows that $(c,\varepsilon,\mathbb{V})$ good cubes $Q \in \calD$ look like intrinsic Lipschitz graphs over $\mathbb{V}$ at scale $\ell(Q)$.

\begin{lemma}\label{l:GeomPart}
Let $E\subset \mathbb{H}$ be a closed $1$-regular set with a system $\mathcal{D}$ of dyadic cubes. Then for all $c>0$ and $M \geq 2C_{0} \geq 1$, there exists $\alpha,\varepsilon >0$, depending only on $c$ and $M$, such that the following holds. If $Q\in \mathcal{D}$ is a $(c,\varepsilon,\mathbb{V})$-good cube, then
\begin{equation}\label{eq:GeomPart}
p\in Q, \,q\in B_Q\cap E\text{ and } d(p,q) \geq \ell(Q)/M \quad\Longrightarrow \quad p^{-1}\cdot q \notin C_{\V}(\alpha).
\end{equation}
\end{lemma}

\begin{proof} Using rotations around the $t$-axis, we may, without loss of generality, suppose that $\mathbb{V} = \{(x,0,0) : x \in \R\}$. Now, fix $c>0$ and $M \geq 2C_{0}$. We also fix arbitrary $\varepsilon,\alpha > 0$ at this point, and we fix a cube $Q \in \calD$ such that Definition \ref{d:GoodCube}(2) is satisfied, that is, $\beta(Q) \leq \varepsilon$. The plan is to show that if \eqref{eq:GeomPart} fails for some $p \in Q$ and $q \in B_{Q}$ with $d(p,q) \geq \ell(Q)/M$, and if $\alpha,\varepsilon > 0$ are small enough, then $Q$ cannot be a $(c,\varepsilon,\V)$-good cube, that is, $\calH^{1}(\pi_{\V}(Q)) < c\calH^{1}(Q)$. Since the constants in Definition \ref{d:GoodCube} are invariant under left translations and dilations, we may arrange that
 \begin{equation}\label{eq:q_pt_M}
 p=0\in Q \subset E \quad \text{and} \quad M^{-1} \leq d(0,q) \leq M.
 \end{equation}
We write in coordinates $q=(x,y,t)$, so that
 \begin{equation}\label{eq:ConeCond}
  q \in C_{\V}(\alpha) \quad \Longleftrightarrow \quad \|(x,0,0)\| \leq  \alpha \left\|\left(0,y,t-\tfrac{xy}{2}\right)\right\|. \end{equation}
 If $\alpha = \alpha_{M} > 0$ is sufficiently small, this implies, together with \eqref{eq:q_pt_M}, that $\|(0,y,t)\| \sim_{M} 1$. Next we will use $\beta(Q)\leq \varepsilon$ to infer that $t$ is small, and hence $q$ lies close to $\{(0,y,0) : y \in \R\}$. But since $Q$ lies close to the segment $[p,q] = [0,q]$, again by $\beta(Q) \leq \epsilon$, and $\pi_{\mathbb{V}}(\{(0,y,0) : y \in \R\}) = \{0\}$, this will eventually show that $\mathcal{H}^1(\pi_{\mathbb{V}}(Q)) < c \calH^{1}(Q)$.

 We turn to the details. Condition \eqref{eq:ConeCond} implies that
 \begin{equation}\label{eq:ConeCondExpl}
 |x|\leq \alpha\left(|y| +  \sqrt{|t|} +
  \frac{\sqrt{|x||y|}}{\sqrt{2}}\right).
 \end{equation}
 Now we consider two cases. If $|x|\leq |y|$, then \eqref{eq:ConeCondExpl} implies
 \begin{equation}\label{eq:ConeConcl}
 |x|\leq 2\alpha(|y|+ \sqrt{|t|}).
 \end{equation}
 On the other hand, if $|y|\leq |x|$, then \eqref{eq:ConeCondExpl} implies that
 \begin{displaymath}
 |x|\left(1-\alpha(1+ 1/\sqrt{2})\right)\leq \alpha\sqrt{|t|}
 \end{displaymath}
 and hence \eqref{eq:ConeConcl} holds true also in this case assuming, as we may, that $\alpha \leq 1/2$. Combined with the assumption that $d(q,0)\geq M^{-1}$, this shows that
 \begin{displaymath}
|y|+ \sqrt{|t|} \geq \frac{M^{-1}}{(1+2\alpha)}.
 \end{displaymath}
 To deduce more precise information about the coordinates of the point $q$, we use the assumption $\beta(Q)\leq \varepsilon$, which ensures the existence of a horizontal line $\ell = p_0 \cdot \mathbb{V}'$ with the property that
 \begin{displaymath}
 \mathrm{dist}(q',\ell) \leq 2 \varepsilon,\qquad q'\in B_Q \cap E.
 \end{displaymath}
 Thus there exist $(a,b)\in \mathbb{R}^2$, $a^2+b^2=1$, $p_0=(x_0,y_0,t_0)\in\mathbb{H}$, and $s\in \mathbb{R}$, such that
 \begin{equation}\label{eq:MetricBounds}
 \max\left\{d(q,p_0 \cdot (as,bs,0)),d(0,p_0)\right\} \leq 2 \varepsilon.
 \end{equation}
Triangle inequality, \eqref{eq:q_pt_M}, \eqref{eq:MetricBounds}, and left-invariance of the metric $d$ yield
 \begin{displaymath}
 M^{-1}- 4\varepsilon \leq d(p_0\cdot (as,bs,0),p_0)= |s| \leq M + 4\varepsilon.
 \end{displaymath}
 Take $4\varepsilon <M^{-1}$.
The estimates \eqref{eq:MetricBounds} then also imply that
 \begin{displaymath}
 |as+x|\leq |x_0|+|x_0+ as-x|\leq 4\varepsilon\quad \text{and}\quad |bs+y|\leq 4\varepsilon.
 \end{displaymath}
 By what we said before, this yields a non-trivial upper bound for $|a|$ (and lower bound for $|b|$):
 \begin{equation}\label{eq:a_bound}
 |a|\left(M^{-1}-4\varepsilon\right)\leq |a| |s| \leq 4\varepsilon + |x| \overset{\eqref{eq:ConeConcl},\eqref{eq:q_pt_M}}{\leq} 4\varepsilon + 2\alpha M. \end{equation}
 Returning to \eqref{eq:MetricBounds}, we have established that
 \begin{displaymath}
 d(q,p_0\cdot (as,bs,0))\leq 2\varepsilon,
 \end{displaymath}
 with $\|p_0\|\leq 2 \varepsilon$, $M^{-1}-4\varepsilon \leq |s|\leq M + 4\varepsilon$, and $(a,b)$ can be picked as close to $(0,1)$ as we like by choosing $\alpha,\epsilon > 0$ small enough. Recall that
 \begin{equation}\label{eq:LineNbhd}
 \{0,q\}\subseteq Q\subseteq B_Q\cap E \subseteq N(\ell,2\varepsilon)\cap B_Q\cap E,
 \end{equation}
 where $N(\ell,2\varepsilon)$ denotes the $2\varepsilon$-neighborhood of $\ell$ in the metric $d$. It follows from
 \eqref{eq:MetricBounds}, \eqref{eq:a_bound}, and \eqref{eq:LineNbhd} that
 \begin{displaymath}
 (x',y',t')\in Q \quad \Longrightarrow \quad |x'|\leq 2\varepsilon + \frac{4\varepsilon + 2\alpha M}{M^{-1}-4\varepsilon}+ 2\varepsilon.
 \end{displaymath}
The right hand side gives an upper bound for $\mathcal{H}^1(\pi_{\V}(Q))$ which tends to zero if $M$ is fixed, and $\alpha,\varepsilon \to 0$. For sufficiently small $\alpha,\varepsilon > 0$, we arrive at $\mathcal{H}^1(\pi_{\mathbb{W}}(Q)) < c$, and hence $Q$ is not a $(c,\varepsilon,\V)$-good cube. The proof is complete.
\end{proof}

The geometry of horizontal lines in $\He$ enters the proof of Theorem \ref{p:BPiLG} only through Lemma \ref{l:GeomPart}. With this result in hand, intrinsic Lipschitz graphs over horizontal subgroups can be constructed inside regular curves by an abstract coding argument, due to Jones \cite{MR1009121}. The construction requires to control the ``bad'' cubes of $\gamma$ that violate the second condition in Definition \ref{d:GoodCube}. For that purpose we first recall the following lemma, which  follows from \cite[Theorem I]{MR3456155}, and the observation in  \cite[Proposition 3.1]{MR3678492}.

\begin{lemma}[Weak geometric lemma (WGL)]\label{WGLLemma} Let $\gamma \subset \He$ be a compact regular curve, and let $\calD$ be a system of dyadic cubes on $\gamma$. Then for every $\varepsilon>0$, we have
\begin{equation}\label{WGL}
\sum_{\beta(Q)>\varepsilon,Q\subseteq Q_0} \ell(Q) \lesssim_{\mathrm{reg}(\gamma),\varepsilon} \ell(Q_{0}), \qquad Q_0\in \mathcal{D}.
\end{equation}
 \end{lemma}

In general, a closed $1$-regular set $E \subset \He$ satisfying \eqref{WGL} is said to \emph{satisfy the WGL}. This lemma is the only spot where we need compact regular curves; quite likely the WGL is true for all regular curves, but it has only been stated for compact ones in the literature.

\begin{thm}\label{t:main2_inv} Let $E\subseteq \He$ be a closed $1$-regular set satisfying the WGL, let $b >0$, and let $\mathbb{V} \subset \He$ be a horizontal subgroup. Then there exist $L \geq 1$ and $N\in \N$, depending only on $b, \mathrm{reg}(E)$, and the WGL constants of $E$, such that the following holds: for every $Q_0 \in \calD$, there exist intrinsic $L$-Lipschitz graphs $\Gamma_{1},\ldots,\Gamma_{N} \subset \He$ over $\mathbb{V}$ such that
\begin{displaymath} \mathcal{H}^1\left(\pi_{\V}\left(Q_0 \setminus \bigcup_{j=1}^N \Gamma_j\right)\right) \leq b \mathcal{H}^1(Q_0). \end{displaymath}
\end{thm}

With the geometric result from Lemma \ref{l:GeomPart} in hand, the proof of \ref{t:main2_inv} only uses the $1$-Lipschitz property of $\pi_{\V}$, and an abstract "coding argument", due to Jones \cite{MR1009121}, which has been applied to prove variants of Theorem \ref{t:main2_inv} for $k$-regular sets in $\mathbb{R}^d$ (\cite[Theorem 2.11]{MR1132876}) and for $(2n+1)$-regular sets in $\mathbb{H}^n$ (\cite[Theorem 3.9]{MR3992573} or \cite{2018arXiv180304819F}) satisfying natural analogues of the WGL property. The argument, and the notation, is nearly verbatim the same as in the proof of \cite[Theorem 3.9]{MR3992573}, so we refer there for details.

The conclusion of Theorem \ref{t:main2_inv} is only meaningful if $\mathcal{H}^1(\pi_{\mathbb{V}}(Q_0))$ is relatively large. If $\gamma \subset \He$ is a regular curve, then Lemma \ref{l:BigProj} below ensures that for every $Q_{0} \in \calD$, there exists a horizontal subgroup $\mathbb{V} \subset \He$ such that
\begin{equation}\label{eq:GoodProjCube}
\mathcal{H}^1(\pi_{\mathbb{V}}(Q_{0})) \gtrsim_{\mathrm{reg}(\gamma)} \ell(Q_{0}). \end{equation}
The enemy is the possibility $Q_{0} \subset \gamma$ "wraps tightly around a vertical line", so that it projects to a set of small $\mathcal{H}^1$ measure on the $xy$-plane, and in particular on \emph{every} horizontal subgroup $\mathbb{V}$. Yet, heuristically, the regular curve $\gamma$ simply cannot resemble a vertical line that much. This eventually gives the existence of $\mathbb{V}$ such that \eqref{eq:GoodProjCube} holds.

\begin{lemma}\label{l:BigProj}
Let $\gamma \subset \He$ be a regular curve. Then $\gamma$ has \emph{big horizontal projections}, which means the following. There exists a constant $c \gtrsim_{\mathrm{reg}(\gamma)} 1$ such that such for all $p_0\in \gamma$ and all $0<r \leq \mathrm{diam}(\gamma)$, there is a horizontal subgroup $\mathbb{V} \subset \He$ such that
\begin{equation}\label{BHP}
\mathcal{H}^1(\pi_{\mathbb{V}}(\gamma\cap B(p_0,r)))\geq c r.
 \end{equation}
\end{lemma}

\begin{proof}[Proof of Lemma \ref{l:BigProj}] Let $\gamma \subset \He$ be a regular curve parametrised by an interval $I \subset \R$. Write $\pi \colon \He \to \R^{2}$ for the projection map $\pi(x,y,t) = (x,y)$. Fix a point $p_{0} \in \gamma$, and a radius $0 < r < \kappa \diam(\gamma)$ for a suitable small absolute constant $\kappa > 0$ (if $\diam(\gamma) = \infty$, there is no restriction for $r > 0$). Consider then the projection $\gamma_{\pi} := \pi(\gamma) \subset \R^{2}$, and write $\gamma_{\pi}(s) := \pi(\gamma(s))$ for $s \in I$.

Assume without loss of generality that $p_{0} = \gamma(0) = 0$. Since $r < \kappa\diam(\gamma)$, there exists another point $p_{1} = \gamma(s_{1}) \in \gamma$ with $\|p_{1}\| \geq r/\kappa$. We choose the smallest parameter $s_1>0$ with this property, and we restrict attention to considering $\gamma|_{[0,s_{1}]}$ and $\gamma_{\pi}|_{[0,s_{1}]}$. We claim that if $\kappa > 0$ was chosen small enough, depending on $\mathrm{reg}(\gamma)$, then there exists a point $s \in [0,s_{1}]$ with the property that
\begin{equation}\label{form4} |\gamma_{\pi}(s)| = r. \end{equation}
We only have to exclude the possibility that
the projection $\gamma_{\pi}|_{[0,s_{1}]}$ stays inside the open disc $U(0,r)$.
 To see this, assume that \eqref{form4} fails for all $0 < s \leq s_{1}$. We assume, for example, that the third component $t_{1}$ of $p_{1} = \gamma(s_{1})$ is strictly positive. Now comparing the conditions
\begin{displaymath} |\pi(p_{1})| = |\gamma_{\pi}(s_{1})| < r \quad \text{and} \quad \|p_{1}\| \geq r/\kappa \end{displaymath}
in fact shows that $\sqrt{t_{1}} \gtrsim r/\kappa$, hence
\begin{equation}\label{form5} t_{1} \gtrsim \frac{r^{2}}{\kappa^{2}}. \end{equation}
To proceed, cover the box $U(0,r) \times [0,t_{1}] \subset \He$ with boundedly overlapping balls of radius $2r$ centred on the $t$-axis or, equivalently, with vertical translates of the box $U(0,2r) \times [-4r^{2},4r^{2}]$. According to \eqref{form5}, the required number of such boxes is $\sim t_{1}/r^{2}$. Moreover, since $\gamma|_{[0,s_{1}]}$ is a continuum satisfying $|\gamma_{\pi}(s)|<r$ for all $s\in [0,s_1]$, and $\gamma(s_{1}) = p_{1}$, it must in fact meet $\gtrsim t_{1}/r^{2}$ of the slightly smaller boxes of the type $U(0,r) \times [-r^{2},r^{2}]$. Finally, by the $1$-regularity of $\gamma$, we have
\begin{displaymath} \gamma \cap [U(0,r) \times [-r^{2},r^{2}]] \neq \emptyset \quad \Longrightarrow \quad \calH^{1}(\gamma \cap [U(0,2r) \times [-4r^{2},4r^{2}]]) \sim r. \end{displaymath}
Since also $\sqrt{t_{1}}$ is much larger than $r$, we on the other hand observe that $U(0,2r) \times [0,t_{1}]$ is covered by the single "$\sqrt{t_{1}}$-ball" $B_{\sqrt{t_{1}}} := U(0,\sqrt{t_{1}}) \times [0,t_{1}]$. This gives us the two-sided estimate
\begin{displaymath} \frac{t_{1}}{r} = \frac{t_1}{r^2}\, r \lesssim \calH^{1}(\gamma \cap [U(0,2r) \times [0,t_{1}]]) \leq \calH^{1}(\gamma \cap B_{\sqrt{t_{1}}}) \lesssim \sqrt{t_{1}}, \end{displaymath}
hence $t_{1} \lesssim r^{2}$. This violates \eqref{form5} for $\kappa > 0$ small enough, and the proof of \eqref{form4} is complete.

Now, we let $s_{0} \in [0,s_{1}]$ be the first parameter such that \eqref{form4} holds, and we also recall that $\gamma(s) \in B(0,r/\kappa)$ for all $s\in [0,s_1]$. Then
\begin{displaymath}
\{0,\gamma_{\pi}(s_0)\}\subseteq \gamma_{\pi}([0,s_0])\subseteq \pi(B(0,r/\kappa)).
\end{displaymath}
Let $\mathbb{V}$ be the horizontal subgroup containing $\gamma_{\pi}(s_0)$. Then, since $\gamma|_{[0,s_{0}]} \subset B(0,r/k)$ is a connected set containing $p_0=0$ and $\gamma(s_0)$, we have
\begin{displaymath}
\mathcal{H}^1(\pi_{\mathbb{V}}(\gamma\cap B(0,r/\kappa)))\geq \calH^{1}([0,\gamma_{\pi}(s_{0})]) = r.
\end{displaymath}
This shows that \eqref{BHP} holds with $c = \kappa$, and the proof is complete. \end{proof}

We then put the pieces together to prove Theorem \ref{p:BPiLG}.
\begin{proof}[Proof of Theorem \ref{p:BPiLG}] Let $\gamma \subset \He$ be a regular curve. Fix $p \in \gamma$ and $0 < r \leq \diam(\gamma)$. Start by choosing a compact regular curve $\gamma_{0} \subset \gamma$ with $\mathrm{reg}(\gamma_{0}) \lesssim \mathrm{reg}(\gamma)$, which contains $p$, and satisfies $\diam(\gamma_{0}) \geq r$. Then $\gamma_{0}$ satisfies the WGL by Lemma \ref{WGLLemma}, and, on the other hand, Lemma \ref{l:BigProj} gives a horizontal subgroup $\V \subset \He$ such that $\mathcal{H}^{1}(\pi_{\V}(B(p,r) \cap \gamma_{0})) \geq cr$, where $c \gtrsim_{\mathrm{reg}(\gamma)} 1$ (to be precise, use the version \eqref{eq:GoodProjCube} for a dyadic cube $Q_{0} \subset B(p,r) \cap \gamma_{0}$ with $\ell(Q_{0}) \sim r$). Finally, apply Theorem \ref{t:main2_inv} to $\gamma_{0}$, with parameter $b = c/2$, and use the $1$-Lipschitz property of $\pi_{\V}$ to deduce that $\mathcal{H}^{1}(\gamma_{0} \cap \Gamma_{i}) \gtrsim c/N$ for some $1 \leq i \leq N$. Since $N$ only depends on the WGL and $1$-regularity constants of $\gamma_{0}$ (both of which are uniform), the proof is complete. \end{proof}

\subsection{Singular integrals on regular curves}\label{s:conclusion} It is now easy to put the pieces together to arrive at the main result, Theorem \ref{t:mainRegularCurve}, which stated that good kernels are CZ kernels for regular curves in $\He$.
\begin{proof}[Proof of Theorem \ref{t:mainRegularCurve}] Let $\gamma \subset \He$ be a regular curve. Then $\gamma$ is contained in an unbounded regular curve (attach horizontal half-lines if necessary). Since it suffices to prove the boundedness of any SIO on the extension, we may assume that $\diam(\gamma) = \infty$ to begin with. Therefore, $\mu := \calH^{1}|_{\gamma} \in \Sigma_{1}$ in the sense of Definition \ref{d:Sigma_k}. By Theorem \ref{p:BPiLG}, moreover, $\gamma$ has BPiLG. This means that, for every ball $B$ centred on $\gamma$, there exists an intrinsic Lipschitz graph $\Gamma_{B}$ with $\mu(\Gamma_{B}) \geq \theta \mu(B)$ (with $\mathrm{reg}(\Gamma_{B})$ uniformly bounded). By Proposition  \ref{p:lipext} (extension of intrinsic Lipschitz graphs), we may moreover arrange that $\Gamma$ is unbounded, and $\sigma_{B} := \calH^{1}|_{\Gamma_{B}} \in \Sigma_{1}$ (with $\mathrm{reg}(\sigma_{B})$ uniformly bounded from above).

Now, let $k \colon \He \, \setminus \, \{0\} \to \C$ be a good kernel, and write $K(p,q) := k(q^{-1} \cdot p)$. We already know, by Theorem \ref{t:mainIntrLipGraph} and Remark \ref{r:maximalRemark}, that the maximal SIO $T_{\sigma_{B}}^{\ast}$ induced by $(K,\sigma_{B})$ is bounded on $L^{p}(\sigma_{B})$, $1 < p < \infty$, with constants independent of the choice of $B$. Therefore, the hypotheses of Theorem \ref{t:AbstractBigPiece} are met for $K$ and $\mu$, and \eqref{eq:ConclAbstrThm} implies that $K$ is a CZ kernel for $\mu$, as claimed in Theorem \ref{t:mainRegularCurve}. \end{proof}

The proof of Theorem \ref{t:mainCL} for regular curves can be completed in the same manner, since we already established it for intrinsic Lipschitz graphs over horizontal subgroups in Theorem \ref{t:ChousionisLiKernel}.

\section{Singular integrals on Lipschitz flags}\label{s:SIOOnFLagSUrfaces}

A \emph{Lipschitz flag} $\mathcal{F} \subset \He$, or just a \emph{flag}, is a set of the form
 $\calF = \{(A(y),y,t) : y,t \in \R\}$, where $A \colon \R \to \R$ is Lipschitz.
 Flags are, in particular, co-dimension $1$ intrinsic Lipschitz
 graphs in the sense of Franchi, Serapioni, and Serra Cassano.
 Indeed, writing $\W := \{(0,y,t) := y,t \in \R\}$, $\V := \{(x,0,0) : x \in \R\}$, and $\varphi(0,y,t) := (A(y),0,0)$, then
 $\varphi:\W \to \V$ is intrinsic Lipschitz according to the definition in \cite{FSS}, and $\calF = \{w \cdot \varphi(w) : w \in \W\}$. In particular, flags are closed $3$-regular subset of $\He$. Here, we apply the $1$-dimensional theory to prove the following result about $3$-dimensional singular integrals on flags:
\begin{thm}\label{t:flags} Let $K \in C^{\infty}(\He \, \setminus \, \{0\})$ be a horizontally odd kernel satisfying
\begin{equation}\label{kernelConstants} |\nabla_{\He}^{n}K(p)| \leq C_{n}\|p\|^{-3 - n}, \qquad p \in \He, \, n \geq 0, \end{equation}
for some constants $C_{n} > 0$. Then $K$ is a CZ kernel for $\calH^{3}$ restricted to any Lipschitz flag in $\He$. \end{thm}
There are two key features of flags which we need in the argument. First, since flags are foliated by vertical lines, one can apply Fourier analysis in the $t$-variable. Second, as opposed to more general intrinsic Lipschitz graphs, flags admit a -- fairly "canonical" -- bilipschitz parametrisation by the plane $\W$. In fact, fix a flag $\calF = \{(A(y),y,t) : y,t \in \R\}$. Consider the (horizontal) curve $\gamma \subset \calF$ given by
\begin{equation}\label{eq:gammaA} \gamma(y) :=
\left(A(y),y,\int_{0}^{y} A(r)\, dr -\tfrac{1}{2} y A(y) \right), \qquad y \in \R, \end{equation}
and the map $\Gamma \colon \W \to \calF$,
\begin{equation}\label{eq:GammaA}
\Gamma(0,y,t) := \gamma(y) \cdot (0,0,t) =
\left(A(y),y,t-\tfrac{1}{2} y A(y)+\int_{0}^{y} A(r)\, dr \right).
\end{equation}
In fact, $\gamma$ is the graph map of the intrinsic Lipschitz function $\phi \colon \V_{y} \to \W_{xt}$,
\begin{displaymath} \phi(0,y,0) := \left( A(y),0,\int_{0}^{y} A(r) \, dr \right), \qquad y \in \R,  \end{displaymath}
mapping the horizontal subgroup $\mathbb{V}_{y} = \{(0,y,0) : y \in \R\}$ to the vertical subgroup $\W_{xt} = \{(x,0,t) : x,t \in \R\}$. So, $\Gamma$ lifts the foliation of $\W$ by horizontal lines
to a foliation of $\calF$ by $1$-dimensional intrinsic Lipschitz graphs over $\mathbb{V}_{y}$. Note that
$\Gamma$ is not the usual intrinsic graph map $w \mapsto w \cdot
\varphi(w)$, which is virtually never Lipschitz.
\begin{lemma}\label{l:area} If $A$ is $L$-Lipschitz, $L > 0$, then the map $\Gamma \colon \W \to \calF$ is $\sim (1 + L)$-bilipschitz, and one has the following area formula for
the spherical Hausdorff measure $\sigma =
\mathcal{S}^{3}|_{\calF}$:
\begin{equation}\label{eq:area_formula_3D} \int f(p) \, d\sigma(p) = c\iint f(\Gamma(y,t))\sqrt{1 + A'(y)^{2}} \, dy \, dt, \qquad f \in L^{1}(\sigma),  \end{equation}
where $c > 0$ is an absolute constant.
\end{lemma}

\begin{proof} To see that $\Gamma$ is bilipschitz with respect to
the Heisenberg metric, we first  compute
\begin{equation}\label{eq:gamma_formula} \gamma(y')^{-1} \cdot \gamma(y) = \left( A(y) -  A(y'),y- y', \int_{y}^{y'}
\left[ \frac{A(y) + A(y') - 2A(r)}{2} \right] \, dr \right)
\end{equation}
for $y,y'\in \mathbb{R}$. Since points on the $t$-axis commute
with all other elements of $\He$, this yields
\begin{align*}
d(\Gamma(0,y,&t),\,\Gamma(0,y',t'))= \|\Gamma(0,y',t')^{-1}\cdot
\Gamma(0,y,t)\|\\& \sim |A(y)-A(y')| + |y-y'| + \left|t-t'+
\int_{y}^{y'} \left[ \frac{(A(y)-A(r)) + (A(y') - A(r))}{2}\right]
\, dr \right|^{\frac{1}{2}}
\end{align*}
for $(0,y,t),(0,y',y')\in \W$. Using that $A$ is $L$-Lipschitz, we deduce that
\begin{displaymath}
d(\Gamma(0,y,t),\Gamma(0,y',t')) \lesssim (1 + L)d((0,y,t),(0,y',t'))
\end{displaymath}
and
\begin{align*}
d((0,y,t),(0,y',t'))& \lesssim |y-y'|+ \left|t-t'+ \int_{y}^{y'}
\left[ \frac{(A(y)-A(r)) + (A(y') - A(r))}{2}\right] \, dr
\right|^{\frac{1}{2}}\\&\qquad + \left|\int_{y}^{y'} \left[
\frac{(A(y)-A(r)) + (A(y') - A(r))}{2}\right] \, dr
\right|^{\frac{1}{2}}\\
&\lesssim d(\Gamma(0,y,t),\Gamma(0,y',t')) + L^{1/2}|y - y'|  \lesssim (1 + L)d(\Gamma(0,y,t),\Gamma(0,y',t')).
\end{align*}
Next, to prove the area formula \eqref{eq:area_formula_3D},
we use that $\sigma$ equals the Euclidean Hausdorff measure
$\mathcal{H}^2|_{\mathcal{F}}$ up to a multiplicative constant $c$
by \cite[Proposition 2.14 + Corollary 7.7]{FSSC}, the surjectivity of the
parametrisation $\Gamma:\mathbb{W}\triangleq \mathbb{R}^2 \to
\mathcal{F}$, and the Euclidean area formula
\begin{align*}
\int f(p) \, d\sigma(p) = c \int_{\mathcal{F}} f(p)\,
d\mathcal{H}^2(p)&=c
\iint f(\Gamma(y,t))|J_{\Gamma}(y,t)|\,dy\,dt.\end{align*}
Here
\begin{displaymath}
|J_{\Gamma}(y,t)|= \sqrt{\det (D\Gamma)^T
(y,t)D\Gamma(y,t)}=\sqrt{1 + A'(y)^{2}},
\end{displaymath}
as can be verified by a straightforward computation, so the proof is complete.
 \end{proof}

Arguing as in Section \ref{s:Exp1}, we may use the existence of a
bilipschitz parametrisation to reduce the proof of Theorem
\ref{t:flags} to a statement concerning the parametric (standard)
kernel $K_{\Gamma}(w,v) := K(\Gamma(v)^{-1} \cdot \Gamma(w))$. We
record this statement separately, since -- as in Theorem
\ref{t:technical} -- we want to make a slightly stronger claim.
\begin{thm}\label{t:flagTechnical}
Let $A,B \colon \R \to \R$ be Lipschitz functions, and define $\Gamma = \Gamma_{A}$ as in \eqref{eq:GammaA}. If $K \in
C^{\infty}(\He \, \setminus \, \{0\})$ is a horizontally odd
kernel satisfying
 the hypothesis \eqref{kernelConstants}, then the kernels
$K_{\Gamma}D_{B,1}$ and $K_{\Gamma}D_{B,2}$ are CZ kernels on
$\W$, where
\begin{displaymath} D_{B,1}(w,v) := \frac{B(x) - B(y)}{x - y} \quad \text{and} \quad D_{B,2}(w,v) := \int_{x}^{y} \frac{B(x) + B(y) - 2B(r)}{2(x - y)^{2}} \, dr. \end{displaymath}
for $w = (x,t)$ and $v = (y,s)$. \end{thm} Theorem \ref{t:flags}
follows from Theorem \ref{t:flagTechnical} by taking $B(x) = x$, so $K_{\Gamma}D_{B,1} = K_{\Gamma}$.

\begin{proof} The proof is a reduction to Theorem
\ref{t:technical}. Note, first, that the factors $D_{B,1}$ and $D_{B,2}$ above are exactly of the form 
\begin{displaymath} D_{A_{0}}(x,y) := \tfrac{A_{0}(x) - A_{0}(y)}{x - y} \quad \text{and} \quad D_{B_{0}}(x,y) := \tfrac{B_{2}(x) - B_{2}(y) - \tfrac{1}{2}[B_{1}(x) + B_{1}(y)](x - y)}{(x - y)^{2}} \end{displaymath}
appearing in the statement of Theorem \ref{t:technical}, with the choices $A_{0} := B$ and $B_{0} = (B,B_{2})$, where $B_{2}(y) = \int_{0}^{y} B(r) \, dr$. Since $B = \dot{B}_{2}$, we see that $B_{0}$ is a tame map by definition. 

The reduction from the $3$-dimensional kernel $K_{\Gamma}(v,w)$ to $1$-dimensional kernels of the form $k(\Phi^{-1}(y) \cdot \Phi(x))$ (as in Theorem \ref{t:technical}) is slightly more involved. The idea is to take Fourier transforms in the $t$-variable, which will reduce our question about the single $3$-dimensional kernel $K_{\Gamma}$ to family of questions regarding the $1$-dimensional kernels appearing in \eqref{form158}. We learned this trick from the paper \cite{MR545307} of Fabes and and Rivi\`ere. The $1$-dimensional questions will eventually be solved by applying Theorem \ref{t:technical}.

For technical convenience, we assume in the following that $K$ is supported away
from the origin to begin with; this allows us to ignore standard
issues of truncations, and all integrals below will be absolutely
convergent. The considerations for $K_{\Gamma}D_{B,1}$ and
$K_{\Gamma}D_{B,2}$ are extremely similar, so we record the full
details only for the latter. Since $D_{B,2}(w,v)$ only depends on
the $x$ and $y$ coordinates of $w = (x,t)$ and $v = (y,s)$, we
write $D_{B,2}(w,v) =: D_{2}(x,y)$. Then, we set
\begin{displaymath} Rf(w) := \int K_{\Gamma}(w,v)D_{2}(x,y)f(v) \, dv, \qquad f \in L^{2}(\W). \end{displaymath}
By Plancherel,
\begin{displaymath} \int_{\W} |Rf(w)|^{2} \, dw = \iint |Rf(x,t)|^{2} \, dt \, dx = \iint |\widehat{Rf}(x,\tau)|^{2} \, d\tau \, dx, \end{displaymath}
where $\widehat{Rf}(x,\tau)$ refers to the \emph{vertical Fourier
transform}, that is, the Fourier transform of $t \mapsto Rf(x,t)$
evaluated at $\tau$. To show that $\|Rf\|_{L^{2}(\W)} \lesssim
\|f\|_{L^{2}(\W)}$, it will evidently suffice to verify that
\begin{equation}\label{form156} \int |\widehat{Rf}(x,\tau)|^{2} \, dx \lesssim \int |\hat{f}(x,\tau)|^{2} \, dx, \qquad \tau \neq 0. \end{equation}
For notational convenience, we only consider $\tau > 0$; in the case $\tau < 0$, several absolute values signs need to be added. Let us compute an expression for $\widehat{Rf}(x,\tau)$ at $x,\tau \in \R$:
\begin{align} \widehat{Rf}(x,\tau) & = \int_{\R} e^{-2\pi i t\tau} Rf(x,t) \, dt \notag\\
& = \int e^{-2\pi i t\tau} \iint K(\gamma(y)^{-1} \cdot \gamma(x) \cdot (0,0,t - s))D_{2}(x,y)f(y,s) \, dy \, ds \, dt \notag\\
& = \iint f(y,s)D_{2}(x,y) \left[ \int e^{-2\pi i t\tau} K(\gamma(y)^{-1} \cdot \gamma(x) \cdot (0,0,t - s)) \, dt  \right] \, ds \, dy \notag\\
& = \iint f(y,s)D_{2}(x,y) \left[ \int e^{-2\pi i (u + s)\tau} K(\gamma(y)^{-1} \cdot \gamma(x) \cdot (0,0,u)) \, du \right] \, ds \, dy \notag\\
&\label{form157} = \int \hat{f}(y,\tau)D_{2}(x,y) \left[ \int e^{-2\pi i u\tau} K(\gamma(y)^{-1} \cdot \gamma(x) \cdot (0,0,u)) \, du \right] \, dy,  \end{align}
where $\hat{f}(y,\tau)$ is the Fourier transform of $s \mapsto f(y,s)$ evaluated at $\tau$. To proceed modifying the expression on line \eqref{form157},
we need to introduce auxiliary kernels. For $\tau > 0$ fixed, write
\begin{displaymath} (\delta_{\tau}K)(x,y,t) := \tau^{-3}K(\delta_{\tau^{-1}}(x,y,t)). \end{displaymath}
The kernels $\delta_{\tau}K$ are horizontally odd, and using the
the chain rule, one observes that they satisfy the decay estimates
\eqref{kernelConstants} with the same constants "$C_{n}$" as $K$
(independently of $\tau > 0$). Motivated by \eqref{form157}, we
then define the final auxiliary kernels
\begin{equation}\label{form158} k_{\tau}(p) := \int e^{-2\pi i \theta} (\delta_{\tau}K)(p \cdot (0,0,\theta))\,d\theta, \qquad p \in \He \, \setminus \, \{|z| = 0\}. \end{equation}
The kernels $k_{\tau}$ are horizontally odd and weakly good:
\begin{lemma}\label{l:calK} Let $K \in C^{\infty}(\He \, \setminus \, \{0\})$ be a horizontally odd kernel satisfying the $3$-dimensional decay assumptions \eqref{kernelConstants} with constants "$C_{n}$". Then, the kernel $k \in C^{\infty}(\He \, \setminus \{|z| = 0\})$,
\begin{displaymath} k(p) := \int e^{-2\pi i \theta} K(p \cdot (0,0,\theta)) \, d\theta \end{displaymath}
is horizontally odd, and satisfies the weak goodness hypothesis
\begin{displaymath} |\nabla_{\mathbb{H}}^nk(p)| \leq c_{n}|z|^{-n - 1}, \qquad p = (z,t) \in \He \, \setminus \, \{|z| = 0\}, \end{displaymath}
with constants $c_{n} \lesssim C_{n}$.
\end{lemma}

\begin{proof} The horizontal oddness follows from
\begin{displaymath} K((-z,t) \cdot (0,0,\theta)) = K(-z,t + \theta) = -K(z,t + \theta) = -K((z,t) \cdot (0,0,\theta)). \end{displaymath}
To obtain the decay estimates \eqref{eq:weakGood}, we note that $p
\cdot (0,0,\theta) = (0,0,\theta) \cdot p$, so there is no problem
with commuting horizontal derivatives and the
$\theta$-integration. We obtain
\begin{align*} |\nabla_{\He}^{n}k(p)| & = \left| \int e^{-2\pi i \theta} (\nabla_{\He}^{n}K)(p \cdot (0,0,\theta)) \, d\theta \right|\\
& \leq C_{n}\int \frac{d\theta}{(|z|^{4} + (\theta + t)^{2})^{(n + 3)/4}}\\
& = C_{n}|z|^{2} \int \frac{du}{(|z|^{4} + |z|^{4}u^{2})^{(n +
3)/4}} = \frac{C_{n}}{|z|^{n + 1}} \left(\int \frac{du}{(1 +
|u|^{2})^{(n + 3)/4}} \right).   \end{align*} The last integral is
bounded (for $n \geq 0$) by an absolute constant, so the claim
follows.  \end{proof}

Now, we may write
\begin{align*} \int e^{-2\pi i u \tau} & K(p \cdot (0,0,u)) \, du =
 \sqrt{\tau} \int e^{-2\pi i \theta} (\delta_{\sqrt{\tau}}K)(\delta_{\sqrt{\tau}}(p) \cdot (0,0,\theta))\,d\theta = \rho \cdot k_{\rho}(\delta_{\rho}(p)), \end{align*}
where $\rho := \sqrt{\tau}$, and hence
\begin{displaymath} \eqref{form157} = \rho\int \hat{f}(y,\tau)D_{2}(x,y)k_{\rho}(\delta_{\rho}(\gamma(y)^{-1} \cdot \gamma(x))) \, dy. \end{displaymath}
Now, we plug the RHS back into the LHS of \eqref{form156}, and change variables in both $x$ and $y$ to find
\begin{align} \eqref{form156} & = \rho^{2} \int \left| \int \hat{f}(y,\tau)D_{2}(x,y)k_{\rho}(\delta_{\rho}[\gamma(y)^{-1} \cdot \gamma(x)]) \, dy \right|^{2} \, dx \notag\\
&\label{form159} = \frac{1}{\rho} \int \left| \int
\hat{f}\left(\tfrac{y}{\rho},\tau
\right)D_{2}\left(\tfrac{x}{\rho},\tfrac{y}{\rho}
\right)k_{\rho}(\delta_{\rho}[\gamma(\tfrac{y}{\rho})^{-1}
\cdot \gamma(\tfrac{x}{\rho})]) \, dy \right|^{2} \, dx.
\end{align} Recalling \eqref{eq:gamma_formula}, note that
\begin{align*} \delta_{\rho}(\gamma(\tfrac{y}{\rho})^{-1} \cdot \gamma(\tfrac{x}{\rho}))
& = \left(\rho A(\tfrac{x}{\rho}) - \rho A(\tfrac{y}{\rho}),x - y, \rho^{2}\int_{x/\rho}^{y/\rho}
 \left[ \frac{A(\tfrac{x}{\rho}) + A(\tfrac{y}{\rho}) - 2A(r)}{2} \right] \, dr \right)\\
& = \left(A_{\rho}(x) - A_{\rho}(y),x - y,\int_{x}^{y}
\left[\frac{A_{\rho}(x) + A_{\rho}(y) - 2A_{\rho}(r)}{2} \right]
dr \right), \end{align*} where $A_{\rho}(x) := \rho \cdot
A(x/\rho)$ has $\mathrm{Lip}(A_{\rho}) = \mathrm{Lip}(A)$. Also,
\begin{displaymath} D_{2}\left(\tfrac{x}{\rho},\tfrac{y}{\rho} \right)
= \int_{x/\rho}^{y/\rho} \frac{B(\tfrac{x}{\rho}) +
B(\tfrac{y}{\rho}) -  2B(r)}{2(x/\rho - y/\rho)^{2}} \, dr
=\int_{x}^{y} \frac{B_{\rho}(x) + B_{\rho}(y) - 2B_{\rho}(r)}{2(x
- y)^{2}} \, dr =: D_{2}^{\rho}(x,y),
\end{displaymath}
with $B_{\rho}(x) = \rho \cdot B(x/\rho)$. Therefore, we may re-write
\begin{equation}\label{form160} \eqref{form159} = \frac{1}{\rho} \int \left| \int \mathfrak{K}_{\rho}(x,y)\hat{f}\left(\tfrac{y}{\rho},\tau \right) \, dy \right|^{2} \, dx, \end{equation}
where
\begin{displaymath} \mathfrak{K}_{\rho}(x,y) := k_{\rho}\left(A_{\rho}(x) - A_{\rho}(y),x - y,\int_{x}^{y}
\left[\frac{A_{\rho}(x) + A_{\rho}(y) - 2A_{\rho}(r)}{2} \right]
dr \right)D_{2}^{\rho}(x,y). \end{displaymath} Since the Lipschitz
constants of the maps $A_{\rho}$ and $B_{\rho}$ are uniformly bounded in $\rho >
0$, and also the kernel constants of $k_{\rho}$ are
independent of $\rho$ by Lemma \ref{l:calK}, we have arrived at a
situation where Theorem \ref{t:technical} can be easily applied: we just need to express $\mathfrak{K}_{\rho}(x,y)$ in the form
\begin{displaymath} \mathfrak{K}_{\rho}(x,y) = k_{\rho}(\Phi_{\rho}^{-1}(y) \cdot \Phi_{\rho}(x))D_{2}^{\rho}(x,y), \end{displaymath} 
where $\Phi_{\rho}(x) = (0,y,0) \cdot \phi_{\rho}(0,y,0)$ is the graph map of the intrinsic Lipschitz function $\phi_{\rho}(0,y,0) = (A_{\rho}(y),0,\int_{0}^{y} A_{\rho}(r) \, dr)$ over $\mathbb{V}_{y}$. Then, Theorem \ref{t:technical} shows that
\begin{displaymath} \eqref{form160} \lesssim \frac{1}{\rho} \int \left|\hat{f}\left(\tfrac{x}{\rho},\tau \right) \right|^{2} \, dx
= \int |\hat{f}(x,\tau)|^2 \, dx, \end{displaymath} as claimed in
\eqref{form156}. This proves that
$\|K_{\Gamma}D_{B,2}\|_{\mathrm{C.Z.}} < \infty$.

As we already mentioned, proving that $\|K_{\Gamma}D_{B,1}\|_{\mathrm{C.Z.}} \lesssim 1$ is extremely similar. After repeating the calculations above,
one ends up considering the kernels
\begin{displaymath} \mathfrak{K}_{\rho,1}(x,y) := k_{\rho}\left(A_{\rho}(x) - A_{\rho}(y),x - y,\int_{x}^{y}
 \left[\frac{A_{\rho}(x) + A_{\rho}(y) - 2A_{\rho}(t)}{2} \right] dr\right)D_{1}^{\rho}(x,y) \end{displaymath}
for $\rho > 0$, where $D_{1}^{\rho}(x,y) = (B_{\rho}(x) - B_{\rho}(y))/(x - y)$. So, Theorem \ref{t:technical} can be applied as before. This completes the proof of Theorem \ref{t:flagTechnical}.  \end{proof}

%

\appendix

\section{On the corona decomposition for Lipschitz functions}\label{s:coronaComparison}

Recall that we needed the following statement regarding $1$-Lipschitz functions.

\begin{thm}\label{appendixCorona} For every $\eta \in (0,1)$, there exists a constant $C \geq 1$ such that the following holds. Let $\phi \colon \R \to \R$ be $1$-Lipschitz. Then, there exists a decomposition $\calD = \calB \dot{\cup} \calQ$ with the properties \eqref{badCarleson}, \eqref{treeDecomposition}, and \eqref{topCarleson}. For every $\calT \in \calF$ there exists a $2$-Lipschitz linear map $L_{\calT} \colon \R \to \R^{2}$ and an $\eta$-Lipschitz map $\psi_{\calT} \colon \R \to \R^{2}$ such that $\psi_{\calT} + L_{\calT}$ approximates $\phi$ well at the resolution of the intervals in $\calT$:
\begin{equation}\label{form15Appendix} |\phi(s) - (\psi_{\calT} + L_{\calT})(s)| \leq \eta |Q|, \qquad s \in 2Q, \; Q \in \calT. \end{equation}
\end{thm}
This version looks slightly different to the corona decomposition for Lipschitz graphs in David and Semmes' monograph, so we explain here briefly, how to bridge the gap. We start by stating the exact corona decomposition in \cite[Definition 3.13, p. 55]{MR1251061}.
\begin{thm}[Corona decomposition of David-Semmes]\label{davidSemmesCorona} For every $\eta > 0$, there exists a constant $C \geq 1$ such that the following holds. Let $\phi \colon \R \to \R$ be $1$-Lipschitz, and write
\begin{displaymath} \Phi(x) := (x,\phi(x)), \qquad x \in \R. \end{displaymath}
There exists a decomposition $\mathcal{D} = \calB \dot{\cup} \calQ$ with the properties \eqref{badCarleson}, \eqref{treeDecomposition}, and \eqref{topCarleson}. For every $\calT \in \calF$, there exists a possibly rotated $\eta$-Lipschitz graph $\Gamma_{\calT} \subset \R^{2}$ such that
\begin{equation}\label{form108} \dist(\Phi(s),\Gamma_{\calT}) \leq \eta |Q|, \qquad s \in 2Q, \; Q \in \calT. \end{equation}
\end{thm}
To deduce Theorem \ref{appendixCorona} from this statement, all we need to do is establish \eqref{form15Appendix}, that is, find an $\eta$-Lipschitz map $\psi_{\calT} \colon \R \to \R$, and a $2$-Lipschitz linear map $L_{\calT} \colon \R \to \R$, such that \eqref{form15Appendix} holds. We start by applying Theorem \ref{davidSemmesCorona} with a sufficiently small parameter $\eta' > 0$, at least so small that $0 < \eta' < \eta/12$. Then, fix $\calT \in \calF$, and $Q \in \calT$. Let
\begin{displaymath} \Gamma_{\calT} = R_{\theta}(\{(x,\phi_{\calT}(x)) : x \in \R\}) \end{displaymath}
be a rotated $\eta'$-Lipschitz graph appearing in \eqref{form108}, that is,
\begin{displaymath} R_{\theta}(x,y) =(x \cos \theta - y \sin \theta, x \sin \theta + y \cos \theta), \end{displaymath}
and $\phi_{\calT} \colon \R \to \R$ is $\eta'$-Lipschitz. We first observe that, if $\eta' > 0$ is small enough, then $|\mathrm{tan \,} \theta| \leq 2$. Namely, the case $\mathrm{tan\, \theta} = 2$ and $\eta' = 0$ would imply, by \eqref{form108}, that $\phi|_{Q}$ is affine with slope in $\{-2,2\}$, contradicting the $1$-Lipschitz assumption. The case of "small $\eta'$" requires a small additional argument, which we leave to the reader.

Now, we claim that $\Gamma_{\calT}$ can be written as the graph of a function of the form $\psi_{\calT} + L_{\calT}$, where $\psi_{\calT}$ is $\eta$-Lipschitz, and $L_{\calT}(x) = x \tan \theta$. To this end, we note that
\begin{displaymath} \Gamma_{\calT} = \{(z(x),x \sin \theta + \phi_{\calT}(x) \cos \theta) : x \in \R\}, \end{displaymath}
where $z(x) = x \cos \theta - \phi_{\calT}(x) \sin \theta$. Here,
\begin{equation}\label{form109} |z(x) - z(x')| \geq [|\mathrm{cos\,}\theta| - \eta'|\mathrm{sin\,}\theta|]|x - x'| \geq \tfrac{1}{4}|x - x'|, \end{equation}
taking $\eta' > 0$ small enough, since $|\mathrm{cos\,}\theta| \geq 1/\sqrt{5}$. In particular, the change-of-variables $x \mapsto z(x)$ is bijective, and it now suffices to find a $\eta$-Lipschitz $\psi_{\calT} \colon \R \to \R$ such that
\begin{displaymath} x\sin\theta + \phi_{\calT}(x)\cos \theta = \psi_{\calT}(z(x)) + z(x)\tan \theta. \end{displaymath}
Plugging in the definition of $z(x) = x \cos \theta - \phi_{\calT}(x) \sin \theta$, this requirement is equivalent to
\begin{displaymath} \psi_{\calT}(z(x)) = \left[ \cos \theta + \frac{\sin^{2}\theta}{\cos \theta} \right]\phi_{\calT}(x) = \frac{\phi_{\calT}(x)}{\cos \theta}. \end{displaymath}
Finally, $\psi_{\calT}$ is indeed $\eta$-Lipschitz:
\begin{displaymath} |\psi_{\calT}(z(x)) - \psi_{\calT}(z(x'))| = \frac{1}{\cos \theta} |\phi_{\calT}(x) - \phi_{\calT}(x')| \leq \frac{\eta'}{\cos \theta}|x - x'| \leq \eta|z(x) - z(x')|, \end{displaymath}
using \eqref{form109} in the last estimate, and recalling that $\cos \theta \geq 1/\sqrt{5} \geq 1/3$, and $\eta' < \eta/12$.

Now we have re-parametrised $\Gamma_{\calT}$ as the graph of the function $\psi_{\calT} + L_{\calT}$, as desired, but we still need to check that \eqref{form15Appendix} holds. This follows easily from \eqref{form108}: if $s \in 2Q$, then \eqref{form108} gives us a point $s' \in \R$ with
\begin{displaymath} \max\{|s - s'|,|\phi(s) - (\psi_{\calT} + L_{\calT})(s')|\} \leq \eta' |Q|.\end{displaymath}
Consequently, using that $\psi_{\calT} + L_{\calT}$ is $3$-Lipschitz, and $\eta' < \eta/4$,
\begin{align*} |\phi(s) - [\psi_{\calT} + L_{\calT}](s)| & \leq |\phi(s) - [\psi_{\calT} + L_{\calT}](s')| + |[\psi_{\calT} + L_{\calT}](s) - [\psi_{\calT} + L_{\calT}](s')| \leq \eta|Q|. \end{align*}

\section{A Littlewood-Paley proposition}\label{a:christProp}

\begin{proposition}\label{p:christProp} Let $\{F_{s}\}_{s \in (0,\infty)}$ be a family of $C^{1}$-functions $F_{s} \colon \R^{n} \to \R$ satisfying
\begin{displaymath} \|F_{s}\|_{L^{\infty}} + \|\nabla F_{s}\|_{L^{\infty}} \leq C_{F}, \qquad s \in (0,\infty), \end{displaymath}
where $C_{F} \geq 1$ is a constant independent of $s \in
(0,\infty)$. Assume also that $(s,x) \mapsto F_{s}(x)$ is Borel.
Let $\varphi \in C^{\infty}_{c}(\R^{n})$ satisfy $\int \varphi =
1$, and write $\varphi_{s}(x) := \tfrac{1}{s^{n}}\varphi(x/s)$.
Further, let $\{\psi_{s}\}_{s > 0} \subset C^{1}(\R^{n} \,
\setminus \{0\})$ be a family of functions which satisfy the
following requirements for some $C_{\psi} > 0$ and $\alpha \in
(0,1]$:
\begin{enumerate}
\item $\spt \psi_{s} \subset B(0,C_{\psi}s)$, \item
$\|\psi_{s}(x)\|_{L^{\infty}(\R^{n})} \leq C_{\psi}/s^{n}$ and
$\|\nabla \psi_{s}\|_{L^{\infty}(\R^{n} \, \setminus \, \{0\})}
\leq C_{\psi}/s^{n + 1}$, and \item $|\hat{\psi}_{s}(\xi)| \leq
C_{\psi}\min\{|s\xi|^{\alpha},|s\xi|^{-\alpha}\}$ for $\xi \in
\R^{n}$.
\end{enumerate}
For $f \in L^{1}_{\mathrm{loc}}(\R^{n})$, define $P_{s}(f) := f
\ast \varphi_{s}$ and $Q_{s}(f) := f \ast \psi_{s}$. Finally, let
$a_{1},\ldots,a_{m} \in L^{\infty}(\R^{n})$, $m \geq 1$, and define the operator
\begin{displaymath} (Tf)(x) := \int_{0}^{\infty} F_{s}[P_{s}(a_{1})(x)\cdots P_{s}(a_{m})(x)] \cdot Q_{s}(f)(x) \, \frac{ds}{s}, \qquad f \in C^{\infty}_{c}(\R^{n}). \end{displaymath}
Then $T$ extends to a bounded operator on $L^{2}$ with
$\|T\|_{L^{2} \to L^{2}} \leq
C(\max_{j} \|a_{j}\|_{L^{\infty}},m,C_{F},\varphi,C_{\psi})$. \end{proposition}

The proposition is a variant of \cite[Proposition 9, p.
57]{MR1104656}, but Christ only gives a proof in the special case
where $Q_{s} = Q^{1}_{s} \circ Q^{2}_{s}$, where
$Q^{1}_{s},Q^{2}_{s}$ are operators of the same type as $Q_{s}$,
and moreover $m=1$, and $\psi_{s}(x) = \tfrac{1}{s^n}\psi(x/s)$
for a fixed zero-mean $\psi \in C^{\infty}(\R^n)$; in our main
application, $m\in \{1,2\}$ and $\psi_{s}$ depends on $s$ in a
more complicated way, and potentially has a jump discontinuity at
$0$. For these reasons, we give all the details. Hofmann also has
a variant of the proposition in \cite[Lemma 2]{MR1484857} in the
parabolic setting, but the technical setup is, once again, a
little different. However, our proof of Proposition
\ref{p:christProp} closely follows his. We start by constructing a
special function; the existence of a function with approximately
these properties is also used in the proof of \cite[Lemma
2]{MR1484857}, but since no proof is given in \emph{op. cit.}, we
include the details here.

\begin{lemma}\label{l:wp} Let $\epsilon \in (0,1)$, and define a distribution $\wp := \wp_{\epsilon}$ on $\R^{n}$ whose Fourier transform lies in $C^{\infty}(\R^{n} \, \setminus \{0\})$ and satisfies
\begin{displaymath} \hat{\wp}(\xi) = \begin{cases} |\xi|^{\epsilon}, & \text{for } |\xi| \leq 1, \\ |\xi|^{-\epsilon}, & \text{for } |\xi| > 2. \end{cases} \end{displaymath}
Then $\wp \in L^{1}(\R^{n})$ with $\int \wp = 0$, and in fact
\begin{equation}\label{eq:wp} |\wp(x)| \lesssim_{\epsilon} \min\{|x|^{\epsilon - n}, |x|^{-\epsilon - n}\}, \qquad x \neq 0. \end{equation}
\end{lemma}
\begin{proof} The conclusion that $\int \wp = 0$ follows immediately from $\hat{\wp}(0) = 0$ once we have managed to prove that $\wp \in L^{1}(\R^{n})$.
To this end, we cover $\R^n$ by three overlapping open sets as
follows:
\begin{displaymath} \R^n =  \{|\xi| < 1\} \cup \{\tfrac{1}{2} < |\xi| < 3\} \cup \{|\xi| > 2\}, \end{displaymath}
and then choose a smooth partition of unity
$\{\widecheck{\varphi}_{1},\widecheck{\varphi}_{2},\widecheck{\varphi}_{3}\}$
subordinate to this cover. Here
$\varphi_{1},\varphi_{2},\varphi_{3} \in \mathcal{S}(\R^{n})$. We
write
\begin{displaymath} \hat{\wp} = \sum_{j} \hat{\wp}\widecheck{\varphi}_{j} = |\xi|^{\epsilon}\widecheck{\varphi}_{1} + \hat{\wp}\widecheck{\varphi}_{2} + |\xi|^{-\epsilon}\widecheck{\varphi}_{3} =: \check{\psi}_{1} + \check{\psi}_{2} + \check{\psi}_{3},  \end{displaymath}
where $\psi_{1},\psi_{2},\psi_{3}$ are \emph{a priori} just
tempered distributions. The aim is to show that
$\psi_{1},\psi_{2},\psi_{3}$ satisfy, individually, the assertions
we made of $\wp$. More precisely, we will establish the bounds \eqref{eq:wp} for each $\psi_{j}$, which will show that $\psi_{j} \in L^{1}(\R^{n})$, and the zero-mean condition then follows automatically from $\hat{\psi}_{j}(0) = 0$, $j \in \{1,2,3\}$.

The bounds \eqref{eq:wp} are clear for $\psi_{2}$, which is the
Fourier transform $\hat{\wp}\check{\varphi}_{2} \in
C^{\infty}_{c}(\R^{n})$. We then consider $\psi_{1} =
\widehat{|\xi|^{\epsilon}} \ast \varphi_{1}$. Since $\varphi_{1}
\in \mathcal{S}(\R^{n})$, we first note that $\psi_{1} \in
C^{\infty}(\R^{n})$ by \cite[Theorem 7.19(a)]{MR1157815}. To
establish the decay $|\psi_{1}(x)| \lesssim |x|^{-n - \epsilon}$
for $|x| \geq 1$, we begin by recalling, from \cite[Theorem
2.4.6]{MR3243734}, that the Fourier transform of $\xi \mapsto
|\xi|^{\epsilon}$ is the \emph{homogeneous distribution}
$\mathfrak{h} := \mathfrak{h}_{-n - \epsilon}$ with index $-n -
\epsilon$. This distribution is defined, for $\varphi \in
\mathcal{S}(\R^n)$, by
\begin{displaymath} \mathfrak{h}(\varphi) = c_{1} \int_{|z| \geq 1} \varphi(z)|z|^{-n - \epsilon} \, dz + c_{2} \varphi(0) + c_{3} \int_{|z| < 1} [\varphi(z) - \varphi(0)]|z|^{-n - \epsilon} \, dz, \end{displaymath}
where $c_{1},c_{2},c_{3} \in \C$ are constants (depending on
$\epsilon,n$). Therefore, $\psi_{1}(x) = (\mathfrak{h} \ast
\varphi_{1})(x) = \mathfrak{h}(\tilde{\varphi}_{1,x})$, with
$\tilde{\varphi}_{1,x}(z) := \varphi_{1}(x - z)$, and by the
definition of $\mathfrak{h}$,
\begin{displaymath} |\psi_{1}(x)| \lesssim_{\epsilon} \int_{|z| \geq 1} |\varphi_{1}(x - z)||z|^{-n - \epsilon} \, dz + |\varphi_{1}(x)| + \int_{|z| < 1} |\varphi_{1}(x - z) - \varphi_{1}(x)||z|^{-n - \epsilon} \, dz.  \end{displaymath}
Since $\varphi_{1} \in \mathcal{S}(\R^n)$, the two latter terms
satisfy $\lesssim_{\epsilon,N} (1 + |x|)^{-N}$ for any $N \geq 1$.
The first term satisfies $\lesssim_{\epsilon} |x|^{-n -
\epsilon}$, using that $\varphi_{1} \in \mathcal{S}(\R^{n})$,
$|z|^{-n - \epsilon} \in L^{1}(\{|z| \geq 1\})$, and performing a
little case chase with annular decompositions around $x$.

We finally come to the piece $\psi_{3}$, and we start by writing
\begin{displaymath} \check{\psi}_{3}(\xi) = |\xi|^{-\epsilon} - (1 - \widecheck{\varphi}_{3}(\xi))|\xi|^{-\epsilon} =: I_{1}(\xi) + I_{2}(\xi). \end{displaymath}
Here $\widehat{I}_{1}(x) = c_{\epsilon,n}|x|^{\epsilon - n} \in
L^{1}_{\mathrm{loc}}(\R^{n})$. On the other hand, $\spt (1 -
\widecheck{\varphi}_{3}) \subset B(0,3)$, so $\widehat{I}_{2}$ is
the convolution of $c_{\epsilon,n}|x|^{\epsilon - n}$ with a
Schwartz function, and hence $\widehat{I}_{2}\in C^{\infty}(\R^{n})$.
So, we conclude that $\psi_{3} \in L^{1}_{\mathrm{loc}}(\R^{n})$,
and $|\psi_{3}(x)| \lesssim_{\epsilon} |x|^{\epsilon - n}$ for $|x| \leq 1$. 
To complete the proof of the lemma, we claim that $|\psi_{3}(x)|
\lesssim_{\epsilon,N} |x|^{-N}$ for $|x| \geq 1$, and for any $N \geq 0$. Indeed,
note that if $N \geq n/2$, then
\begin{displaymath} \bigtriangleup^{N}\check{\psi}_{3} = \bigtriangleup^{N}[\xi \mapsto \widecheck{\varphi}_{3}(\xi)|\xi|^{-\epsilon}] \in L^{1}(\R^{n}) \end{displaymath}
using the Leibniz rule, and noting that $\widecheck{\varphi}_{3}$
is supported away from the origin. Consequently $x \mapsto
|x|^{2N}\psi_{3}(x) \in L^{\infty}(\R^{n})$ for all $N \geq n/2$. This implies $|\psi_{3}(x)| \lesssim_{\epsilon,N} |x|^{-N}$ for $|x| \geq 1$ (for any $N \geq 0$). \end{proof}

Now we are equipped to prove Proposition \ref{p:christProp}.

\begin{proof}[Proof of Proposition \ref{p:christProp}] In this proof,
 the constants in the "$\lesssim$" notation may depend on the data $\max_{j} \|a_{j}\|_{L^{\infty}(\R^n)}$, $m$,
 $\alpha \in (0,1]$, $C:=C_{\psi}$, $\varphi$, and $C_{F}$. For $s \in (0,\infty)$ fixed, write
\begin{displaymath} K_{s}(x,y) := F_{s}[P_{s}(A)(x)]\psi_{s}(x - y), \end{displaymath}
where $P_{s}(A)(x) := P_{s}(a_{1})(x)\cdots P_{s}(a_{m})(x)$ and then
\begin{displaymath} K(x,y) := \int_{0}^{\infty} K_{s}(x,y) \, \frac{ds}{s}. \end{displaymath}
We begin by verifying that $K$ is an $n$-SK with $\|K\|_{n,strong}
\lesssim 1$.
 Fix $x \neq y$, and note that $K_{s}(x,y) \neq 0$ only if $s \geq |x - y|/C$.
 Since $\|F_{s}(P_{s}(A))\|_{L^{\infty}(\R^n)} \lesssim 1$, it follows from $\|\psi_{s}\|_{L^{\infty}} \lesssim s^{-n}$ that
\begin{displaymath} |K(x,y)| \lesssim \int_{|x - y|/C} \frac{ds}{s^{n + 1}} \lesssim \frac{1}{|x - y|^{n}}. \end{displaymath}
Second, fix $x,x',y \in \R^{n}$ with $|x - x'| \leq |x - y|/2$.
Then,
\begin{align*} |K(x',y) - K(x,y)| & \leq \int_{0}^{\infty} |F_{s}(P_{s}(A))(x') - F_{s}(P_{s}(A))(x)||\psi_{s}(x' - y)| \, \frac{ds}{s}\\
& \quad + \int_{0}^{\infty} |F_{s}(P_{s}(A))(x)||\psi_{s}(x' - y)
- \psi_{s}(x - y)| \, \frac{ds}{s}. \end{align*} To estimate the
first integral, note that $\psi_{s}(x' - y) = 0$ if $s < |x -
y|/(2C)$, so using also $\|F_{s}'\|_{L^{\infty}(\R)} \lesssim 1$,
$\|\psi_{s}\|_{L^{\infty}(\R)} \lesssim s^{-n}$, and $|P_{s}(A)(x) - P_{s}(A)(x')| \lesssim |x - x'|/s$, we find
\begin{align*} \int_{0}^{\infty} |F_{s}[P_{s}(A)(x')] - F_{s}[P_{s}(A)(x)]||\psi_{s}(x' - y)| \, \frac{ds}{s} & \lesssim \int_{\tfrac{|x - y|}{2C}} |P_{s}(A)(x') - P_{s}(A)(x)| \, \frac{ds}{s^{n + 1}}\\
& \lesssim |x - x'|\int_{\tfrac{|x - y|}{2C}} \frac{ds}{s^{n + 1}}
\sim \frac{|x - x'|}{|x - y|^{n + 1}}. \end{align*} To estimate
the second integral, we use $|F_{s}(P_{s}(A))(x)| \lesssim 1$, and
that the line segment connecting $x - y$ and $x' - y$ lies in
$\R^{n} \, \setminus \, \{0\}$,\footnote{This argument is only
relevant for $n = 1$.} so $|\psi_{s}(x - y) - \psi_{s}(x' - y)|
\lesssim |x - x'|\|\nabla \psi_{s}\|_{L^{\infty}(\R^{n} \,
\setminus \, \{0\})} \lesssim |x - x'|/s^{n + 1}$. Since further
\begin{displaymath} \psi_{s}(x' - y) = 0 = \psi_{s}(x - y), \qquad s \leq |x - y|/(2C), \end{displaymath}
it follows that
\begin{displaymath} |K(x',y) - K(x,y)| \lesssim \frac{|x - x'|}{|x - y|^{n + 1}}, \end{displaymath}
A similar estimate for $|K(y,x') - K(y,x)|$ is even easier to
obtain, as there is no need to introduce cross terms.

Since $K$ is an $n$-SK, to check that $\|T\|_{L^{2} \to L^{2}}
\lesssim 1$, it suffices to verify the conditions of the $T1$
theorem, and more precisely that
\begin{equation}\label{cProp1} \fint_{B_{0}} |T(b)| \lesssim 1 \quad \text{and} \quad  \fint_{B_{0}} |T^{t}(b)| \lesssim 1 \end{equation}
whenever $B_{0} = B(x_{0},r_{0})$ is a ball, and $b \in
C^{\infty}(\R)$ satisfies $\mathbf{1}_{2B_{0}} \leq b \leq
\mathbf{1}_{3B_{0}}$, recall \eqref{universalT1}. We will ignore the standard issues of $\epsilon$-truncation in this argument. The first
estimate in \eqref{cProp1} easily follows from the fact that $y \mapsto K_{s}(x,y)$ has zero mean (note that $\hat{\psi}_s(0) = 0$ by assumption (3)) and is supported in $B(x,Cs)$ for every $s > 0$. With this in hand, one starts by fixing $x \in B(x_{0},r_{0})$ and writing
\begin{equation}\label{form149} |T(b)(x)| \lesssim \int_{0}^{r_{0}/C} \left| \int_{B(x,Cs)} K_{s}(x,y)b(y) \, dy \, \right| \frac{ds}{s} + \int_{r_{0}/C}^{\infty} \|\psi_{s}\|_{2}\|b\|_{2} \, \frac{ds}{s}. \end{equation}
Regarding the first term, note that $B(x,Cs) \subset 2B_{0}$ for
$x \in B_{0}$ and $0 < s < r_{0}/C$, so $b \equiv 1$ on the
support of $y \mapsto K_{s}(x,y)$. Hence the first term vanishes
by the zero-mean property of $y \mapsto K_{s}(x,y)$. To treat the
second term, note that $\|b\|_{2} \sim r_{0}^{n/2}$, and
\begin{equation}\label{form145} \|\psi_{s}\|_{2} \lesssim \left(\int_{B(0,Cs)} \frac{1}{s^{2n}} \, dy \right)^{1/2} \lesssim s^{-n/2}. \end{equation}
This implies that $\|T(b)\|_{L^{\infty}(B_{0})} \lesssim 1$ and
yields the first part of \eqref{cProp1}.

We then consider the second estimate in \eqref{cProp1}. One may easily reduce to the case $x_{0} = 0$ and $r_{0} = 1$: indeed, one simply performs a change-of-variables to write
\begin{displaymath} \fint_{B_{0}} |T^{t}(b)| = \fint_{B(0,1)} |\widetilde{T}^{t}(\tilde{b})|, \end{displaymath}
where $\tilde{b}(x) := b(r_{0}x + x_{0})$ satisfies $\mathbf{1}_{B(0,2)} \leq \tilde{b} \leq \mathbf{1}_{B(0,3)}$, and $\widetilde{T}^{t}$ is of the same form as $T^{t}$ (see \eqref{cProp2} below). It is critical, but easy to check, that the family of functions $\{x \mapsto r_{0}^{n}\psi_{r_{0}s}(r_{0}x)\}_{s > 0}$ satisfies the same conditions (1)-(3) as $\{\psi_{s}\}_{s > 0}$, with the same constants.

The kernel of $T^{t}$ is $(x,y) \mapsto K(y,x)$, so, for $x \in
B_{0} := B(0,1)$,
\begin{align} T^{t}(b)(x) = \int K(y,x)b(y) \, dy & =  \int_{0}^{\infty}\int  K_{s}(y,x)b(y) \, \frac{ds}{s} \, dy \notag\\
& = \int_{0}^{\infty} \int F_{s}[P_{s}(A)(y)]\psi_{s}(y - x)b(y) \, \frac{ds}{s} \, dy \notag\\
&\label{cProp2} =: \int_{0}^{\infty} Q_{s}(F_{s}[P_{s}(A)] \cdot
b)(x) \, \frac{ds}{s},  \end{align} where $Q_{s}$ refers to
convolution with $z \mapsto \psi_{s}(-z)$. We note that
\begin{displaymath} \left| \int_{1}^{\infty} Q_{s}(F_{s}[P_{s}(A)] \cdot b)(x) \, \frac{ds}{s} \right| \lesssim 1 \end{displaymath}
by the argument we used for the second term in \eqref{form149}.
Therefore, the second part of \eqref{cProp1} follows once we
manage to show that
\begin{equation}\label{cProp4a} \int \left[ \int_{0}^{1} Q_{s}(F_{s}[P_{s}(A)] \cdot b)(x) \, \frac{ds}{s} \right] \, g(x) \, dx \lesssim 1 \end{equation}
for any $g \in L^{\infty}(\R)$ with $\spt g \subset B_{0}$ and
$\|g\|_{L^{\infty}} = 1$.
 We note in passing that the value of \eqref{cProp4a} remains unchanged if we now replace the function $a_{1},\ldots,a_{m}$ by their restrictions to a ball
 $B(0,C_0)$, where $C_0 = C_0(\spt \varphi) \geq 1$ is a constant depending only on $\spt \varphi$. Hence, we may assume in the sequel that
\begin{equation}\label{cProp12} \max_{j} \|a_{j}\|_{L^{2}(\R^{n})} \lesssim 1. \end{equation}
Next, using Fubini and Plancherel, and setting $B_{s} :=
F_{s}[P_{s}(A)] \cdot b$, we re-write
\begin{equation}\label{form150} \eqref{cProp4a} = \int_{0}^{1} \int \hat{\psi}_{s}(\xi)\widehat{B}_{s}(\xi) \cdot \overline{\hat{g}}(\xi) \, d\xi \, \frac{ds}{s}. \end{equation}
We then factorise $\hat{\psi}_{s}(\xi) =  \hat{\wp}(s \xi) \cdot
\hat{q}_{s}(\xi)$, where $\wp$ is the special function appearing
in Lemma \ref{l:wp} with parameter $\epsilon := \alpha/2$, and
$\hat{q}_{s}(\xi) := \hat{\psi}_{s}(\xi)/\hat{\wp}(s\xi)$. Note
that the function $\widehat{\wp}$ in Lemma \ref{l:wp} may be
chosen so that $\widehat{\wp} > 0$, and then $\hat{\wp}(s\xi) \sim
\min\{|s\xi|^{\alpha/2},|s\xi|^{-\alpha/2}\}$ for all $s > 0$ and
$\xi \in \R$. It follows from our assumption
$|\hat{\psi}_{s}(\xi)| \leq
C\min\{|s\xi|^{\alpha},|s\xi|^{-\alpha}\}$ that
\begin{equation}\label{cProp11} |\hat{q}_{s}(\xi)| \lesssim \min\{|s\xi|^{\alpha/2},|s\xi|^{-\alpha/2}\}, \qquad \xi \in \R^{n}, \, s > 0. \end{equation}
We also recall from Lemma \ref{l:wp} that $\wp \in L^{1}(\R)$ with
$\int \wp = 0$. Then, continuing from \eqref{form150}, and using
Cauchy-Schwarz and Plancherel, we find
\begin{equation}\label{cProp10} \eqref{form150} \leq \left( \int_{0}^{1} \int |(B_{s} \ast \wp_{s})(x)|^{2} \, dx \, \frac{ds}{s} \right)^{1/2} \left(\int_{0}^{1} \int |\hat{q}_{s}(\xi)\hat{g}(\xi)|^{2} \, d\xi \, \frac{ds}{s} \right)^{1/2}, \end{equation}
where $\wp_{s} = \tfrac{1}{s^{n}}\wp(\frac{\cdot}{s}) \in
L^{1}(\R^{n})$. The second factor is easily treated with
\eqref{cProp11}:
\begin{displaymath} \int_{0}^{1} \int |\hat{q}_{s}(\xi)\overline{\hat{g}}(\xi)|^{2} \, d\xi \, \frac{ds}{s} \lesssim \int \left[ \int_{0}^{1} \min\{|s\xi|^{\alpha},|s\xi|^{-\alpha}\} \, \frac{ds}{s} \right] |\hat{g}(\xi)|^{2} \, d\xi \lesssim \|g\|_{2}^{2} \lesssim 1. \end{displaymath}
We then turn to the first factor in \eqref{cProp10}. It may be
worth pointing out what we have gained compared to
\eqref{cProp4a}: at the expense of trading "$\psi_{s}$" to the
slightly (not essentially) worse function "$\wp_{s}$", we have
managed to replace the $L^{1}$-norm by an $L^{2}$-norm. Note that
the most simple-minded application of Cauchy-Schwarz in
\eqref{cProp4a} would not have given the same result, because
$\int \int_{0}^{1} g(x)^{2} \, dx \, ds/s = \infty$. We fix $x \in
\R^{n}$, and estimate
\begin{align}\label{cProp3} |(B_{s} \ast \wp_{s})(x)| & \leq \left| \int \wp_{s}(x - z)(F_{s}[P_{s}(A)(z)] - F_{s}[P_{s}(A)(x)])b(z) \, dz \right|\\
&\label{cProp4} \qquad + |F_{s}[P_{s}(A)(x)]|\left| \int \wp_{s}(x
- z)b(z)\, dz \right|.   \end{align} If one plugs the term
\eqref{cProp4} back into the first factor in \eqref{cProp10} and
uses $|F_{s}[P_{s}(A)(x)]| \lesssim 1$, Plancherel, and
$|\hat{\wp}_{s}(\xi)| \sim
\min\{|s\xi|^{\alpha/2},|s\xi|^{-\alpha/2}\}$, the result is
bounded by a constant times
\begin{displaymath} \left( \int_{0}^{1} \int |(\wp_{s} \ast b)(x)|^{2} \, dx \, \frac{ds}{s} \right)^{1/2} \sim \left( \int \left[\int_{0}^{1} \min\{|s\xi|^{\alpha},|s\xi|^{-\alpha}\} \, \frac{ds}{s} \right] |\hat{b}(\xi)|^{2} \, d\xi \right)^{1/2} \sim \|b\|_{2} \sim 1. \end{displaymath}
It remains to consider the contribution from \eqref{cProp3}. First, place absolute values inside, and recall that $\spt b \subset B(0,3)$ to obtain
\begin{displaymath} \eqref{cProp3} \lesssim \int_{B(0,3)} |\wp_{s}(x - z)||P_{s}(a_{1})(z)\cdots P_{s}(a_{m}(z)) - P_{s}(a_{1})(x)\cdots P_{s}(a_{m})(x)| \, dz. \end{displaymath}
Then, introducing cross terms, and using that
$\|P_{s}(a_{j})\|_{L^{\infty}} \leq \|a_{j}\|_{L^{\infty}} <
\infty$, the right hand side is bounded by the sum of the terms
\begin{displaymath} \int_{B(0,3)} |\wp_{s}(x - z)||P_{s}(a_{j})(z) - P_{s}(a_{j})(x)| \, dz, \qquad 1 \leq j \leq m. \end{displaymath}
Each one of these will be plugged into \eqref{cProp10}
individually. As a result, after applying Cauchy-Schwarz in the
$z$-variable, and Plancherel, the contribution of \eqref{cProp3}
to (the first factor in) \eqref{cProp10} is bounded by the maximum
(over $1 \leq j \leq m$) of the quantities
\begin{align*} \lesssim & \left( \int \int_{0}^{\infty} \int |\wp_{s}(x - z)||P_{s}(a_{j})(z) - P_{s}(a_{j})(x)|^{2} \, dz  \, \dfrac{ds}{s} \, dx \right)^{1/2}\\
& \quad = \left(\int_{0}^{\infty} \iint |\wp_{s}(u)||P_{s}(a_{j})(z) - P_{s}(a_{j})(z + u)|^{2} \, du \, dz \, \frac{ds}{s} \right)^{1/2}\\
& \quad = \left(\int_{0}^{\infty} \int |\wp_{s}(u)| \int |\widehat{\varphi}(s\xi)|^{2}|e^{2\pi i u \xi} - 1|^{2}|\hat{a}_{j}(\xi)|^{2} \, d\xi \, du \, \frac{ds}{s} \right)^{1/2}\\
& \quad \lesssim \left( \int_{0}^{\infty} \left[ \int \tfrac{1}{s^n}|\wp(u/s)| \left|\frac{u}{s}\right|^{\alpha/4} \, du \right] \cdot \left[ \int |\widehat{\varphi}(s\xi)|^{2}|\hat{a}_{j}(\xi)|^{2} |s\xi|^{\alpha/4} \, d\xi \right]  \, \frac{ds}{s} \right)^{1/2}\\
& \quad \stackrel{u \mapsto vs}{=} \left( \int |\wp(v)||v|^{\alpha/4} \, dv \right)^{1/2} \left( \int |\hat{a}_{j}(\xi)|^{2} \int_{0}^{\infty} |\widehat{\varphi}(s\xi)|^{2}|s\xi|^{\alpha/4} \, \frac{ds}{s} \, d\xi \right)^{1/2}\\
& \quad = \left( \int |\wp(v)||v|^{\alpha/4} \, dv
\right)^{1/2}\left( \int_{0}^{\infty}
|\widehat{\varphi}(t)|^{2}|t|^{\alpha/4} \, \frac{dt}{t}
\right)^{1/2} \|a_{j}\|_{L^{2}(\R)} \lesssim 1. \end{align*} In the
last estimate, we used \eqref{cProp12}, and that $|\wp(v)|
\lesssim_{\alpha} \min\{|v|^{\alpha/2 - n},|v|^{-\alpha/2 - n}\}$
by \eqref{eq:wp}. This shows that the first factor in
\eqref{cProp10} is $\lesssim 1$, and completes the proof.
\end{proof}

\bibliographystyle{plain}
\bibliography{references}

\end{document}